\documentclass[11pt,reqno]{amsart}

\usepackage{amsthm, mathrsfs, amsmath, amstext, esint, amsxtra, amsfonts, slashed, dsfont, amssymb, dsfont}
\usepackage{relsize}
\usepackage{enumitem}
\usepackage[dvipsnames]{xcolor}
\usepackage[colorlinks, linkcolor=red, citecolor=blue, urlcolor=blue, pagebackref, hypertexnames=false]{hyperref}
\usepackage{extarrows}
\usepackage{comment} 

\usepackage{lmodern}
\usepackage[T1]{fontenc}

\usepackage[belowskip=2pt,aboveskip=0pt]{caption}
\usepackage{booktabs}
\usepackage{geometry}
\newcommand{\one}{\mathds{1}}
\newcommand{\N}{\mathbb N} 
\newcommand{\R}{\mathbb R} 
\newcommand{\C}{\mathbb C} 
\newcommand{\T}{\mathbb T} 
\newcommand{\E}{\mathbb E} 

\newcommand{\Hc}{\mathcal H}

\newcommand{\Mcal}{\mathcal M}

\newcommand{\Sb}{\mathbb S} 
\newcommand{\Sc}{\mathfrak S}

\newcommand{\Wc}{\mathcal W}

\newcommand{\Sh}{\mathscr{S}} 
\newcommand{\vareps}{\varepsilon}

\newcommand{\im}{\mathrm{i \,}}

\def\<{\left\langle}
\def\>{\right\rangle}

\DeclareMathOperator*{\opt}{opt}

\DeclareMathOperator*{\supp}{supp}

\DeclareMathOperator*{\ima}{Im}
\DeclareMathOperator*{\rea}{Re}
\newcommand{\scal}[1]{\left\langle #1 \right\rangle}

\DeclareMathOperator*{\Tr}{Tr}

\DeclareMathOperator*{\lo}{lo}
\DeclareMathOperator*{\hi}{hi}
\DeclareMathOperator*{\id}{Id}

\numberwithin{equation}{section}
\newtheorem{theorem}{Theorem}[section]
\newtheorem{lemma}[theorem]{Lemma}
\newtheorem{proposition}[theorem]{Proposition}

\newtheorem{definition}[theorem]{Definition}
\newtheorem{assumption}[theorem]{Assumption}

\newtheoremstyle{remarkstyle}
{}{}{
}{ }{\bfseries}{.}{ }{\thmname{#1}\thmnumber{ #2}\thmnote{ (#3)}}
\theoremstyle{remarkstyle}
\newtheorem{remark}{Remark}[section]

\makeatletter
\def\subsubsection{\@startsection{subsubsection}{3}%
	\z@{.5\linespacing\@plus.7\linespacing}{-.5em}%
	{\normalfont\bfseries}}
\makeatother

\makeatletter
\def\@tocline#1#2#3#4#5#6#7{\relax
  \ifnum #1>\c@tocdepth 
  \else
    \par \addpenalty\@secpenalty\addvspace{#2}%
    \begingroup \hyphenpenalty\@M
    \@ifempty{#4}{%
      \@tempdima\csname r@tocindent\number#1\endcsname\relax
    }{%
      \@tempdima#4\relax
    }%
    \parindent\z@ \leftskip#3\relax \advance\leftskip\@tempdima\relax
    \rightskip\@pnumwidth plus4em \parfillskip-\@pnumwidth
    #5\leavevmode\hskip-\@tempdima
      \ifcase #1
       \or\or \hskip 1em \or \hskip 2em \else \hskip 3em \fi%
      #6\nobreak\relax
      \dotfill
      \hbox to\@pnumwidth{\@tocpagenum{#7}}
    \par
    \nobreak
    \endgroup
  \fi}
\makeatother

\date{January, 2023}

\title[Invariant Gibbs measure]{Invariant Gibbs measures for 1D NLS in a trap} 
\author[V. D. Dinh and N. Rougerie]{Van Duong Dinh and Nicolas Rougerie}
\address[V. D. Dinh]{Ecole Normale Supérieure de Lyon \& CNRS, UMPA (UMR 5669), Lyon, France}
\email{contact@duongdinh.com}
\address[N. Rougerie]{Ecole Normale Supérieure de Lyon \& CNRS, UMPA (UMR 5669), Lyon, France}
\email{nicolas.rougerie@ens-lyon.fr}


\begin{document}

\begin{abstract}
	We consider the one dimensional cubic nonlinear Schr\"odinger equation with trapping potential behaving like $|x|^s$ ($s>1$) at infinity. We construct Gibbs measures associated to the equation and prove that the Cauchy problem is globally well-posed almost surely on their support. Consequently, the Gibbs measure is indeed invariant under the flow of the equation. We also address the construction and invariance of canonical Gibbs measures  (conditioned on the $L^2$ mass) and make remarks regarding higher non-linearities than cubic. 
\end{abstract}

\maketitle

\tableofcontents

\section{Introduction}
\label{sec:intro}
\setcounter{equation}{0}

We study the statistical mechanics of the trapped 1D cubic nonlinear Schr\"odinger equation
\begin{equation}\label{eq:intro NLS}
\im  \partial_t u = h u \pm |u|^2 u  
\end{equation}
on $\R$ with the Schr\"odinger operator 
\[ 
h := - \partial_x ^2 + V(x),
\]
where $V(x) \underset{|x| \to \infty}{\rightarrow} +\infty$ is a trapping potential. Namely, we construct the associated Gibbs probability measures given \emph{formally} by 
\begin{equation}\label{eq:intro gc}
d\mu^{\rm{gc}}_\nu (u) =  z_\nu^{-1} \exp\left( -  \left\langle u, (h + \nu) u \right\rangle_{L^2} \mp \frac{1}{2}\int_{\R} |u|^4  \right) du 
\end{equation}
and 
\begin{equation}\label{eq:intro c}
d\mu^{\rm c}_m (u) =   z_m ^{-1} \one_{\left\{ \int_{\R} |u |^2 = m \right\}} \mu^{\rm gc}_0 (du)
\end{equation}
and prove their invariance under the flow of~\eqref{eq:intro NLS}. These measures correspond, for this nonlinear model, to the well known grand-canonical ensemble (\eqref{eq:intro gc} with chemical potential $\nu \in \R$) and canonical ensemble (\eqref{eq:intro c} with particle number/mass $m>0$) of statistical mechanics. Our main physical motivation comes from the mean-field approximation of Bose gases and Bose--Einstein condensates~\cite{LieSeiSolYng-05,Rougerie-EMS,Schlein-08}, whence our restriction to the cubic equation, corresponding to short-range pair interactions between particles. In this context, the measure~\eqref{eq:intro gc} was rigorously derived from many-body quantum mechanics in~\cite{LewNamRou-14d,LewNamRou-17,FroKnoSchSoh-16,FroKnoSchSoh-17,RouSoh-22}.

The question of the invariance of such measures under the associated Hamiltonian flow has a rich history, elements of which we recall below. As is well-known, the main issue is that, once rigorously defined (in particular, in the focusing case, a $L^2$-mass cutoff is necessary in~\eqref{eq:intro gc}), the above measures live on functional spaces of regularity/integrability levels at which it is challenging or impossible to construct the flow of~\eqref{eq:intro NLS}. Instead of relying on such a deterministic flow, one often constructs a probabilistic Cauchy theory, taking advantage of the formal invariance of the Gibbs measures to substitute for conservation laws. The main datum of the problem, setting the regularity/integrability level, is in our context the growth at infinity of the potential $V$. We assume that it is polynomial, of order $s>1$, which is the threshold for the measure~\eqref{eq:intro gc} to make sense (see below). 

\begin{assumption} \label{assu-V}
Let $V \in C^\infty(\R, \R_+)$ satisfy for some $s>1$:
	
\medskip
	
\noindent (i) There exists $C\geq 1$ so that for all $|x|\geq 1$, $\frac{1}{C} \<x\>^s\leq V(x) \leq C\<x\>^s$.

\medskip

\noindent (ii) For any $j \in \N$, there exists $C_j>0$ so that $|\partial^jV(x)|\leq C_j \scal{x}^{s-j}$.
\end{assumption} 

The reader may think throughout that 
\begin{equation}\label{eq:model pot}
V(x) = \left(1 + |x|^2\right)^{s/2},
\end{equation}
although we do not need such an exact formula. We construct a global-in-time Cauchy theory on the support of the measures~\eqref{eq:intro gc} for any $s>1$ in the defocusing case ($+$ sign in~\eqref{eq:intro NLS}) and any $s>8/5$ in the focusing case ($-$ sign in~\eqref{eq:intro NLS}). There is a noticeable dichotomy at $s=2$ (the harmonic oscillator). Indeed, for $s>2$, one can essentially construct the flow deterministically at the appropriate level of regularity (mostly based on tools from~\cite{YajZha-01,YajZha-04,ZhaYajLiu-06}), whereas for $s\leq 2$, we must rely on the randomization of initial data. We shall thus be particularly interested in the case $s\leq 2$, where the measures do not live on $L^2$ and thus in particular~\eqref{eq:intro c} must be interpreted in a renormalized sense (vaguely, $m = \infty - m'$ with $m' \in \R$).

Our results generalize known theorems and allow to simplify the proofs of some of them. Indeed, for the periodic NLS ($s=+\infty$ formally, restriction to a compact setting) the Cauchy theory was constructed in~\cite{Bourgain-94}, and the invariance of~\eqref{eq:intro gc} was deduced. The canonical measure~\eqref{eq:intro c} was then constructed and proved to be invariant in~\cite{OhQua-13} (see also~\cite{Brereton}). In~\cite{BurThoTzv-10} (see also~\cite{BouDebFuk-18}), the invariance of~\eqref{eq:intro gc} was obtained for $s=2$. In all other cases (in particular for the canonical measure~\eqref{eq:intro c} on $\R$ with any kind of trap), our results seem new. In particular, they answer a question raised after~\cite[Theorem 1.1]{BurThoTzv-10} concerning more general potentials than~\eqref{eq:model pot} for $s=2$, having the same behavior at infinity.

One of our motivations for considering the more general case of $s\neq +\infty, 2$ is that all approaches of the topic at hand that we are aware of rely on the spectral problem for the linear operator $h= - \partial_x ^2 + V(x)$ being exactly soluble. A detailed explicit knowledge of the eigenfunctions is used to construct the measure and the flow. This is certainly the case for $s=+\infty$ (plane waves~\cite{Bourgain-94,OhQua-13}) and $s=2$ (link with Hermite polynomials~\cite{BurThoTzv-10,Deng,RobSeoTolWan-22}) but also in other cases considered in the literature, like radial NLS on the disk or sphere (link with Bessel functions~\cite{BouBul-14a,BouBul-14b,Tzvetkov-06,Tzvetkov-08}). All (non-radial) known results in higher dimension~\cite{Bourgain-96,Bourgain-97,DenNahYue-19,DenNahYue-21,DenNahYue-22} seem to rely on plane waves. 

In this paper, we propose a softer approach to the Cauchy theory (in particular, the probabilistic one, for $s\leq 2$), which relies less on exact formulae and allows the aforementioned generalizations to $s < 2$. A price to pay is that we do not prove multilinear estimates, and consequently our results are restricted to small nonlinearities. We can allow more general non-linearities than cubic (namely, behaving like $|u|^{\kappa-2} u$ for more general, $s$-dependent, $\kappa>2$) but we only state remarks in this direction for brevity, and because the cubic nonlinearity is the most physically relevant one.

Another motivation to consider a general $s$ is that, in experiments with cold alkali gases, the trapping potential can be quite general. The link between the first-principles, many-body, description of these experiments and the above formalism was made in~\cite{LewNamRou-14d,LewNamRou-17,LewNamRou-20,FroKnoSchSoh-16,FroKnoSchSoh-20,FroKnoSchSoh-22,Sohinger-22} at the static level of equilibrium states and in~\cite{FroKnoSchSoh-17,RouSoh-22} for the dynamics. The Cauchy theory on the support of the (defocusing) measures we consider for $2<s<+\infty$ was used as a working assumption in~\cite{FroKnoSchSoh-17}, whose results it would be interesting to generalize to $s\leq 2$ in view of ours and~\cite{BurThoTzv-10}.

In the next section, we state our results and related remarks precisely. The rest of the paper is devoted to proofs.

\bigskip

\noindent\textbf{Acknowledgments.} This work was supported by the European Union's Horizon 2020 Research and Innovation Programme (Grant agreement CORFRONMAT No. 758620). After posting a first version of this paper, we had interesting exchanges with T. Robert, K. Seong, L. Tolomeo, and Y. Wang regarding the connections to their recent paper~\cite{RobSeoTolWan-22}. This has triggered the inclusion of several remarks in the present version the text. 

\section{Main results}
\label{sec:main results}
\setcounter{equation}{0}

\subsection{Random data Cauchy theory} In a finite dimensional setting, one may consider a system of ODEs 
\[
\left\{
\renewcommand*{\arraystretch}{1.3}
\begin{array}{rcl}
\partial_t x_j &=& \frac{\partial H}{\partial \xi_j}, \\
\partial_t \xi_j &=& -\frac{\partial H}{\partial x_j},
\end{array}
\quad j=1,\cdots, n
\right.
\]
with the Hamiltonian $H(x,\xi) = H(x_1,\cdots, x_n, \xi_1,\cdots, \xi_n)$. The Gibbs measure associated to the system is given by
\[
d\mu(x,\xi)=\frac{1}{Z} e^{-H(x,\xi)} dxd\xi,
\]
where $Z>0$ is a normalization constant. By the invariance of Lebesgue measure $dxd\xi=\prod_{j=1}^n dx_j d\xi_j$ (thanks to Liouville's theorem) and the conservation of the Hamiltonian $H$, the Gibbs measure $\mu$ is invariant under the time evolution of the system, i.e., for any measurable set $A \subset \R^{2n}$, $\mu(A)=\mu(\Phi(t)(A))$ for all $t\in \R$, where $\Phi(t)$ is the solution map.

Equation \eqref{eq:intro NLS} also has a Hamiltonian structure, namely
\[
\partial_t u = - i \frac{\partial H}{\partial \overline{u}},
\]
where $H=H(u)$ is the Hamiltonian given by
\[
H(u) = \langle u, h u\rangle_{L^2} \pm \frac{1}{2} \int_{\R} |u|^4. \tag{Hamiltonian}
\]
This Hamiltonian is conserved under the dynamics of \eqref{eq:intro NLS} as well as the mass
\[
M(u) =\int_{\R}|u|^2. \tag{Mass} 
\]
Following the rationale of the finite dimensional case, one expects the Gibbs measure of the form
\begin{align} \label{eq:Gibbs}
d\mu(u) = \frac{1}{Z} e^{-H(u)} du
\end{align}
to be invariant under the dynamics of \eqref{eq:intro NLS}. However, the above expression is formal since there is no infinite dimensional Lebesgue measure. 

In the seminal work \cite{LebRosSpe-88}, Lebowitz, Rose, and Speer studied, by means of techniques from constructive quantum field theory~\cite{Simon-74,GliJaf-87}, the normalizability and non-normalizability of Gibbs measures for 1D periodic NLS. More precisely, by rewriting \eqref{eq:Gibbs} as 
\begin{equation}\label{eq:Gibbs bis}
d\mu(u) = \frac{1}{Z} e^{\mp \frac{1}{2}\int_{\R}|u|^4} e^{-\langle u, h u\rangle_{L^2} } du, 
\end{equation}
they defined $\mu$ as an absolutely continuous probability measure with respect to the Gaussian measure 
\[
d\mu_0(u)=\frac{1}{Z_0} e^{-\langle u, h u\rangle_{L^2} } du.
\]
Here $h$ should be understood as $-\partial^2_x +1$ on $\T$. The above can be defined as a probability measure on $H^\theta(\T)$ for any $\theta <\frac{1}{2}$, where $H^\theta(\T)$ is the Sobolev space on the torus. For the focusing nonlinearity, the Gibbs measure is constructed with a mass cutoff, namely
\[
d\mu(u) = \frac{1}{Z} e^{\frac{1}{2}\int_{\R}|u|^4} \mathds{1}_{\left\{\int_{\R^2}|u|^2 \leq m\right\}} d\mu_0(u).
\]
It was claimed in \cite{LebRosSpe-88} that this measure is normalizable (i.e. the partition function $Z$ is a finite positive number) if $2<p<6$ for any mass cutoff $m>0$, and if $p=6$ for any $m<\|Q\|^2_{L^2}$, where $Q$ is the unique (up to symmetries) optimizer for the Gagliardo--Nirenberg inequality
\[
\|u\|^6_{L^6(\R)} \leq C_{\opt} \|\partial_x u\|^2_{L^2(\R)} \|u\|^4_{L^2(\R)}.
\]
The probabilistic proof for the normalizability presented in \cite{LebRosSpe-88} however contains a gap, as pointed out and repaired for $p<6$ in \cite{CarFroLeb-16}  (see \cite[Introduction]{OhSoTo-22} for more details). Later, Bourgain \cite{Bourgain-94} gave an alternative proof, using an analytic Fourier approach, for the normalizability when $2<p<6$ with any $m>0$, and when $p=6$ with $m>0$ sufficiently small. Recently, Oh, Sosoe, and Tolomeo \cite{OhSoTo-22} provided a proof for the normalizability when $p=6$ and  $m \leq \|Q\|^2_{L^2}$, thus fully filling the gap in~\cite{LebRosSpe-88} and, remarkably, extending the result to a value of the $L^2$ mass at which blow-up occurs for NLS dynamics. 

In \cite{Bourgain-94}, Bourgain pursued the study of Gibbs measures for 1D periodic NLS and proved the invariance of $\mu$ under the NLS flow. Here, by invariance, we mean that $\mu(A)=\mu(\Phi(t)(A))$ for any measurable set $A \subset H^\theta(\T)$ with some $\theta<\frac{1}{2}$ and any $t\in \R$, where $\Phi(t)$ is the solution map. Moreover, there exists a set of full $\mu$-measure such that the solution exists globally in time for all initial data belonging to this set. Such a result is usually referred to as almost sure global existence. In this context, the invariant Gibbs measure serves as a substitute for conserved quantities to control the growth in time of solutions. This allows to extend local in time solutions to global ones almost surely. 

There are many other works devoted to invariant Gibbs measure and almost-sure global existence on the support of Gibbs measure for other dispersive equations (see e.g., \cite{Zhidkov-91, Bourgain-96} for periodic NLS, \cite{Tzvetkov-06, Tzvetkov-08} for NLS on the disc, \cite{Friedlander-85, Zhidkov-94, BurTzv-08b, BurTzv-07, Suzzoni-11, BouBul-14a} for nonlinear wave equations, \cite{Oh-DIE, Oh-SIAM} for KdV-type systems, and \cite{NOBS, ThoTzv-10} for 1D derivative NLS,...). 

In the above-mentioned works, invariant Gibbs measures were constructed in compact settings (torus or ball). There are much fewer works addressing the invariant Gibbs measure on non-compact frameworks (see e.g., \cite{BurThoTzv-10, Deng, RobSeoTolWan-22} for NLS with harmonic potential, \cite{CacSuz-14} for 1D NLS with cubic nonlinearity multiplied by a sufficiently smooth and integrable function, and also \cite{Thomann-09, PRT} for almost sure well-posedness with (an)-harmonic potentials). 

\subsection{Grand-canonical measures}

The first goal of this paper is to extend the result of Burq, Thomann, and Tzvetkov \cite{BurThoTzv-10} to the 1D cubic NLS with potentials satisfying Assumption \ref{assu-V}. We first recall the rigorous definition of the measure~\eqref{eq:intro gc}, setting $\nu = 0$ for simplicity of notation.

We start with the definition of the Gaussian measure. Under Assumption~\ref{assu-V}, the linear operator $h$ is Hermitian and has compact resolvent. We write its' spectral decomposition as 
\begin{equation}\label{eq:spectral}
h = \sum_{j\geq 1} \lambda_j |u_j\rangle \langle u_j|
\end{equation}
with a non-decreasing sequence of positive eigenvalues $\lambda_j \to \infty$ and the associated eigenfunctions $u_j$. For $\theta \in \R$, we introduce the Sobolev space associated to $h$ as
\begin{align} \label{eq:H-theta}
\Hc^\theta :=\left\{ u =\sum_{j\geq 1} \alpha_j u_j : \|u\|_{\Hc^\theta} := \Big(\sum_{j\geq 1} \lambda_j^\theta |\alpha_j|^2\Big)^{1/2}<\infty \right\}
\end{align}
with 
\begin{align} \label{eq:alpha-j}
\alpha_j = \scal{u_j, u}_{L^2}.
\end{align} 
For $\beta \geq 0$ and $1\leq p \leq \infty$, we define 
\begin{align} \label{eq:Sobo-h}
\Wc^{\beta, p} :=\left\{u \in \Sh'(\R) : h^{\beta/2} u\in L^p(\R)\right\}
\end{align}
which is equipped with the norm 
\[
\|u\|_{\Wc^{\beta,p}} := \|h^{\beta/2} u\|_{L^p}.
\]
Here $\Sh'(\R)$ is the space of tempered distributions on $\R$. When $p=2$, we actually have $\Wc^{\beta,2} \equiv \Hc^\beta$.

\begin{definition}[\textbf{Gaussian measure}]\label{def:gaussian}\mbox{}\\
Let $\Lambda \geq \lambda_1$. On 
\begin{align} \label{eq:E Lamb}
E_{\leq \Lambda} := \mathrm{span} \left\{u_j : \lambda_j \leq \Lambda\right\},
\end{align}
we define the finite-dimensional Gaussian measure 
\begin{align} \label{eq:mu-0-Lamb}
d\mu_0^{\leq \Lambda} (u) := \prod_{\lambda_j\leq \Lambda} \frac{\lambda_j}{\pi} e^{-\lambda_j |\alpha_j|^2} d \alpha_j,
\end{align}
where $\alpha_j$ is as in \eqref{eq:alpha-j} and $d\alpha_j = d\rea(\alpha_j) d\ima(\alpha_j)$ is the Lebesgue measure on $\C$. 

The sequence of measures $\{\mu_0^{\leq \Lambda}\}_{\Lambda \geq \lambda_1}$ is tight in the Hilbert space $\Hc^\theta$ for any $\theta < \frac{1}{2}-\frac{1}{s}$ with $s>1$ as in Assumption~\ref{assu-V}. Consequently, it defines a probability measure $\mu_0$ on this space, having $\mu_0^{\leq \Lambda}$ as cylindrical projection on $E_{\leq \Lambda}$.
\end{definition}

The tightness is proved in~\cite[Example~3.2]{LewNamRou-14d}. We recall the main argument in Appendix~\ref{sec:app} below for the convenience of the reader. The existence of the infinite-dimensional measure then follows from~\cite[Lemma 1]{Skorokhod-74}. 

For the interacting measures, we have the following result.

\begin{proposition}[\textbf{Grand-canonical measures}]\label{pro:gc}\mbox{}\\
Let $s>1$, $V$ satisfy Assumption~\ref{assu-V}, and $\mu_0$ as defined above. 

\noindent (1) \textbf{Defocusing case.} For any $s>1$, the map $u\mapsto e^{- \frac{1}{2}\int_\R |u|^4}$ is in $L^1(d\mu_0)$. Consequently, 
\begin{equation}\label{eq:gc defoc}
d\mu (u) :=  \frac{1}{Z} e^{- \frac{1}{2}\int_\R |u|^4} d\mu_0 (u)
\end{equation}
makes sense as a probability measure.

\noindent (2) \textbf{Focusing case, $s>2$.} For any $s>2$ and any $m>0$, the map $u\mapsto e^{\frac{1}{2}\int_\R |u|^4}\one_{\left\{\int_\R |u|^2 \leq m\right\}}$ is in $L^1(d\mu_0)$. Consequently,
\begin{equation}\label{eq:gc foc 1}
d\mu (u) :=  \frac{1}{Z} e^{ \frac{1}{2}\int_\R |u|^4} \one_{\left\{\int_\R |u|^2 \leq m\right\}} d\mu_0 (u)
\end{equation}
makes sense as a probability measure.

\noindent (3) \textbf{Focusing case, $\frac{8}{5}<s \leq 2$.} For any $1<s\leq 2$, the sequence $\{\Mcal_{\leq \Lambda}(u)\}_{\Lambda \geq \lambda_1}$, with
\[ 
\Mcal_{\leq \Lambda}(u) := \sum_{\lambda_j\leq \Lambda} \left( |\alpha_j| ^2 - \int |\alpha_j| ^2 d\mu_0 (u)\right),
\]
is Cauchy in $L^2 (d\mu_0)$. It has thus a limit (the renormalized mass), denoted by $\Mcal(u)$. In addition, for any $\frac{8}{5}<s\leq 2$ and any $m>0$, the map $u\mapsto e^{\frac{1}{2}\int_\R |u|^4} \one_{\left\{| \Mcal (u)| \leq m\right\}}$ is in $L^1 (d\mu_0)$, hence
\begin{equation}\label{eq:gc foc 2}
d\mu (u) :=  \frac{1}{Z} e^{ \frac{1}{2}\int_\R |u|^4} \one_{\left\{ |\Mcal (u)| \leq m \right\}} d\mu_0 (u)
\end{equation}
makes sense as a probability measure.
\end{proposition}

The defocusing case is dealt with in~\cite[Section~5]{LewNamRou-14d} for $s>2$ and in~\cite[Section~3]{LewNamRou-17} for $s>1$, generalizing~\cite[Section~3]{BurThoTzv-10} for $s=2$. We recall the arguments below. The definition of the focusing measure for $s\neq 2, +\infty$ is new. For the harmonic potential $V(x)=|x|^2$, the construction of the measure was performed in~\cite[Section~3]{BurThoTzv-10} with a continuous cut-off $\zeta(\Mcal(u))$ instead of the rough one in~\eqref{eq:gc foc 2}. We vindicate that the above definitions make sense in Section~\ref{sec:gc measures} below.

We have the following result for the evolution of these measures under a suitably defined flow.

\begin{theorem}[\textbf{Invariance of grand-canonical measures}]\label{theo-Gibbs-meas}\mbox{}\\
Let $s>1$ and $V$ satisfy Assumption \eqref{assu-V}. Assume in addition that $s>\frac{8}{5}$ for the focusing nonlinearity. Then there exist $\theta<\frac{1}{2}-\frac{1}{s}$ and a set $\Sigma \subset \Hc^\theta$ such that:
	
\noindent (1) $\mu(\Sigma)=1$;
	
\noindent (2) Equation \eqref{eq:intro NLS} is globally well-posed for initial data $f\in \Sigma$ with flow $\Phi(t):\Sigma \mapsto \Sigma$;
	
\noindent (3) $\mu$ is invariant under the flow of \eqref{eq:intro NLS}, $\mu (\Phi(t)A) = \mu (A)$ for all measurable sets $A\subset \Sigma$;
	
\noindent (4) There exists $C>0$ such that for any $f\in \Sigma$, 
		\[
		\|\Phi(t) f \|_{\Hc^\theta} \leq C \left(\omega(f)+\log^{\frac{1}{2}} (1+|t|)\right), \quad \forall t\in \R
		\]
		for some constant $\omega(f)$ depending on $f$.
\end{theorem}

An  ingredient of our proof is the following estimate~(see e.g., \cite{LewNamRou-14d, LewNamRou-17, LewNamRou-20}) on the $k$-particle density matrix associated to $\mu_0$
\begin{align} \label{eq:mu-0-k-intro}
\int |u^{\otimes k}\rangle \langle u^{\otimes k}| d\mu_0(u) \leq k! (h^{-1})^{\otimes k}, \quad \forall k\geq 1.
\end{align}
Another observation (see Lemma \ref{lem-Lp}) is that for $0\leq \beta<\frac{1}{2}$, the function 
\begin{equation}\label{eq:Green Lp}
x\mapsto h^{\beta-1}(x,x) \in L^p(\R) 
\end{equation}
for all $\max\left\{1,\frac{2}{s(1-2\beta)}\right\} <p \leq \infty$, where
\begin{align} \label{eq:ker-h}
h^{-1}(x,y)=\sum_{j\geq 1} \lambda_j^{-1} u_j(x) \overline{u}_j(y)
\end{align}
is the integral kernel of $h^{-1}$. This property together with \eqref{eq:mu-0-k-intro} enable us to prove (see Lemma \ref{lem-expo-p}) that the Gaussian measure $\mu_0$ is supported on Sobolev spaces $\Wc^{\beta,p}$. Moreover, there exist $C,c>0$ such that
\[
\int e^{c\|u\|^2_{\Wc^{\beta, p}}} d\mu_0 (u) \leq C
\]
provided that $0\leq \beta <\frac{1}{2}$ and $p>\max\left\{2,\frac{4}{s(1-2\beta)}\right\}$ an even integer. As a result, we are able to define the defocusing Gibbs measure for all $s>1$. Combining with a Gagliardo--Nirenberg inequality from~\cite{BreMir-19}, we define the focusing measure for $s>\frac{8}{5}$ (a restriction which is perhaps technical). Our proof of~\eqref{eq:Green Lp} uses Lieb--Thirring-type inequalities from~\cite{DolFelLosPat-06} and standard inequalities for operators in Schatten ideals (H\"older and Kato--Seiler--Simon~\cite{Simon-79}) rather than $L^p$-bounds on individual eigenfunctions. This makes it applicable to general potentials. Since $\mu_0$ is invariant under the linear Schr\"odinger flow associated with $h$, we also have that for any time $t$, 
\[
\|e^{-\im t h} u\|_{\Wc^{\beta, p}}  < \infty \quad \mu_0 \mbox{ almost surely}.
\]
Combining with fractional chain rules, we directly recover estimates e.g., like 
\[ 
\|(e^{-\im t h} u)^2\|_{\Hc^\beta}  < \infty \quad \mu_0 \mbox{ almost surely}
\]
at $s=2$, which was originally obtained as~\cite[Equation~(1.2)]{BurThoTzv-10} using bilinear estimates for Hermite functions. The above probabilistic Strichartz estimates are important tools in our construction of a local Cauchy theory.  

For super-harmonic potentials, i.e., $s>2$, since the Gibbs measure is supported on $L^2$-based Sobolev spaces of positive indices, there is hope to obtain a deterministic local well-posedness on its' support. This is indeed feasible, using Strichartz estimates with a loss of derivatives proved by Yajima and Zhang~\cite{YajZha-04}. For (sub)-harmonic potentials, i.e., $1<s\leq 2$, the Gibbs measure lives on $L^2$-based Sobolev spaces of negative indices, so it is difficult to obtain a satisfying deterministic local theory on its' support. In \cite{BurThoTzv-10}, such a deterministic result was proved by using a delicate multilinear estimate which relies heavily on $L^p$-bound of derivatives of eigenfunctions obtained in \cite{KocTat-05}. Here we give a softer argument and prove an almost sure local well-posedness for the equation. 

\subsection{Canonical measures}

We next deal with the construction and invariance of canonical Gibbs measures~\eqref{eq:intro c} conditioned on the mass. This type of measures has so far been constructed only for the 1D periodic NLS~\cite{OhQua-13,CarFroLeb-16,CarFroLebWan-19} and for the 1D periodic derivative NLS~\cite{Brereton}. To the best of our knowledge, there is no work constructing canonical Gibbs measures in non-compact settings.

To give rigorous meaning to~\eqref{eq:intro c}, we follow the basic strategy of~\cite{OhQua-13}. First, we construct a Gaussian measure conditioned on the $L^2$ mass. For $s\leq 2$ it is infinite almost surely, so that we need to use the renormalized mass as defined in Item (3) of Proposition~\ref{pro:gc}. To unify the presentation we always condition the Gaussian measure on the renormalized mass, keeping in mind that when $s>2$, the $L^2$ mass is finite almost surely, equal to the renormalized one plus its finite expectation with respect to the Gaussian measure.

More precisely, for $m \in \R$ and $m>-{\Tr}[h^{-1}]$ if $s>2$, we will take the limit $\vareps \to 0^+$ of the approximate canonical Gaussian measure
\begin{equation}\label{eq:mu-0-m-eps}
d\mu_0^{m,\vareps}(u) = \frac{1}{Z_0^{m,\vareps}} \mathds{1}_{\left\{m-\vareps<\Mcal(u)<m+\vareps\right\}} d\mu_0(u),
\end{equation}
where $\Mcal(u)$ is as in Proposition~\ref{pro:gc}. Since we aim at conditioning on a zero-probability event, the normalization constant 
\[
Z_0^{m,\vareps} = \mu_0\left(m-\vareps<\Mcal(u)<m+\vareps\right)
\]
converges to zero as $\vareps \to 0^+$, hence taking the limit in~\eqref{eq:mu-0-m-eps} is not straightforward. We however prove that 
\[
\lim_{\vareps \to 0^+} \frac{1}{2\vareps} Z_0^{m,\vareps} > 0
\]
so that the definition 
\[
 d\mu_0^{m}(u) := \lim_{\vareps \to 0^+} d\mu_0^{m,\vareps}(u)
\]
indeed yields a probability measure. Then the interacting canonical Gibbs measure is defined by
\[
d\mu^m(u) = \frac{1}{Z^m} e^{\mp \frac{1}{2}\int_{\R}|u|^4} d\mu_0^m(u).
\]
Its' invariance with respect to the NLS flow is a direct consequence of the invariance of the standard Gibbs measure in Theorem~\ref{theo-Gibbs-meas}. We summarize this in the following theorem.

\begin{theorem}[\textbf{Canonical Gibbs measures}]\label{theo-Gibbs-fix-mass}\mbox{}\\
	Let $s>1$, $V$ satisfy Assumption \ref{assu-V}, $m\in \R$, and assume $m>-{\Tr}[h^{-1}]$ if $s>2$. Assume in addition that $s>\frac{8}{5}$ for the focusing nonlinearity. Then $\mu^m$ makes sense as a probability measure. In addition, it is invariant under the flow of~\eqref{eq:intro NLS} (as defined in Theorem~\ref{theo-Gibbs-meas}).
\end{theorem}
%
Canonical measures were constructed before in~\cite{OhQua-13,Brereton}. In these references the normalization of the canonical Gibbs measure is proved by using a large deviation principle, a $L^2$-based Sobolev embedding (for instance, $\Hc^{1/4}\subset L^4(\R)$), and a sum over dyadic pieces. Adapting this argument to our context results in an unnecessary restriction to $s>4$ because of the use of the $L^2$-based Sobolev embedding. Our approach is different, based primarily on the almost sure $L^p$-regularity alluded to in the discussion below Theorem~\ref{theo-Gibbs-meas}. This allows us to define the canonical Gibbs measure as soon as the grand-canonical one is defined.

\subsection{Possible extensions of the main results}
 Before turning to the more technical parts of the paper, we emphasize that the method we present therein can be adapted to 1D NLS with anharmonic potential (as in Assumption \ref{assu-V}) and more general nonlinearities, namely
\begin{align}\label{eq:NLS-general}
	\im  \partial_t u = h u \pm |u|^{\kappa-2} u, \quad (t,x) \in \R \times \R
\end{align}
with $\kappa>2$. More precisely, we have the following observations:

\medskip

\noindent \textbf{1}. Regarding the construction of (grand)-canonical Gibbs measures (see Remarks \ref{rem-Gibbs-general} and \ref{rem-meas-inva-mass-general}), the following cases can be covered by straightforward adaptations of our arguments: 

\medskip

 $\bullet$ defocusing case with  $\max\left\{2,\frac{4}{s}\right\}<\kappa<\infty$.

\medskip

 $\bullet$ focusing case with $s>2$ and $2<\kappa<6$, $m>0$ or with $\frac{14}{9}<s\leq 2,m>0$ and 
 $$\frac{4}{s}<\kappa<\frac{3s+2+\sqrt{9s^2 +4s -28}}{2}.$$ 

\medskip

\noindent \textbf{2}. Assuming the (grand)-canonical Gibbs measure can be constructed (in particular, in the cases above), our methods prove its' invariance under the flow in the following cases (see Remarks \ref{rem-meas-inva-sup}, \ref{rem-meas-inva-sub}, and \ref{rem-meas-inva-mass-general}) : 

\medskip

 $\bullet$ super-harmonic potential $s>2$ with $2<\kappa<4+s$.

\medskip

 $\bullet$ (sub)-harmonic potential $s\leq2$ with $\frac{4}{s}<\kappa<6$.
 
 \medskip

In a recent work, Robert et al. \cite{RobSeoTolWan-22} studied the focusing Gibbs measure with harmonic potential $V(x)=|x|^2$. They proved that the measure is normalizable if and only if $2<\kappa<6$ and $m>0$. The proof of the normalizability in \cite{RobSeoTolWan-22} is based on the Bou\'e-Dupuis variational formula. When $s=2$, our method also gives normalizability of the focusing Gibbs measure for any $2<\kappa<6$ and $m>0$, thus providing an alternative proof for part of the results from~\cite{RobSeoTolWan-22}. In turn, if supplemented with some of our estimates from Section~\ref{sec:gc measures}, the arguments of~\cite{RobSeoTolWan-22} allow\footnote{We learned this from Yuzhao Wang and Leonardo Tolomeo.} to extend the construction of the measure to the case $s<2$ and $\frac{4}{s} < \kappa < 2s+2$. As mentioned above, our construction of the global probabilistic Cauchy theory extends to this case. 

For large powers of the nonlinearity ($\kappa \geq 4+s$ when $s>2$, or $\kappa\geq 6$ when $s\leq 2$), the invariance of the defocusing Gibbs measures remains an open problem, except for $s=2$, cf. ~\cite{BurThoTzv-10} .  


\subsection{Organization of the paper} 
Section \ref{sec:gc measures} is devoted to the construction of Gibbs measures associated to \eqref{eq:intro NLS}. We also define and prove some properties of approximate measures which are needed for the proof of measure invariance. In Section~\ref{sec:invari super}, we consider the Cauchy problem in the case of super-harmonic potential, $s >2$. The more difficult Cauchy problem for (sub)-harmonic potentials, $s\leq 2$, will be addressed in Section~\ref{sec:invari sub}. The construction as well as the invariance of canonical measures conditioned on the mass are considered in Section~\ref{sec:cano measures}. Some estimates and inequalities from the literature, used throughout the paper are recalled in appendices for the convenience of the reader.

\section{Grand-canonical measures}
\label{sec:gc measures}
\setcounter{equation}{0}

\subsection{Basic estimates and the defocusing case}

The definition of the Gaussian measure $\mu_0$ was recalled in Definition~\ref{def:gaussian}. Several of our estimates will be based on the following observation. We quote it from~\cite{LewNamRou-14d} but it is probably well-known, see e.g.,~\cite{FroKnoSchSoh-16,FroKnoSchSoh-17} and references therein.
	
	\begin{lemma}[\textbf{Density matrices of the Gaussian measure}] \label{lem-mu-0}\mbox{}\\
		Let $s>1$, $V$ satisfy Assumption \ref{assu-V}, and $\theta<\frac{1}{2}-\frac{1}{s}$. Then there exists a unique measure $\mu_0$ living over $\Hc^\theta$ such that for every $\Lambda \geq \lambda_1$, $\mu_0^{\leq \Lambda}$ is the cylindrical projection of $\mu_0$ on $E_{\leq\Lambda}$. Moreover, for every integer $k\geq 1$,
		\begin{align} \label{eq:mu-0-k}
		\int |u^{\otimes k}\rangle \langle u^{\otimes k}| d\mu_0(u) \leq k! (h^{-1})^{\otimes k}
		\end{align}
		as operators. In particular, for any self-adjoint operator $A$, 
		\begin{align} \label{eq:mu-0-k-A}
		\int |(Au)^{\otimes k}\rangle \langle (Au)^{\otimes k}| d\mu_0(u) \leq k! (Ah^{-1}A)^{\otimes k}.
		\end{align}
	\end{lemma}
	
		As explained in~\cite[Lemma 3.3]{LewNamRou-14d}, Wick's theorem for Gaussian measures yields
		\[
		\int |u^{\otimes k}\rangle \langle u^{\otimes k}| d\mu_0(u) =k! P^k_s (h^{-1})^{\otimes k} P^k_s,
		\]
		where $P^k_s$ is the orthogonal projection on $k$-symmetric functions, i.e.,
		\[
		P^k_s v(x_1,\ldots,x_k) = \frac{1}{k!} \sum_{\sigma \in \Pi(k)} v\left(x_{\sigma(1)}, \ldots, x_{\sigma(k)}\right)
		\]
		with $\Pi(k)$ the group of all permutations of $\{1,\cdots, k\}$. Since $P^k_s$ commutes with $(h^{-1})^{\otimes k}$, we have
		\[
		(h^{-1})^{\otimes k} = P^k_s (h^{-1})^{\otimes k}P^k_s + (1-P^k_s) (h^{-1})^{\otimes k} (1-P^k_s)
		\]
		and~\eqref{eq:mu-0-k} follows. 
		
		As a direct consequence of density matrices, we have the following decay of $\Hc^\theta$-norms with respect to the Gaussian measure $\mu_0$.
		
		\begin{lemma}[\textbf{$L^2$-Regularity on the support of the Gaussian measure}] \label{lem-H-theta}\mbox{}\\ 
			Let $s>1$, $V$ satisfy Assumption \ref{assu-V}, and $\theta <\frac{1}{2}-\frac{1}{s}$. Then there exist $C,c>0$ such that 
			\begin{align} \label{eq:expo-H-theta}
			\int e^{c\|u\|^2_{\Hc^\theta}} d\mu_{0}(u) \leq C.
			\end{align}
			In particular, for all $\lambda>0$
			\begin{align} \label{eq:boun-H-theta}
			\mu_{0}\left(u \in \Hc^\theta : \|u\|_{\Hc^\theta}>\lambda\right) \leq Ce^{-c\lambda^2}.
			\end{align}
		\end{lemma}
		
		\begin{proof}
			It suffices to prove \eqref{eq:expo-H-theta} since \eqref{eq:boun-H-theta} follows from \eqref{eq:expo-H-theta} and the Chebyshev inequality. We write
			\[
			\int e^{c\|u\|^2_{\Hc^\theta}} d\mu_{0}(u) = \sum_{k\geq 0} \frac{c^k}{k!} \int \|u\|^{2k}_{\Hc^\theta} d\mu_{0}(u)
			\]
			and observe that
			\begin{align*}
			\int \|u\|^{2k}_{\Hc^\theta} d\mu_{0}(u) &= \int \underbrace{\|u\|^2_{\Hc^\theta} \cdots \|u\|^{2}_{\Hc^\theta}}_{k \text{ times}} d\mu_{0}(u) \\
			&=\int \cdots \int \Big(\int |h^{\theta/2} u(x_1)|^2 \cdots |h^{\theta/2}u(x_k)|^2 d\mu_{0}(u)\Big) dx_1 \cdots dx_k.
			\end{align*}
			Denote $\delta^{(\eta)}_x$ a mollification of the Dirac delta function at $x$ so that 
			$$ \delta^{(\eta)}_x \underset{\eta \to 0}{\to} \delta_x$$
			as measures. Identifying it with the associated multiplication operator we have
			\begin{multline*}
			\int |h^{\theta/2} u(x_1)|^2 \cdots |h^{\theta/2}u(x_k)|^2 d\mu_{0}(u) \\
			= \lim_{\eta \to 0} {\Tr}\left[\delta_{x_1} ^{(\eta)} \otimes \cdots \otimes \delta_{x_k} ^{(\eta)} \left(\int |(h^{\theta/2}u)^{\otimes k}\rangle \langle (h^{\theta/2}u)^{\otimes k}| d\mu_{0}(u)\right)\delta_{x_k} ^{(\eta)} \otimes \cdots \otimes \delta_{x_1} ^{(\eta)}\right].
			\end{multline*}
			But~\eqref{eq:mu-0-k} implies  
			\begin{multline*} 
			\Tr\left[\delta_{x_1} ^{(\eta)} \otimes \cdots \otimes \delta_{x_k} ^{(\eta)} \left(\int |(h^{\theta/2}u)^{\otimes k}\rangle \langle (h^{\theta/2}u)^{\otimes k}| d\mu_{0}(u)\right)\delta_{x_k} ^{(\eta)} \otimes \cdots \otimes \delta_{x_1} ^{(\eta)}\right]\\ \leq k! {\Tr}\left[\delta_{x_1} ^{(\eta)} \otimes \cdots \otimes \delta_{x_k} ^{(\eta)} \left(h^{\theta-1}\right)^{\otimes k} \delta_{x_k} ^{(\eta)} \otimes \cdots \otimes \delta_{x_1} ^{(\eta)} \right]
			\end{multline*}
            and since 
            $$ \Tr \left[ \delta ^{(\eta)}_x h^{\theta-1} \delta ^{(\eta)}_x \right] = \iint \delta ^{(\eta)}_x (y_1) h^{\theta-1}(y_1, y_2) \delta ^{(\eta)}_x (y_2) dy_1 dy_2 $$
            we may let $\eta \to 0$ to deduce 
            $$ \int |h^{\theta/2} u(x_1)|^2 \cdots |h^{\theta/2}u(x_k)|^2 d\mu_{0}(u) \leq k! h^{\theta-1}(x_1, x_1) \cdots h^{\theta-1}(x_k,x_k).$$
			This implies
			\[
			\int \|u\|^{2k}_{\Hc^\theta} d\mu_{0}(u) \leq k! \Big(\int_{\R} h^{\theta-1}(x,x)dx\Big)^k = k!\left({\Tr}[h^{-(1-\theta)}]\right)^k.
			\]
			In particular, we obtain
			\[
			\int e^{c\|u\|^2_{\Hc^\theta}} d\mu_{0}(u) \leq \sum_{k\geq 0} \left(c{\Tr}[h^{-(1-\theta)}]\right)^k \leq C
			\]
			provided that $c {\Tr}[h^{-(1-\theta)}]<1$. This proves \eqref{eq:expo-H-theta}. Note that ${\Tr}[h^{-(1-\theta)}]<\infty$ due to $\theta<\frac{1}{2}-\frac{1}{s}$ (see Lemma \ref{lem-trac-h-p}). 
		\end{proof}
	
		Regularity properties of typical samples of the Gaussian measure are, as per \eqref{eq:mu-0-k}, connected to properties of the covariance $h^{-1}$ (Green function of the Schr\"odinger operator). When $s>2$, it is relatively easy to construct the (defocusing, at least) interacting measures since $h^{-1}$ is a trace-class operator, see~\cite[Example~5.2]{LewNamRou-14d} and Appendix~A. For $s\leq 2$, the following lemma plays a key role in our analysis. It generalizes estimates from~\cite[Section~3]{LewNamRou-17}.

\begin{lemma}[\textbf{Integrability of the density}]\label{lem-Lp}\mbox{}\\
	Let $s>1$, $V$ satisfy Assumption \ref{assu-V}, and $0\leq \beta <\frac{1}{2}$. Then $x \mapsto h^{\beta-1}(x,x)$ is in $L^p(\R)$ for all
	\[
	\max\left\{1,\frac{2}{s(1-2\beta)}\right\} <p \leq \infty,
	\]
	where $h^{\beta-1}(x,y)$ is the integral kernel of $h^{\beta-1}$ defined as in~\eqref{eq:ker-h}.
\end{lemma}
\begin{proof}
	It suffices to prove that for any multiplication operator $\chi\geq 0$ satisfying $\chi^2 \in L^q(\R)$ with $\frac{1}{p}+\frac{1}{q}=1$, we have
	\begin{equation}\label{eq:tsoin}
	\int_{\R} \chi^2(x) h^{\beta-1}(x,x) dx = {\Tr}[\chi h^{\beta-1} \chi]=\|h^{(\beta-1)/2} \chi\|^2_{\Sc^2} \leq C \|\chi^2\|_{L^q(\R)},
	\end{equation}
	where for any $1 \leq p \leq \infty$,
	\[
	\Sc^p:= \left\{A : \|A\|_{\Sc^p}:=\left({\Tr}[(A^*A)^{p/2}]\right)^{1/p}<\infty \right\}
	\]
	is the $p$-th Schatten class of operators on a Hilbert space~\cite{Schatten-60,Simon-79} ($p=1,2,\infty$ correspond respectively to trace-class, Hilbert-Schmidt, and compact operators).  
	
	For $0<\alpha <\frac{1-\beta}{2}$, we write
	\[
	h^{(\beta-1)/2} \chi = h^{\alpha+(\beta-1)/2} \left(h^{-\alpha}(1-\partial^2_x)^\alpha\right) \left((1-\partial^2_x)^{-\alpha} \chi\right).
	\]
	We have
	\[
	h^{\alpha+(\beta-1)/2} \in \Sc^{2p} \Longleftrightarrow {\Tr} \left[h^{-2p\left(\frac{1-\beta}{2}-\alpha\right)}\right] <\infty,
	\]
	which holds provided that (see Lemma \ref{lem-trac-h-p})
	\[
	2p\left(\frac{1-\beta}{2} -\alpha\right) >\frac{1}{2}+\frac{1}{s}.
	\]
	Since $h\geq \frac{1}{2}(-\partial^2_x+\lambda_1)$ with $\lambda_1>0$ the lowest eigenvalue of $h$, we infer that $h\geq C(1-\partial^2_x)$ for some $C>0$. Since $\alpha < \frac{1}{2}$ The operator monotonicity of $x\mapsto x^{2\alpha}$ (see \cite[Theorem 2.6]{Carlen-10}) implies
	\[
	h^{2\alpha} \geq C^{2\alpha} (1-\partial^2_x)^{2\alpha}
	\]
	or 
	\[
	h^{-\alpha}(1-\partial^2_x)^{2\alpha} h^{-\alpha} \leq \frac{1}{C^{2\alpha}}
	\]
	hence $h^{-\alpha}(1-\partial^2_x)^\alpha$ is a bounded operator for $0<\alpha<(1-\beta)/2$.
	
	By the Kato--Seiler--Simon inequality (see e.g., \cite[Theorem 4.1]{Simon-79}) for $1\leq r <\infty$,
	\[ 
	\|f (-\im \nabla) g(x) \|_{\Sc^{r}} \leq \|f\|_{L^{r}} \|g\|_{L^{r}},
	\]
	we have
	\[
	\|(1-\partial^2_x)^{-\alpha} \chi\|_{\Sc^{2q}} \leq \left(\int_{\R}  \frac{d\xi}{(1+|\xi|^2)^{2\alpha q}}\right)^{1/2q} \|\chi\|_{L^{2q}(\R)} \leq C\|\chi\|_{L^{2q}(\R)}
	\]
	provided that $4\alpha q>1$. Here we need $q<\infty$ hence $p>1$.
	
	Combining these estimates, the H\"older inequality in Schatten spaces (see \cite[Theorem 2.8]{Simon-79}) yields
	\begin{align}\nonumber
	\|h^{(\beta-1)/2} \chi\|^2_{\Sc^2} &\leq \|h^{\alpha+(\beta-1)/2}\|_{\Sc^{2p}}^2 \|h^{-\alpha}(1-\partial^2_x)^\alpha\|^2_{\Sc^\infty} \|(1-\partial^2_x)^{-\alpha} \chi\|^2_{\Sc^{2q}} \\
	&\leq C\|\chi\|^2_{L^{2q}(\R)} = C\|\chi^2\|_{L^q(\R)}\label{eq:pouet}.
	\end{align}
	This estimate holds true if the following conditions are fulfilled:
	\begin{align} \label{eq:cond-alph}
	2p\left(\frac{1-\beta}{2}-\alpha\right) >\frac{1}{2}+\frac{1}{s}, \quad 4\alpha q>1.
	\end{align}
	For $0\leq \beta<1/2$ and $\max\left\{1,\frac{2}{s(1-2\beta)}\right\}<p\leq \infty$, we pick $0<\alpha<\frac{1-\beta}{2}$ such that
	\[
	\frac{1}{4q}<\alpha<\frac{1-\beta}{2}-\frac{1}{2p}\left(\frac{1}{2}+\frac{1}{s}\right),
	\]
	we see that~\eqref{eq:cond-alph} is satisfied and the result follows by inserting~\eqref{eq:pouet} into~\eqref{eq:tsoin}.
\end{proof}

Applying the above, we will deduce that $\mu_0$ is supported on Sobolev spaces $\Wc^{\beta,p}$ based on $h$ as in~\eqref{eq:Sobo-h}. In fact, we have the following Fernique-type estimate giving the decay of such norms.

\begin{lemma}[\textbf{$L^p$-Regularity on the support of the Gaussian measure}] \label{lem-expo-p}\mbox{}\\
	Let $s>1$, $V$ satisfy Assumption \ref{assu-V}, $0\leq \beta<\frac{1}{2}$, and $p>\max\left\{2,\frac{4}{s(1-2\beta)}\right\}$ be an even integer. Then there exist $C,c>0$ such that 
	\begin{align} \label{eq:expo-p}
		\int e^{c\|u\|^2_{\Wc^{\beta,p}}} d\mu_0(u) \leq C.
	\end{align}
	In particular, for $\theta<\frac{1}{2}-\frac{1}{s}$ and all $\lambda>0$,
	\begin{align} \label{eq:boun-beta-p}
		\mu_0\left( u \in \Hc^\theta : \|u\|_{\Wc^{\beta, p}}>\lambda\right) \leq C e^{-c\lambda^2}.
	\end{align}
\end{lemma}

\begin{proof}
	It suffices to prove \eqref{eq:expo-p}. To this end, we write
	\[
	\int e^{c\|u\|^2_{\Wc^{\beta,p}}} d\mu_0(u) = \sum_{k\geq 0} \frac{c^k}{k!} \int \|u\|^{2k}_{\Wc^{\beta,p}} d\mu_0(u)
	\]
	and estimate the summands.
	
	 For $2k=pm$ with $m\geq 0$ an integer, we have
	\begin{align*}
	\int \|u\|^{2k}_{\Wc^{\beta,p}} d\mu_0(u) &= \int \underbrace{\|u\|^p_{\Wc^{\beta,p}} \cdots \|u\|^p_{\Wc^{\beta,p}}}_{m \text{ times}} d\mu_0(u) \\
	&=\int_{\R} \cdots \int_{\R} \Big(\int |h^{\beta/2} u(x_1)|^p \cdots |h^{\beta/2}u(x_m)|^p d\mu_0(u)\Big) dx_1 \cdots dx_m.
	\end{align*}
	Using \eqref{eq:mu-0-k-A}, we see that\footnote{Strictly speaking the Dirac delta functions must be regularity as in the proof of Lemma~\ref{lem-H-theta}.}
	\begin{align*}
	&\int |h^{\beta/2} u(x_1)|^p \cdots |h^{\beta/2}u(x_m)|^p d\mu_{0}(u) \\
	&= {\Tr} \Big[(\delta_{x_1})^{\otimes n} \otimes \cdots \otimes (\delta_{x_m})^{\otimes n} \int |(h^{\beta/2}u)^{\otimes (mn)} \rangle \langle (h^{\beta/2} u)^{\otimes (mn)}| d\mu_{0}(u) (\delta_{x_1})^{\otimes n} \otimes \cdots \otimes (\delta_{x_m})^{\otimes n}\Big]  \\
	&\leq (mn)!{\Tr} \Big[(\delta_{x_1})^{\otimes n} \otimes \cdots \otimes (\delta_{x_m})^{\otimes n} (h^{\beta-1})^{\otimes (mn)} (\delta_{x_1})^{\otimes n} \otimes \cdots \otimes (\delta_{x_m})^{\otimes n}\Big]  \\
	&=(mn)! {\Tr} \Big[(\delta_{x_1})^{\otimes n} \otimes \cdots \otimes (\delta_{x_m})^{\otimes n} (h^{\beta-1}\otimes \cdots \otimes h^{\beta-1})^{\otimes n} (\delta_{x_1})^{\otimes n} \otimes \cdots \otimes (\delta_{x_m})^{\otimes n}\Big]  \\
	&=(mn)!{\Tr} \Big[\Big((\delta_{x_1}\otimes \cdots \otimes \delta_{x_m}) (h^{\beta-1}\otimes \cdots \otimes h^{\beta-1}) (\delta_{x_1} \otimes \cdots \otimes \delta_{x_m})\Big)^{\otimes n}\Big] \\
	&\leq (mn)!\Big( {\Tr} \Big[(\delta_{x_1}\otimes \cdots \otimes \delta_{x_m}) (h^{\beta-1}\otimes \cdots \otimes h^{\beta-1}) (\delta_{x_1} \otimes \cdots \otimes \delta_{x_m})\Big]\Big)^n \\
	&\leq (mn)! \left( h^{\beta-1}(x_1,x_1)\cdots h^{\beta-1}(x_m,x_m)\right)^n,
	\end{align*}
	where $p=2n$. Thus we get
	\[
	\int \|u\|^{2k}_{\Wc^{\beta,p}} d\mu_{0}(u)\leq k! B^k_{\beta,p},
	\]
	where
	\begin{align} \label{eq:B-beta-p}
	B_{\beta,p}:=\left(\int_{\R} \left(h^{\beta-1}(x,x)\right)^{p/2} dx\right)^{2/p}.
	\end{align}
	Note that $B_{\beta,p}$ is finite thanks to Lemma \ref{lem-Lp} and the assumptions on $p$ and $\beta$.
	
	For $pm<2k<p(m+1)$ with $m\geq 0$ an integer, we use H\"older's inequality to get
	\begin{align*}
	\int \|u\|^{2k}_{\Wc^{\beta, p}} d\mu_0(u) & \int \left(\|u\|^{pm}_{\Wc^{\beta,p}}\right)^\delta \left(\|u\|^{p(m+1)}_{\Wc^{\beta,p}}\right)^{1-\delta} d\mu_0(u) \\
	&\leq \left( \int \|u\|^{pm}_{\Wc^{\beta,p}} d\mu_0(u)\right)^\delta \left( \int \|u\|^{p(m+1)}_{\Wc^{\beta,p}} d\mu_0(u)\right)^{1-\delta} \\
	&\leq \left( \left(\frac{pm}{2}\right)! B_{\beta,p}^{\frac{pm}{2}} \right)^\delta \left( \left(\frac{p(m+1)}{2}\right)! B_{\beta,p}^{\frac{p(m+1)}{2}}\right)^{1-\delta} \\
	&= \left( \left(\frac{pm}{2}\right)! \right)^\delta \left( \left(\frac{p(m+1)}{2}\right)! \right)^{1-\delta} B_{\beta,p}^k,
	\end{align*}
	where $\delta \in (0,1)$ satisfying $2k=pm\delta + p(m+1)(1-\delta)$ or $\delta = m+1-\frac{2k}{p}$. We claim that
	\begin{align} \label{eq:claim-p}
	\left( \left(\frac{pm}{2}\right)! \right)^\delta \left( \left(\frac{p(m+1)}{2}\right)! \right)^{1-\delta} \leq \left(\frac{p}{2}\right)! k!
	\end{align}
	Assuming this claim for the moment, we get
	\[
	\int \|u\|^{2k}_{\Wc^{\beta,p}} d\mu_0(u) \leq \left(\frac{p}{2}\right)! k! B_{\beta,p}^k
	\]
	hence
	\[
	\int e^{c\|u\|^2_{\Wc^{\beta,p}}} d\mu_0(u) \leq \left(\frac{p}{2}\right)! \sum_{k\geq 0} (cB_{\beta,p})^k \leq C
	\]
	provided that $c>0$ is chosen such that $cB_{\beta,p}<1$. This proves \eqref{eq:expo-p}. 
	
	It remains to prove the claim \eqref{eq:claim-p}. We write $2k=pm+2l$ with $1\leq l\leq n-1$. In particular, we have $\delta = 1-\frac{l}{n}$ and $1-\delta =\frac{l}{n}$. We also have
	\begin{align*}
	\left( \left(\frac{pm}{2}\right)! \right)^\delta \left( \left(\frac{p(m+1)}{2}\right)! \right)^{1-\delta} &= \left((k-l)!\right)^\delta \left((k-l+n)!\right)^{1-\delta} \\
	&= \left(k! \frac{1}{k\cdots (k-l+1)}\right)^{1-\frac{l}{n}} \left(k! (k+1)\cdots (k-l+n)\right)^{\frac{l}{n}} \\
	&=k! \left( \frac{(k+1)^l \cdots (k-l+n)^l}{k^{n-l} \cdots (k-l+1)^{n-l}}\right)^{\frac{1}{n}} \\
	&=k! \left(\underbrace{\frac{(k+1)\cdots (k-l+n)}{k^{n-l}} \cdots \frac{(k+1)\cdots (k-l+n)}{(k-l+1)^{n-l}}}_{l \text{ times}} \right)^{\frac{1}{n}}.
	\end{align*}
	We observe that each factor inside the bracket is of the form
	\[
	\frac{(a+j)\cdots(a+j+n-l-1)}{a^{n-l}} 
	\]
	with $j=1,\cdots, l$ and $a\geq 1$. We can bound this term as
	\[
	\left(1+\frac{j}{a}\right)\cdots \left(1+\frac{j+n-l-1}{a}\right) \leq (1+j) \cdots (j+n-l) \leq n!
	\]
	for all $j=1,\cdots, l$. As a result, we obtain
	\[
	\left( \left(\frac{pm}{2}\right)! \right)^\delta \left( \left(\frac{p(m+1)}{2}\right)! \right)^{1-\delta} \leq k! (n!)^{\frac{l}{n}} \leq n!  k!
	\]
	This proves~\eqref{eq:claim-p}.
\end{proof}

Given the above estimates, the construction of the defocusing Gibbs measure is straightforward. 
	
	\begin{proof}[Proof of~Proposition~\ref{pro:gc}, Item (1)]
	It suffices to prove that 
	\[
	Z =\int e^{-\frac{1}{2}\|u\|^4_{L^4}} d\mu_0(u) \in (0,\infty).
	\] 
	Since $\mu_0$ is a probability measure, we obviously have $Z \leq 1$. As in the proof of Lemma \ref{lem-expo-p}, we have for any $s>1$,
	\[
	\int \|u\|^4_{L^4} d\mu_0(u) \leq 2! B_{0,4}^2<\infty,
	\]
	where $B_{0,4}$ is as in \eqref{eq:B-beta-p}. By Jensen's inequality
	\[
	\int e^{-\frac{1}{2} \|u\|^4_{L^4}} d\mu_0(u) \geq \exp \left(-\frac{1}{2} \mathlarger{\int}\|u\|^4_{L^4} d\mu_0(u) \right),
	\]
	we infer that $Z>0$.
 	\end{proof}

\subsection{Focusing measure}

We next tackle the more subtle definition of the focusing Gibbs measure. For $u \in \Hc^\theta$ and $\Lambda \geq \lambda_1$, we denote
\begin{align} \label{eq:P-Lamb}
P_{\leq \Lambda} u = \sum_{\lambda_j \leq \Lambda} \alpha_j u_j, \quad P_{>\Lambda} u =\sum_{\lambda_j>\Lambda} \alpha_j u_j
\end{align}
the projections on the low and high frequencies respectively.
	
	\begin{lemma}[\textbf{Decay estimates for the $L^4$ norm}] \label{lem-L4-Lambda}\mbox{}\\
		Let $s>1$, $V$ satisfy Assumption \ref{assu-V}, and $0\leq \rho <\frac{s-1}{2s}$. Then there exist $C,c>0$ such that for $\theta<\frac{1}{2}-\frac{1}{s}$ and all $\Lambda, R>0$,
		\begin{align} \label{eq:boun-L4-Lamb}
		\mu_0\left(u \in \Hc^\theta : \|P_{>\Lambda}u\|_{L^4}>R\right) \leq C e^{-c\Lambda^\rho R^2}.
		\end{align}
	\end{lemma}
	
	\begin{proof}
		Let $t>0$ be a positive constant to be chosen later. We estimate
		\begin{align*}
		\mu_0\left(\|P_{> \Lambda}u\|_{L^4}>R\right) &\leq e^{-tR^2} \int e^{t\|P_{>\Lambda}u\|_{L^4}^2} d\mu_0(u) \\
		&=e^{-tR^2} \sum_{k\geq 0} \frac{t^k}{k!} \int \|P_{>\Lambda}u\|^{2k}_{L^4} d\mu_0(u).
		\end{align*}
		By the same argument as in the proof of Lemma \ref{lem-expo-p}, we have
		\[
		\int \|P_{>\Lambda} u\|^{2k}_{L^4} d\mu_0(u) \leq 2! k! B_{\Lambda,0,4}^k, \quad \forall k\geq 0,
		\]
		where
		\begin{align}\label{eq:B-Lamb}
		B_{\Lambda,0,4}:= \left( \int_{\R} \left( (P_{>\Lambda} h)^{-1}(x,x)\right)^2 dx\right)^{1/2}
		\end{align}
		with 
		\[
		(P_{>\Lambda} h)^{-1}(x,x) = \sum_{\lambda_j>\Lambda} \lambda_j^{-1} |u_j(x)|^2.
		\]
		It follows that
		\[
		\int e^{t\|P_{>\Lambda} u\|^2_{L^4}} d\mu_0(u) \leq 2! \sum_{k\geq 0} (tB_{\Lambda,0,4})^k.
		\]
		For $0\leq \rho <\frac{s-1}{2s}$, we have
		\begin{align*}
		(P_{>\Lambda} h)^{-1}(x,x) &=\sum_{\lambda_j >\Lambda} \lambda_j^{-1} |u_j(x)|^2  \\
		&\leq \Lambda^{-\rho} \sum_{\lambda_j >\Lambda} \lambda_j^{\rho-1} |u_j(x)|^2 \\
		&\leq \Lambda^{-\rho} h^{\rho-1}(x,x).
		\end{align*}
		Thanks to Lemma \ref{lem-Lp}, we see that $x\mapsto h^{\rho-1}(x,x) \in L^2(\R)$ for all $0\leq \rho <\frac{s-1}{2s}$, hence $B_{\Lambda,0,4} \leq C \Lambda^{-\rho}$ for some constant $C>0$. In particular, we have
		\[
		\mu_0\left(\|P_{>\Lambda} u\|_{L^4} >R\right) \leq e^{-tR^2} 2! \sum_{k\geq 0} ( Ct \Lambda^{-\rho})^k.
		\]
		Taking $t=\nu \Lambda^\rho$ with $\nu>0$ sufficiently small so that $Ct\Lambda^{-\rho} = C\nu<1$, we obtain \eqref{eq:boun-L4-Lamb}.
	\end{proof}

	When $1<s\leq 2$, the Gaussian measure $\mu_0$ lives over negative Sobolev spaces $\Hc^\theta$ with $\theta<\frac{1}{2}-\frac{1}{s}$, hence the mass is infinite $\mu_0$-almost surely. In this situation, we consider the renormalized mass as follows.
	
	\begin{lemma}[\textbf{Renormalized mass}]\label{lem-reno-mass}\mbox{}\\
		Let $1<s\leq 2$ and $V$ satisfy Assumption \ref{assu-V}. For every $\Lambda \geq \lambda_1$, we define the truncated renormalized mass
		\[
		\Mcal_{\leq \Lambda}(u) := \|P_{\leq \Lambda} u\|^2_{L^2}- \int \|P_{\leq \Lambda} u\|^2_{L^2} d\mu_0(u).
		\]
		Then $\{\Mcal_{\leq \Lambda}\}_{\Lambda \geq \lambda_1}$ converges strongly to a limit in $L^2(d\mu_0)$, i.e., 
		\[
		\Mcal(u) :=\lim_{\Lambda\rightarrow \infty} \Mcal_{\leq \Lambda}(u) \text{ in } L^2(d\mu_0).
		\]
		In particular, we have
		\[
		\int (\Mcal(u))^2 d\mu_0(u) = {\Tr}[h^{-2}]<\infty.
		\]
	\end{lemma}
	
	\begin{proof}
	This is the same argument as in~\cite[Lemma~5.2]{LewNamRou-20}. We reproduce it for the reader's convenience. For $\Lambda\geq \lambda_1$, we have
		\[
		\|P_{\leq \Lambda} u\|^2_{L^2} = \sum_{\lambda_j \leq \Lambda} |\alpha_j|^2
		\] 
		and
		\begin{align*}
		\int \|P_{\leq \Lambda} u\|^2_{L^2} d\mu_0(u) &= \sum_{\lambda_j \leq \Lambda} \int |\alpha_j|^2 d\mu_0(u) \\
		&= \sum_{\lambda_j \leq \Lambda} \left(\int_{\C} |\alpha_j|^2 \frac{\lambda_j}{\pi}e^{-\lambda_j |\alpha_j|^2} d\alpha_j \right) \left(\prod_{\lambda_k\ne \lambda_j} \int_{\C} \frac{\lambda_k}{\pi}e^{-\lambda_k |\alpha_k|^2} d\alpha_k\right) \\
		&=\sum_{\lambda_j \leq \Lambda} \frac{1}{\lambda_j}.
		\end{align*}
		Thus we get
		\begin{align} \label{eq:M-Lamb-low}
		\Mcal_{\leq \Lambda} (u) = \sum_{\lambda_j \leq \Lambda} |\alpha_j|^2-\lambda_j^{-1}.
		\end{align}
		Now for $\Theta\geq \Lambda$, we compute
		\begin{align*}
		\int |\Mcal_{\leq \Theta} (u) &- \Mcal_{\leq \Lambda}(u)|^2 d\mu_0(u) \\
		&= \int \Big(\sum_{\Lambda<\lambda_j \leq \Theta} |\alpha_j|^2-\lambda_j^{-1}\Big)^2 d\mu_0(u) \\
		&= \sum_{\Lambda<\lambda_j,\lambda_k \leq \Theta} \int \Big(|\alpha_j|^2-\lambda_j^{-1}\Big) \Big(|\alpha_k|^2-\lambda_k^{-1}\Big) d\mu_0(u) \\
		&= \sum_{\Lambda<\lambda_j,\lambda_k \leq \Theta} \int \Big(|\alpha_j|^2 |\alpha_k|^2 -\lambda_j^{-1}|\alpha_k|^2 -\lambda_k^{-1} |\alpha_j|^2 +\lambda_j^{-1}\lambda_k^{-1} \Big) d\mu_0(u) \\
		&= \sum_{\Lambda<\lambda_j,\lambda_k \leq \Theta} \int \Big(|\alpha_j|^2|\alpha_k|^2 -\lambda_j^{-1} \lambda_k^{-1} \Big)d\mu_0(u) \\
		&= \sum_{\Lambda<\lambda_j \leq \Theta} \int \Big(|\alpha_j|^4 -\lambda_j^{-2} \Big)d\mu_0(u) + \sum_{\Lambda<\lambda_j,\lambda_k \leq \Theta \atop \lambda_j \ne \lambda_k} \int \Big(|\alpha_j|^2 |\alpha_k|^2-\lambda_j^{-1} \lambda_k^{-1} \Big)d\mu_0(u) \\
		&=\sum_{\Lambda<\lambda_j \leq \Theta} \lambda_j^{-2} \rightarrow 0 \text{ as } \Lambda, \Theta \rightarrow \infty
		\end{align*}
		due to ${\Tr}[h^{-2}]=\sum_{j\geq 1} \lambda_j^{-2}<\infty$ since $2>\frac{1}{2}+\frac{1}{s}$ for all $1<s\leq 2$ (see Lemma \ref{lem-trac-h-p}). This shows that $\{\Mcal_{\leq \Lambda}(u)\}_{\Lambda\geq \lambda_1}$ is a Cauchy sequence in $L^2(d\mu_0)$. Thus there exists $\Mcal(u) \in L^2(d\mu_0)$ such that $\Mcal_{\leq \Lambda}(u) \rightarrow \Mcal(u)$ strongly in $L^2(d\mu_0)$ as $\Lambda \to \infty$. Moreover, from the above computation, we have
		\[
		\int (\Mcal(u))^2 d\mu_0(u)= \lim_{\Lambda \rightarrow \infty} \int (\Mcal_{\leq \Lambda}(u))^2 d\mu_0(u) = \lim_{\Lambda\rightarrow \infty} \sum_{\lambda_j \leq \Lambda} \lambda_j^{-2}={\Tr}[h^{-2}].
		\]
	\end{proof}
	
	We have the following observation regarding the renormalized mass.
	\begin{lemma}[\textbf{Decay estimates for the renormalized mass}] \label{lem-Lambda-sub}\mbox{}\\
		Let $1<s\leq 2$, $V$ satisfy Assumption \ref{assu-V}, and $0\leq\gamma<\frac{3s-2}{4s}$. Then there exist $C,c>0$ such that for $\theta<\frac{1}{2}-\frac{1}{s}$ and all $\Lambda, R>0$,
		\begin{align} \label{eq:boun-M-Lamb-hi}
		\mu_0(u \in \Hc^\theta : |\Mcal_{>\Lambda}(u)|>R) \leq Ce^{-c \Lambda^\gamma R},
		\end{align}
		where
		\begin{align} \label{eq:M-Lamb-hi}
		\Mcal_{>\Lambda}(u) :=  \sum_{\lambda_j>\Lambda} |\alpha_j|^2 -\lambda_j^{-1}.
		\end{align}
	\end{lemma}
	
	\begin{proof}
		We have
		\[
		\mu_0(|\Mcal_{>\Lambda}(u)|>R) \leq \mu_0(\Mcal_{>\Lambda}(u)>R) + \mu_0(\Mcal_{>\Lambda}(u)<-R)=:(\text{I}) + (\text{II}).
		\]
		For (I), we estimate for $0<t<\frac{\Lambda}{2}$ to be chosen later,
		\begin{align*}
		\mu_0(\Mcal_{>\Lambda}(u)>R) &\leq e^{-tR} \int e^{t\Mcal_{> \Lambda}(u)} d\mu_0(u) \\
		&=e^{-tR}\int e^{t\left(\sum_{\lambda_j>\Lambda} |\alpha_j|^2-\lambda_j^{-1}\right)} d\mu_0(u) \\
		&=e^{-tR}\int \prod_{\lambda_j>\Lambda} e^{-t\lambda_j^{-1}} e^{t|\alpha_j|^2}d\mu_0(u) \\
		&=e^{-tR} \prod_{\lambda_j>\Lambda} e^{-t\lambda_j^{-1}} \left(\int_{\C} \frac{\lambda_j}{\pi} e^{-\lambda_j(1-t\lambda_j^{-1})|\alpha_j|^2}d\alpha_j\right) \left(\prod_{k\geq 1 \atop \lambda_k \ne \lambda_j} \int_{\C} \frac{\lambda_k}{\pi} e^{-\lambda_k|\alpha_k|^2}d\alpha_k\right) \\
		&=e^{-tR} \prod_{\lambda_j>\Lambda} e^{-t\lambda_j^{-1}} \frac{1}{1-t\lambda_j^{-1}}.
		\end{align*}
		Note that for $0\leq x\leq \frac{1}{2}$, we have $\frac{1}{1-x} \leq Ce^{x+x^2}$ for some constant $C>0$. Applying this to  $x=t\lambda_j^{-1} \leq t \Lambda^{-1}\leq \frac{1}{2}$, we get
		\[
		\mu_0(\Mcal_{>\Lambda}(u)>R)\leq Ce^{-tR}\prod_{\lambda_j>\Lambda} e^{t^2\lambda_j^{-2}} = Ce^{-tR} e^{t^2\sum_{\lambda_j>\Lambda} \lambda_j^{-2}}.
		\]
		For $0\leq\gamma<\frac{3s-2}{4s}$, we observe that
		\begin{align*}
		\sum_{\lambda_j>\Lambda} \lambda_j^{-2} \leq \Lambda^{-2\gamma} \sum_{\lambda_j>\Lambda} \lambda_j^{-2+2\gamma} \leq \Lambda^{-2\gamma} {\Tr}[h^{-2+2\gamma}].
		\end{align*}
		Here ${\Tr}[h^{-2+2\gamma}]<\infty$ since $2-2\gamma>\frac{1}{2}+\frac{1}{s}$ for all $1<s\leq 2$ (see Lemma \ref{lem-trac-h-p}). Thus we get
		\[
		\mu_0(\Mcal_{>\Lambda}(u)>R)\leq Ce^{-tR+t^2\Lambda^{-2\gamma}{\Tr}[h^{-2+2\gamma}]}.
		\]
		Taking $t=\nu \Lambda^\gamma$ with $\nu>0$ small, we obtain
		\[
		\mu_0(\Mcal_{>\Lambda}(u)>R)\leq Ce^{-c\Lambda^\gamma R}.
		\]
		For (II), we have for $0<t<\frac{\Lambda}{2}$, 
		\begin{align*}
		\mu_0(\Mcal_{>\Lambda}(u)<-R) \leq e^{-tR} \int e^{-t \Mcal_{>\Lambda}(u)} d\mu_0(u) = e^{-tR} \prod_{\lambda_j>\Lambda} e^{t\lambda_j^{-1}} \frac{1}{1+t\lambda_j^{-1}}.
		\end{align*}
		Note that for $0\leq x \leq \frac{1}{2}$, we have $\frac{1}{1+x} \leq C e^{-x+x^2}$ for some constant $C>0$. Estimating as in (I), we prove as well that
		\[
		\mu_0(\Mcal_{>\Lambda}(u)<-R) \leq Ce^{-c\Lambda^\gamma R}.
		\]
		Collecting both terms, we prove \eqref{eq:boun-M-Lamb-hi}.
	\end{proof}
	
	Our construction of the focusing Gibbs measure uses\footnote{This is the origin of our technical restriction $s>8/5$.} the following interpolation inequality, due to Br\'ezis and Mironescu~\cite{BreMir-19}. It generalizes well-known results~\cite{Nirenberg-59,Nirenberg-66} to fractional Sobolev spaces, which is handy in our case, for our measures can only afford at best $1/2$ derivative. 
	
	\begin{lemma}[\textbf{Fractional Gagliardo--Nirenberg inequality}]\label{lem-sobo-ineq}\mbox{}\\
		Let $1<s\leq 2$, $V$ satisfy Assumption \ref{assu-V}, $0<\beta <\frac{1}{2}$, $1\leq p \leq \infty$, and 
		$$0 < \delta=\frac{1}{2+4\beta-\frac{4}{p}} < 1. $$
		Then there exists $C>0$ such that
		\begin{align} \label{eq:sobo-ineq}
		\|u\|_{L^4} \leq C \|u\|^{1-\delta}_{L^2} \|u\|^\delta_{\Wc^{\beta,p}}.
		\end{align}
	\end{lemma}
	
		We have the inequality in usual Sobolev spaces $W^{\beta,p}(\R)$ (see\cite{BreMir-19})
		\[
		\|u\|_{L^4(\R)}\leq C\|u\|^{1-\delta}_{L^2(\R)} \|u\|^\delta_{W^{\beta,p}(\R)}
		\]
		provided that $1\leq p\leq \infty$ and $\delta \in (0,1)$ satisfy
		\[
		\frac{1}{4}=\frac{1-\delta}{2} + \frac{\delta}{p} -\delta \beta
		\]
		or
		\[
		\delta = \frac{1}{2+4\beta-\frac{4}{p}}.
		\]
		Hence \eqref{eq:sobo-ineq} follows from the norm equivalence \eqref{equi-norm}.
		
	We are now able to define the focusing Gibbs measure of Proposition~\ref{pro:gc}.
	
	\medskip
	
	\begin{proof}[Proof of Proposition~\ref{pro:gc}, Items (2) and (3)] ~
		
		\noindent \underline{\it Item (2).} Let $s>2$. We will prove that 
		\[
		Z=\int e^{\frac{1}{2}\|u\|^4_{L^4}} \mathds{1}_{\left\{\|u\|^2_{L^2} \leq m\right\}}d\mu_0(u)  \in (0,\infty).
		\]
		We first have
		\[
		Z \geq \int \mathds{1}_{\left\{\|u\|^2_{L^2} \leq m\right\}} d\mu_0(u) \geq \frac{1}{m} \int \|u\|^2_{L^2} d\mu_0(u) = \frac{1}{m} {\Tr}[h^{-1}]>0.
		\]
		To see that $Z <\infty$, we use the layer cake representation to write
		\[
		Z = \int_0^\infty \mu_0\left(e^{\frac{1}{2}\|u\|^4_{L^4}} >\lambda, \|u\|^2_{L^2} \leq m\right) d\lambda = \int_0^\infty \mu_0\left(\|u\|_{L^4}>(2\log \lambda)^{1/4}, \|u\|^2_{L^2}\leq m \right) d\lambda.
		\]
		Denote 
		\begin{align} \label{eq:Lamb-0}
			\Lambda_0:= \left(\log \lambda\right)^l
		\end{align}
		for some $l>0$ to be determined later. By the triangle inequality, we have
		\begin{align*}
		\mu_0\left(\|u\|_{L^4}>(2\log \lambda)^{1/4}, \|u\|^2_{L^2}\leq m\right) &\leq \mu_0\left(\|P_{\leq \Lambda_0}u\|_{L^4}>\frac{1}{2}(2\log \lambda)^{1/4}, \|u\|^2_{L^2}\leq m\right) \\
		&\quad +\mu_0\left(\|P_{>\Lambda_0}u\|_{L^4}>\frac{1}{2}(2\log \lambda)^{1/4}, \|u\|^2_{L^2}\leq m\right).
		\end{align*}
		By Lemma \ref{lem-L4-Lambda}, we have
		\begin{align*}
		\mu_0\left(\|P_{> \Lambda_0}u\|_{L^4}>\frac{1}{2}(2\log \lambda)^{1/4}, \|u\|^2_{L^2}\leq m\right) &\leq \mu_0\left(\|P_{> \Lambda_0}u\|_{L^4}>\frac{1}{2}(2\log \lambda)^{1/4}\right) \\
		&\leq Ce^{-c \Lambda_0^\rho (\log \lambda)^{1/2}} = Ce^{-c(\log \lambda)^{\rho l +\frac{1}{2}}}
		\end{align*}
		for all $0\leq \rho <\frac{s-1}{2s}$. On the other hand, by the fractional Gagliardo--Nirenberg inequality \eqref{eq:sobo-ineq}, we have
		\begin{align} \label{sobo-ineq-appl}
		\|P_{\leq \Lambda_0} u\|_{L^4} \leq C \|P_{\leq \Lambda_0} u\|_{L^2}^{1-\delta} \|P_{\leq \Lambda_0} u\|^\delta_{\mathcal{W}^{\beta,p}}
		\end{align}
		with $0<\beta <\frac{1}{2}$, $1\leq p \leq \infty$, and $\delta \in (0,1)$ satisfying
		\begin{align} \label{delta}
		\delta = \frac{1}{2-\frac{4}{p}+4\beta}.
		\end{align}
		Thus for $u$ satisfying $\|P_{\leq \Lambda_0} u\|_{L^4}>\frac{1}{2}(2\log \lambda)^{1/4}$ and $\|u\|^2_{L^2}\leq m$, we deduce from \eqref{sobo-ineq-appl} that
		\begin{align*}
			\frac{1}{2}(2\log \lambda)^{1/4} <\|P_{\leq \Lambda_0} u\|_{L^4} \leq C m^{(1-\delta)/2} \|P_{\leq \Lambda_0} u\|^\delta_{\Wc^{\beta,p}}
		\end{align*}
		or
		\[
		\|P_{\leq \Lambda_0} u\|_{\Wc^{\beta,p}} >C(m) (\log \lambda)^{\frac{1}{4\delta}}.
		\]
		Thus we get
		\begin{align*}
		\mu_0\left(\|P_{\leq \Lambda_0}u\|_{L^4}>\frac{1}{2}(2\log \lambda)^{1/4}, \|u\|^2_{L^2}\leq m\right) \leq \mu_0\left(\|P_{\leq \Lambda_0} u\|_{\Wc^{\beta,p}} >C(m) (\log \lambda)^{\frac{1}{4\delta}}\right).
		\end{align*}
		For $0<\beta<\frac{1}{2}$, we can find an even integer $p$ sufficiently large so that $p>\frac{4}{s(1-2\beta)}$. Thus, by Lemma \ref{lem-expo-p}, we have
		\[
		\mu_0\left(\|P_{\leq \Lambda_0}u\|_{L^4}>\frac{1}{2}(2\log \lambda)^{1/4}, \|u\|^2_{L^2}\leq m\right) \leq Ce^{-c(\log \lambda)^{\frac{1}{2\delta}}}.
		\]
		Using the fact that for any $L>0,\vareps >0$, there exists $C>0$ such that 
		$$e^{-c(\log \lambda)^{1+\vareps}} \leq C \lambda^{-L},$$
		the partition function $Z$ is finite if we have
		\[
		\rho l +\frac{1}{2}>1, \quad \frac{1}{2\delta}>1
		\]
		with $0\leq \rho <\frac{s-1}{2s}$ and $\delta$ as in \eqref{delta}. The above conditions are fulfilled by taking 
		$$l=\frac{s}{s-1}+\eta, \quad \rho =\frac{s-1}{2s}-\eta, \quad \beta=\frac{1}{2}-\eta, \quad p =\eta^{-1}$$
		for some suitably small number $0 < \eta \ll 1$.
		
		\medskip
		
		\noindent\underline{\it Item (3).} Let $\frac{8}{5}<s \leq 2$. We have
		\[
		Z \geq \int \mathds{1}_{\left\{|\Mcal(u)| \leq m\right\}} d\mu_0(u) \geq \frac{1}{m^2} \int (\Mcal(u))^2 d\mu_0(u) =\frac{1}{m^2} {\Tr}[h^{-2}]>0.
		\]
		It remains to show that $Z<\infty$. We have
		\begin{align*}
		Z &= \int e^{\frac{1}{2}\|u\|^4_{L^4}} \mathds{1}_{\left\{|\Mcal(u)| \leq m\right\}} d\mu_0(u) \\
		&=\int_0^\infty \mu_0 \left(e^{\frac{1}{2}\|u\|^4_{L^4}} >\lambda, |\Mcal(u)| \leq m\right) d\lambda \\
		&=\int_0^\infty \mu_0\left(\|u\|_{L^4}>\left(2 \log \lambda\right)^{1/4}, |\Mcal(u)| \leq m\right) d\lambda.
		\end{align*}
		For $\Lambda_0$ as in \eqref{eq:Lamb-0}, the triangle inequality gives
		\begin{align*}
		\mu_0\left(\|u\|_{L^4}>\left(2 \log \lambda\right)^{1/4}, |\Mcal(u)| \leq m\right) &\leq \mu_0 \left(\|P_{>\Lambda_0} u\|_{L^4} >\frac{1}{2}(2\log \lambda)^{1/4}, |\Mcal(u)| \leq m\right) \\
		&\quad + \mu_0 \left(\|P_{\leq \Lambda_0} u\|_{L^4} >\frac{1}{2}(2\log \lambda)^{1/4}, |\Mcal(u)| \leq m\right) \\
		&=: (\text{I}) + (\text{II}).
		\end{align*}
		By Lemma \ref{lem-L4-Lambda}, we have for any $0\leq \rho <\frac{s-1}{2s}$,
		\begin{align} \label{term-1}
		(\text{I}) \leq \mu_0\left(\|P_{>\Lambda_0} u\|_{L^4} >\frac{1}{2}(2\log \lambda)^{1/4}\right) \leq C e^{-c \Lambda_0^\rho (\log \lambda)^{1/2}} = Ce^{-c(\log \lambda)^{\rho l +1/2}}.
		\end{align}
		For (II), we denote
		\[
		A_{\Lambda_0, \lambda} := \left\{\|P_{\leq \Lambda_0} u\|_{L^4} >\frac{1}{2}(2\log \lambda)^{1/4}, |\Mcal(u)| \leq m \right\}.
		\]
		We estimate 
		\begin{align}
		\|P_{\leq \Lambda_0} u\|^2_{L^2} &= \sum_{\lambda_j\leq \Lambda_0} |\alpha_j|^2 \nonumber\\
		&=\sum_{\lambda_j\leq \Lambda_0} |\alpha_j|^2 -\lambda_j^{-1} + \sum_{\lambda_j\leq \Lambda_0} \lambda_j^{-1} \nonumber\\
		&=\Mcal_{\leq \Lambda_0}(u) +\sum_{\lambda_j\leq \Lambda_0} \lambda_j^{-1} \nonumber\\
		&=\Mcal(u) - \Mcal_{>\Lambda_0}(u) +\sum_{\lambda_j\leq \Lambda_0} \lambda_j^{-1}, \label{eq:low-fre-1}
		\end{align}
		where $\Mcal_{\leq \Lambda}(u)$ and $\Mcal_{>\Lambda}(u)$ are defined in \eqref{eq:M-Lamb-low} and \eqref{eq:M-Lamb-hi} respectively. Observe that 
		\begin{align} \label{eq:low-fre-2}
		\sum_{\lambda_j\leq \Lambda_0} \lambda_j^{-1} \leq \Lambda_0^\nu\sum_{\lambda_j\leq \Lambda_0} \lambda_j^{-1-\nu} = \Lambda_0^\nu {\Tr}[h^{-1-\nu}]
		\end{align}
		with ${\Tr}[h^{-1-\nu}]<\infty$ provided that $\nu>\frac{1}{s}-\frac{1}{2}$. Define the set 
		$$\Omega_{\Lambda_0,\lambda}:= \left\{ |\Mcal_{>\Lambda_0}(u)| \leq \Lambda_0^\nu\right\}.$$
		We have
		\[
		(\text{II})=\mu_0(A_{\Lambda_0,\lambda}) \leq \mu_0\left(A_{\Lambda_0,\lambda} \cap \Omega_{\Lambda_0,\lambda}^c\right) + \mu_0\left(A_{\Lambda_0,\lambda} \cap \Omega_{\Lambda_0,\lambda}\right) =:(\text{II}_1) + (\text{II}_2).
		\]
		By Lemma \ref{lem-Lambda-sub}, we have for any $0\leq \gamma <\frac{3s-2}{4s}$,
		\begin{align} \label{term-2}
		(\text{II}_1) \leq \mu_0(\Omega_{\Lambda_0,\lambda}^c) \leq C e^{-c \Lambda_0^{\gamma+\nu}} = Ce^{-c(\log \lambda)^{(\gamma+\nu)l}}.
		\end{align}
		For $u \in A_{\Lambda_0,\lambda} \cap \Omega_{\Lambda_0, \lambda}$, we deduce from \eqref{eq:low-fre-1} and \eqref{eq:low-fre-2} that
		\[
		\|P_{\leq \Lambda_0} u\|_{L^2} \leq \sqrt{m+ C\Lambda_0^\nu} \leq C \Lambda_0^{\nu/2}.
		\] 
		Moreover, using \eqref{sobo-ineq-appl}, we find that for any $u \in A_{\Lambda_0,\lambda} \cap \Omega_{\Lambda_0, \lambda}$,
		\[
		\frac{1}{2}(2\log \lambda)^{1/4} < \|P_{\leq \Lambda_0} u\|_{L^4} \leq C\|P_{\leq \Lambda_0} u\|^{1-\delta}_{L^2} \|P_{\leq \Lambda_0} u\|^{\delta}_{\Wc^{\beta,p}} \leq C \Lambda_0^{\nu(1-\delta)/2} \|P_{\leq \Lambda_0} u\|^{\delta}_{\Wc^{\beta,p}}
		\]
		with $\delta$ as in \eqref{delta}, hence
		\[
		\|P_{\leq \Lambda_0} u\|_{\Wc^{\beta,p}} > C\left(\frac{(\log \lambda)^{1/4}}{\Lambda_0^{\nu(1-\delta)/2}} \right)^{\frac{1}{\delta}} \sim (\log \lambda)^{\frac{1}{4\delta} -\frac{l\nu}{2}\left(\frac{1}{\delta}-1\right)} \sim (\log \lambda)^{\frac{1}{2}+\beta-\frac{1}{p} -\frac{l\nu}{2}\left(1+4\beta-\frac{4}{p}\right)}.
		\]
		Thus, by Lemma \ref{lem-expo-p}, we get
		\begin{align} \label{term-3}
		(\text{II}_2) &\leq \mu_0\left(\|P_{\leq \Lambda_0} u\|_{\Wc^{\beta, p}}>C(\log \lambda)^{\frac{1}{2}+\beta-\frac{1}{p} -\frac{l\nu}{2}\left(1+4\beta-\frac{4}{p}\right)}\right) \nonumber \\
		&\leq C e^{-c (\log \lambda)^{1+2\beta-\frac{2}{p} -l\nu\left(1+4\beta-\frac{4}{p}\right)}}.
		\end{align}
		Collecting \eqref{term-1}, \eqref{term-2}, and \eqref{term-3}, we obtain
		\begin{align*}
		\mu_0\left(\|u\|_{L^4}>\left(2 \log \lambda\right)^{1/4}, |\Mcal(u)| \leq m\right) &\leq Ce^{-c(\log \lambda)^{\rho l +1/2}} + Ce^{-c(\log \lambda)^{(\gamma+\nu)l}} \\
		&\quad+ C e^{-c (\log \lambda)^{1+2\beta-\frac{2}{p} -l\nu\left(1+4\beta-\frac{4}{p}\right)}}
		\end{align*}
		for any $0\leq\rho <\frac{s-1}{2s}$, $0\leq \gamma <\frac{3s-2}{4s}$, $\nu>\frac{2-s}{2s}$, $0<\beta<\frac{1}{2}$, and $p$ a large even integer satisfying $p>\frac{4}{s(1-2\beta)}$. To make $Z<\infty$, we will choose suitable values of $l, \rho, \gamma, \nu, \beta$, and $p$ so that the following conditions are satisfied:
		\[
		\rho l +1/2 >1, \quad  (\gamma+\nu)l>1, \quad 1+2\beta-\frac{2}{p} -l\nu\left(1+4\beta-\frac{4}{p}\right)>1.
		\]		
		By taking a suitably small $0< \eta \ll 1$
		$$\rho=\frac{s-1}{2s}-\eta, \quad \gamma=\frac{3s-2}{4s}-\eta, \quad \nu=\frac{2-s}{2s}+\eta,$$ 
		the first two conditions imply 
		\[
		l>\max\left\{\frac{s}{s-1}+,\frac{4s}{s+2}\right\}.
		\] 
		Taking then 
		$$\beta=\frac{1}{2}-\eta,\quad p=\eta^{-1},$$
		the last condition yields
		\[
		l<\frac{2s}{3(2-s)}.
		\]
		Since $\frac{4s}{s+2}\leq \frac{s}{s-1}<\frac{2s}{3(2-s)}$ for $\frac{8}{5}<s\leq 2$, we can choose $l$ satisfying the above conditions. This shows that for any $L>0$,
		\[
		\mu_0\left(\|u\|_{L^4}>\left(2 \log \lambda\right)^{1/4}, |\Mcal(u)| \leq m\right) \leq C_L \lambda^{-L}
		\]
		which ensures $Z<\infty$. The proof is complete.
	\end{proof}

	\begin{remark}[Higher non-linearities] \label{rem-Gibbs-general}
		The arguments presented above can be applied to construct the Gibbs measures for \eqref{eq:NLS-general}. More precisely, since the Gaussian measure $\mu_0$ is supported in $L^\kappa(\R)$ with $\max\{2,\frac{4}{s}\}<\kappa<\infty$ (see Lemma \ref{lem-expo-p} with $\beta=0$), one can easily construct the defocusing Gibbs measure 
		\[
		d\mu(u) = \frac{1}{Z} e^{-\frac{2}{\kappa}\|u\|^\kappa_{L^\kappa}} d\mu_0(u)
		\]
		for any $\max\{2,\frac{4}{s}\}<\kappa<\infty$. For the focusing Gibbs measure, when $s>2$, one can construct the measure
		\[
		d\mu(u) = \frac{1}{Z} e^{\frac{2}{\kappa}\|u\|^\kappa_{L^\kappa}} \mathds{1}_{\left\{\int_{\R}|u|^2 \leq m\right\}} d\mu_0(u)
		\]
		for any
		\begin{align}\label{foc-mea-sup}
		2<\kappa<6, \quad m>0.
		\end{align} 
		Moreover, when $s\leq 2$, one can construct the measure
		\[
		d\mu(u) = \frac{1}{Z} e^{\frac{2}{\kappa}\|u\|^\kappa_{L^\kappa}} \mathds{1}_{\left\{|\Mcal(u)| \leq m\right\}} d\mu_0(u)
		\]
		for any $\frac{14}{9}<s \leq 2$ and
		\begin{align} \label{foc-mea-sub} 
		\frac{4}{s}<\kappa<\frac{3s+2+\sqrt{9s^2+4s-28}}{2}, \quad m>0.
		\end{align}  
		To see \eqref{foc-mea-sup} and \eqref{foc-mea-sub}, we follow the same line of reasoning as in the previous construction and take into account the following estimate (which is a consequence of Lemmas \ref{lem-Lp} and \ref{lem-L4-Lambda}): 
		\begin{align} \label{est-L-kappa}
		\mu_0(u\in \Hc^\theta : \|P_{>\Lambda}u\|_{L^\kappa} >R) \leq Ce^{-c\Lambda^\rho R^2}
		\end{align}
		for $\kappa>\max\left\{2,\frac{4}{s}\right\}$ and $0\leq \rho <\frac{\kappa s-4}{2\kappa s}$. More precisely, for $s>2$, we need 
		\[
		\rho l + \frac{2}{\kappa}>1, \quad \frac{2}{\kappa\delta} >1
		\]
		with 
		\begin{align} \label{rho-delta}
		0 \leq \rho <\frac{\kappa s-4}{2\kappa s}, \quad \delta = \frac{\frac{1}{2}-\frac{1}{\kappa}}{\frac{1}{2}-\frac{1}{p}+\beta}.
		\end{align} 
		By taking $l=\frac{2s(\kappa-2)}{\kappa s-4}+$, $\rho =\frac{\kappa s-4}{2\kappa s}-$, $\beta=\frac{1}{2}-$, and $p=\infty-$, the above conditions are reduced to $\frac{4}{\kappa-2}>1$. Together with $\kappa>2$ this gives~\eqref{foc-mea-sup}. When $s\leq 2$, we need
		\[
		\rho l +\frac{2}{\kappa}>1, \quad (\gamma+\nu)l>1, \quad \frac{2}{\kappa \delta} - \nu l \left(\frac{1}{\delta}-1\right)>1
		\]
		with $\rho, \delta$ as in \eqref{rho-delta} and
		\[
		0\leq \gamma<\frac{3s-2}{4s}, \quad \nu>\frac{2-s}{2s}, \quad 0<\beta<\frac{1}{2}, \quad p>\frac{4}{s(1-2\beta)}.
		\]
		We now choose
		\[
		\rho =\frac{\kappa s-4}{2\kappa s}-, \quad \gamma=\frac{3s-2}{4s}-, \quad \nu =\frac{2-s}{2s}+, \quad \beta =\frac{1}{2}-, \quad p=\infty-
		\]
		which yields
		\[
		\max\left\{\frac{2s(\kappa -2)}{\kappa s-4}, \frac{4s}{s+2}\right\} <l <\frac{2s(6-\kappa)}{(2-s)(p+2)}.
		\]
		Such a choice is possible provided that
		\[
		\frac{2s(\kappa-2)}{\kappa s-4}<\frac{2s(6-\kappa)}{(2-s)(\kappa+2)}
		\]
		or
		\[
		s >\frac{\kappa^2-2\kappa+8}{3\kappa-2}
		\]
		which, together with $\kappa>\frac{4}{s}$ (see \eqref{est-L-kappa}) yields~\eqref{foc-mea-sub}.\hfill$\diamond$
	\end{remark}
	
\subsection{Approximating measures}
	
	We next define approximate measures for $\mu$ which are useful in proving the invariance of Gibbs measures under the flow of \eqref{eq:intro NLS}. Following \cite{BurThoTzv-10}, we introduce for $\Lambda \geq \lambda_1$,
	\begin{align} \label{eq:Q-Lamb}
	Q_\Lambda u:= \sum_{j\geq 1} \chi\left(\frac{\lambda_j}{\Lambda}\right) \alpha_j u_j = \chi(h/\Lambda) u,
	\end{align}
	where $\chi \in C^\infty_0(\R)$ satisfies $\supp(\chi)\subset [-1,1]$, $\chi \in [0,1]$, and $\chi=1$ on $[-1/2,1/2]$. It is known that 
	\begin{align} \label{eq:Q-P-Lamb}
	\quad Q_\Lambda P_{\leq \Lambda} = P_{\leq \Lambda} Q_\Lambda =Q_\Lambda
	\end{align}
	and there exists $C>0$ such that for all $\Lambda \geq \lambda_1$,
	\begin{align}\label{eq:est-Q-Lamb}
	\|Q_\Lambda\|_{L^p \to L^p} \leq C, \quad 1\leq p\leq \infty.
	\end{align}
	The latter follows from the $L^p$-boundedness of semi-classical pseudo-differential operators as we briefly recall in~Appendix~\ref{sec:app3}. 

	For the defocusing nonlinearity, we define the approximate measure as
	
	\begin{align} \label{eq:mu-Lamb}
	d\mu_\Lambda(u) := \frac{1}{Z^\Lambda} e^{-\frac{1}{2} \|Q_\Lambda u\|^4_{L^4}} d\mu_0(u) = d\mu^{\leq \Lambda}(u) \otimes d\mu_0^{>\Lambda}(u),
	\end{align}
	where
	\begin{align} \label{eq:mu-0-Lamb-hi}
	d\mu_0^{>\Lambda}(u) = \prod_{\lambda_j>\Lambda} \frac{\lambda_j}{\pi} e^{-\lambda_j|\alpha_j|^2} d\alpha_j
	\end{align}
	and
	\[
	d\mu^{\leq \Lambda}(u) = \frac{1}{Z^{\Lambda}} e^{-\frac{1}{2} \|Q_\Lambda u\|^4_{L^4}} d\mu_0^{\leq \Lambda}(u)
	\]
	with $\mu_0^{\leq \Lambda}$ as in \eqref{eq:mu-0-Lamb} and
	\[
	Z^\Lambda =\int e^{-\frac{1}{2} \|Q_\Lambda u\|^4_{L^4}} d\mu_0(u)=\int e^{-\frac{1}{2} \|Q_\Lambda u\|^4_{L^4}} d\mu_0^{\leq \Lambda}(u).
	\]
	For the focusing nonlinearity, the definitions include the natural cut-offs:
	\begin{align} \label{eq:mu-Lamb-focu-supe}
	d\mu_\Lambda(u):=d\mu^{\leq \Lambda}(u) \otimes d\mu_0^{>\Lambda}(u),
	\end{align}
	where 
	\[
	d\mu^{\leq \Lambda}(u) = \frac{1}{Z^{\Lambda}} e^{\frac{1}{2} \|Q_\Lambda u\|^4_{L^4}} \mathds{1}_{\left\{\|P_{\leq \Lambda}u\|^2_{L^2}<m\right\}} d\mu_0^{\leq \Lambda}(u)
	\]
	if $s>2$ and
	\[
	d\mu^{\leq \Lambda}(u) = \frac{1}{Z^{\Lambda}} e^{\frac{1}{2} \|Q_\Lambda u\|^4_{L^4}} \mathds{1}_{\left\{|\Mcal_{\leq \Lambda}(u)|<m\right\}} d\mu_0^{\leq \Lambda}(u)
	\]
	if $\frac{8}{5} < s \leq 2$. The partition functions $Z^\Lambda$ turn these into probability measures, as usual.
	
	\begin{lemma}[\textbf{Approximating measures}]\label{lem-mu-Lambda}\mbox{}\\
		Let $s>1$, $V$ satisfy Assumption \ref{assu-V}, $\theta <\frac{1}{2}-\frac{1}{s}$, and $\Lambda \geq \lambda_1$. Assume in addition that $s>\frac{8}{5}$ for the focusing nonlinearity. Then the measures $\mu_\Lambda$ are well-defined and absolutely continuous with respect to the Gaussian measure $\mu_0$. Moreover, $\mu_\Lambda$ converge to $\mu$ in the sense that for any measurable set $A\subset \Hc^\theta$,
		\begin{align} \label{eq:limi-mu-Lamb}
		\lim_{\Lambda \to \infty} \mu_\Lambda(A) = \mu(A).
		\end{align}
	\end{lemma}

\begin{proof}
	Thanks to \eqref{eq:Q-P-Lamb} and \eqref{eq:est-Q-Lamb}, the well-definedness of approximating measures follows exactly as the case $\Lambda=\infty$ of the full measure considered above. In addition, the normalization constants $Z^\Lambda$ are finite uniformly in $\Lambda$.
	
	We first prove~\eqref{eq:limi-mu-Lamb} in the defocusing case. Denote
	\begin{align} \label{G-Lambda-defo}
	G_\Lambda(u) := e^{-\frac{1}{2}\|Q_\Lambda u\|^4_{L^4}}, \quad G(u) := e^{-\frac{1}{2} \|u\|^4_{L^4}}
	\end{align}
	We first claim that $G_\Lambda(u) \to G(u)$ in measure with respect to $\mu_0$, i.e.,
	\begin{align} \label{limi-G-Lambda}
	\forall \vareps>0, \quad \lim_{\Lambda \to \infty} \mu_0 \left( u \in \Hc^\theta : |G_\Lambda(u)-G(u)| >\vareps \right) =0.
	\end{align}
	Since $\mu_0(\Hc^\theta)=1$, the convergence in measure is preserved under composition and multiplication by continuous functions. It suffices to show that $\|Q_\Lambda u\|_{L^4} \to \|u\|_{L^4}$ in measure with respect to $\mu_0$. By Chebyshev's inequality, namely
	\[
	\mu_0\left(|\|Q_\Lambda u\|_{L^4} - \|u\|_{L^4} |>\vareps\right) \leq \frac{1}{\vareps^4} \int |\|Q_\Lambda u\|_{L^4}-\|u\|_{L^4}|^4 d\mu_0(u),
	\]
	it suffices to show that
	\[
	\int \|Q_\Lambda u- u\|^4_{L^4} d\mu_0(u)  \to 0 \text{ as } \Lambda \to \infty.
	\]
	To see this, we denote $R_\Lambda := Q_\Lambda - \id$, hence
	\[
	R_\Lambda u = \sum_{j\geq 1} R(\lambda_j) \alpha_j u_j, \quad R(\lambda_j):= \chi\left(\frac{\lambda_j}{\Lambda}\right) -1.
	\] 
	Using \eqref{eq:mu-0-k-A}, we have\footnote{Again, regularizing the delta functions as in the proof of Lemma~\ref{lem-H-theta}.}
	\begin{align*}
	\int |R_\Lambda u(x)|^4 d\mu_0(u) &= {\Tr}\left[(\delta_x)^{\otimes 2} \int |(R_\Lambda u)^{\otimes 2}\rangle \langle (R_\Lambda u)^{\otimes 2}| d\mu_0(u) (\delta_x)^{\otimes 2}\right] \\
	&\leq 2! {\Tr}\left[(\delta_x)^{\otimes 2} (R_\Lambda h^{-1} R_\Lambda)^{\otimes 2} (\delta_x)^{\otimes 2}\right] \\
	&\leq 2! \left({\Tr}[\delta_x R_\Lambda h^{-1}R_\Lambda \delta_x]\right)^2\\
	&= 2! \left(\sum_{j\geq 1} |R(\lambda_j)|^2 \lambda_j^{-1}|u_j(x)|^2\right)^2
	\end{align*}
	By the choice of $\chi$, we have $R(\lambda_j) =0$ for $\lambda_j \leq \Lambda/2$ and $|R(\lambda_j)| \leq 2$ for all $j$, hence
	\[
	\int \|R_\Lambda u\|^4_{L^4} d\mu_0(u) \leq 32 \int \Big(\sum_{\lambda_j>\Lambda/2} \lambda_j^{-1} |u_j(x)|^2 \Big)^2 dx \to 0 \text{ as } \Lambda \to \infty.
	\]
	Thus the claim follows.
	
	Now let $A$ be a measurable set in $\Hc^\theta$. We will show that 
	\begin{align} \label{limi-Q-Lambda}
	\lim_{\Lambda \to \infty} \int \mathds{1}_A G_\Lambda(u) d\mu_0(u) = \int \mathds{1}_A G(u) d\mu_0(u)
	\end{align}
	which is \eqref{eq:limi-mu-Lamb}. Let $\vareps>0$. We introduce the set
	\[
	B_{\Lambda, \vareps}:= \left\{ u\in \Hc^\theta : |G_\Lambda(u) - G(u)| \leq \vareps\right\}.
	\]
	We have
	\[
	\Big|\int_{B_{\Lambda, \vareps}} \mathds{1}_A (G_\Lambda(u) -G(u)) d\mu_0(u) \Big| \leq \vareps.
	\]
	On the other hand, since $G_\Lambda(u), G(u) \in L^2(d\mu_0)$ uniformly in $\Lambda$, we deduce from \eqref{limi-G-Lambda} that
	\[
	\Big|\int_{B^c_{\Lambda, \vareps}} \mathds{1}_A (G_\Lambda(u)-G(u)) d\mu_0(u)\Big| \leq \|G_N(u)-G(u)\|_{L^2(d\mu_0)} \sqrt{\mu_0(B^c_{\Lambda, \vareps})} \to 0 \text{ as } \Lambda \to \infty.
	\]
	Combining the above estimates, we prove \eqref{limi-Q-Lambda}. 
	
	For the focusing case, we denote
	\begin{align} \label{G-Lambda-focu-supe}
	G_\Lambda(u) := e^{\frac{1}{2} \|Q_\Lambda u\|^4_{L^4}} \mathds{1}_{\left\{\|P_{\leq \Lambda} u\|^2_{L^2}<m\right\}}, \quad G(u):= e^{\frac{1}{2} \|u\|^4_{L^4}} \mathds{1}_{\left\{\|u\|^2_{L^2}<m\right\}}
	\end{align}
	for $s>2$ and
	\begin{align} \label{G-Lambda-focu-sub}
	G_\Lambda(u) := e^{\frac{1}{2} \|Q_\Lambda u\|^4_{L^4}} \mathds{1}_{\left\{|\Mcal_{\leq \Lambda}(u)|<m\right\}}, \quad G(u):= e^{\frac{1}{2} \|u\|^4_{L^4}} \mathds{1}_{\left\{|\Mcal(u)|<m\right\}}
	\end{align}
	for $\frac{8}{5}<s\leq 2$. Since $G_\Lambda(u), G(u) \in L^2(d\mu_0)$ uniformly in $\Lambda$, the same argument as in the defocusing case shows \eqref{eq:limi-mu-Lamb}.
\end{proof}
	
An immediate interest of the above measures is that they are easily shown to be invariant under the flow of the approximation NLS equation
\begin{align} \label{NLS-appro}
\left\{
\begin{array}{rcl}
\im  \partial_t u_\Lambda - h u_\Lambda &=& \pm Q_\Lambda(|Q_\Lambda u_\Lambda|^2 Q_\Lambda u_\Lambda), \quad (t,x) \in \R \times \R, \\
\left. u_\Lambda\right|_{t=0} &=& f.
\end{array}
\right.
\end{align}
As in e.g.,~\cite{Bourgain-96,Bourgain-97,BurThoTzv-10}, we shall use this fact extensively.

\begin{lemma}[\textbf{Approximate NLS dynamics}]\label{lem:app dyn}\mbox{}\\ 
	Let $s>1$, $V$ satisfy Assumption \ref{assu-V}, $\theta<\frac{1}{2}-\frac{1}{s}$, and $f\in \Hc^\theta$. Then for each $\Lambda \geq \lambda_1$, the solution to \eqref{NLS-appro} exists globally in time. Moreover, the measures $\mu_\Lambda$ defined in ~\eqref{eq:mu-Lamb} or~\eqref{eq:mu-Lamb-focu-supe} are invariant under the flow of~\eqref{NLS-appro}.
\end{lemma}

\begin{proof}
\noindent \textbf{Step 1, global existence.} We write $u_\Lambda =u^{\lo}_\Lambda + u^{\hi}_\Lambda$ with $u^{\lo}_\Lambda:= P_{\leq \Lambda} u_\Lambda$ and $u^{\hi}_\Lambda:= P_{>\Lambda} u_\Lambda$. As $P_{>\Lambda}Q_\Lambda =0$, the high frequency part satisfies
	\begin{align} \label{NLS-appro-hi}
	\left\{
	\begin{array}{rcl}
	\im  \partial_t u^{\hi}_\Lambda - h u^{\hi}_\Lambda &=& 0, \quad (t,x) \in \R \times \R, \\
	\left. u^{\hi}_\Lambda\right|_{t=0} &=& P_{>\Lambda}f.
	\end{array}
	\right.
	\end{align}
	If we write
	\[
	u^{\hi}_\Lambda(t,x) = \sum_{\lambda_j>\Lambda} \alpha^{\hi}_j(t) u_j(x), \quad P_{>\Lambda} f = \sum_{\lambda_j>\Lambda} \alpha_{j0}  u_j, \quad \alpha_{j0} = \scal{f,u_j},
	\]
	then we have
	\[
	\left\{
	\begin{array}{rcl}
	\im  \partial_t \alpha^{\hi}_j  -\lambda_j \alpha^{\hi}_j &=& 0,  \\
	\left. \alpha^{\hi}_j\right|_{t=0} &=& \alpha_{j0},
	\end{array}
	\right.
	\]
	hence $\alpha^{\hi}_j(t) = e^{-\im t\lambda_j} \alpha_{j0}$ or
	\[
	u^{\hi}_\Lambda(t,x) = \sum_{\lambda_j>\Lambda} e^{-\im t\lambda_j}\alpha_{j0} u_j(x).
	\]
	In particular, the high frequency part exists globally in time.
	
	On the other hand, by applying $P_{\leq \Lambda}$ to both sides of \eqref{NLS-appro} and using \eqref{eq:Q-P-Lamb}, we get
	\begin{align} \label{NLS-appro-lo}
	\left\{
	\begin{array}{rcl}
	\im  \partial_t u^{\lo}_\Lambda - h u^{\lo}_\Lambda &=&  \pm Q_\Lambda\left(|Q_\Lambda u^{\lo}_\Lambda|^2 Q_\Lambda u^{\lo}_\Lambda\right), \quad (t,x) \in \R \times \R, \\
	\left. u^{\lo}_\Lambda\right|_{t=0} &=& P_{\leq\Lambda}f.
	\end{array}
	\right.
	\end{align}
	If we write
	\[
	u^{\lo}_\Lambda(t,x) = \sum_{\lambda_j \leq \Lambda} \alpha^{\lo}_j(t) u_j(x), \quad \alpha^{\lo}_j = a^{\lo}_j + \im b^{\lo}_j,
	\]
	then~\eqref{NLS-appro-lo} is a Hamiltonian ODE of the form
	\[
	\partial_t a^{\lo}_j = \frac{\partial H}{\partial b_j^{\lo}}, \quad \partial_t b^{\lo}_j = - \frac{\partial H}{\partial a^{\lo}_j}, \quad \lambda_j \leq \Lambda,
	\]
	where
	\begin{align*}
	H(u^{\lo}_\Lambda)
	&=\|h^{1/2} u^{\lo}_\Lambda\|^2_{L^2} \pm \frac{1}{2} \|Q_\Lambda u^{\lo}_\Lambda\|^4_{L^4} \\
	&= \sum_{\lambda_j\leq \Lambda} \lambda_j |\alpha^{\lo}_j|^2 \pm \frac{1}{2} \Big\|Q_\Lambda \Big(\sum_{\lambda_j\leq \Lambda} \alpha^{\lo}_j u_j \Big)\Big\|^4_{L^4}=:H(a^{\lo}_j, b^{\lo}_j)
	\end{align*}
	is conserved under the flow of \eqref{NLS-appro-lo}. By the Cauchy-Lipschitz theorem, there exists a local solution to \eqref{NLS-appro-lo}. In addition, thanks to the conservation of mass
	\[
	M(u^{\lo}_\Lambda) = \|u^{\lo}_\Lambda\|^2_{L^2} = \sum_{\lambda_j \leq \Lambda} |\alpha^{\lo}_j|^2,
	\]
	local solutions can be extended globally in time. Thus $u^{\lo}_\Lambda$ exists globally in time. 
	
	\medskip 
	
	\noindent\textbf{Step 2, measure invariance.} We show that $\mu_0^{>\Lambda}$ and $\mu^{\leq \Lambda}$ are invariant under the flow of \eqref{NLS-appro-hi} and \eqref{NLS-appro-lo}. To see this, we denote by $\Phi^{\hi}_\Lambda(t)$ and $\Phi^{\lo}_\Lambda(t)$ the solution maps of \eqref{NLS-appro-hi} and \eqref{NLS-appro-lo} separately.
	
	Let $A$ be a measurable set in $\Hc^\theta$ with $\theta <\frac{1}{2}-\frac{1}{s}$. We have
	\begin{align*}
	\mu_0^{>\Lambda}(\Phi^{\hi}_\Lambda(t)(A)) &= \int_A d\mu_0^{>\Lambda}(u^{\hi}_\Lambda(t)) \\
	&= \int_A \prod_{\lambda_j>\Lambda} \frac{\lambda_j}{\pi} e^{-\lambda_j|\alpha^{\hi}_\Lambda(t)|^2} d\alpha^{\hi}_j(t) \\
	&= \int_A \prod_{\lambda_j>\Lambda} \frac{\lambda_j}{\pi} e^{-\lambda_j|e^{-i\lambda_j} \alpha_{j0}|^2} d\left(e^{-\im t\lambda_j}\alpha_{j0}\right) \\
	&=\int_A \prod_{\lambda_j>\Lambda} \frac{\lambda_j}{\pi} e^{-\lambda_j|\alpha_{j0}|^2} d\alpha_{j0} \\
	&=\int_A d\mu_0^{>\Lambda}(u_0) = \mu_0^{>\Lambda}(A),
	\end{align*}
	where the fourth line follows from the invariance of the Lebesgue measure under rotations. This shows that $\mu_0^{>\Lambda}$ is invariant under the flow of \eqref{NLS-appro-hi}.
	
	Since $\mu^{\leq \Lambda}$ is finite dimensional, the invariance of $\mu^{\leq \Lambda}$ under the flow of \eqref{NLS-appro-lo} follows directly from the conservation of the Hamiltonian $H(u^{\lo}_\Lambda)$ and the Liouville theorem for $du^{\lo}_\Lambda = \prod_{\lambda_j\leq \Lambda} da^{\lo}_j db^{\lo}_j$. We also use  conservation of mass for the focusing measures.
\end{proof}

\section{Cauchy problem and invariant measure, super-harmonic case}
\label{sec:invari super}
\setcounter{equation}{0}

We start our analysis of the Cauchy problem with the case $s>2$, where the Gibbs measure is supported on Sobolev spaces with positive indices, i.e., in $\Hc^\theta$ with $0<\theta<\frac{1}{2}-\frac{1}{s}$. Thus there is a chance to expect a deterministic local well-posedness with initial data lying in the support of the Gibbs measure. This is indeed what we provide in Section~\ref{sec:LWPsuper}. In Section~\ref{sec:GloSuper}, we globalize the flow by using the invariance of the approximate measures introduced above.

\subsection{Deterministic local well-posedness}\label{sec:LWPsuper}

When $s>2$, we may use tools from~\cite{YajZha-01,YajZha-04,ZhaYajLiu-06} to obtain the following deterministic local well-posedness result. 

\begin{proposition}[\textbf{Deterministic local well-posedness, $s>2$}]\label{prop-dete-lwp}\mbox{}\\
Let $s>2$, $V$ satisfy Assumption \ref{assu-V}, and $\theta$ satisfy 
\begin{align} \label{cond-theta}
\frac{1}{2}\left(\frac{1}{2}-\frac{1}{s}\right)<\theta<\frac{1}{2}-\frac{1}{s}.
\end{align}
Let $(p,q)$ be a Strichartz-admissible pair as in Appendix~\ref{sec:app2} and $\gamma>0$ satisfying
\begin{align} \label{pq-dete-lwp}
p>4, \quad \frac{1}{2}-\frac{2}{p}<\gamma<\theta-\frac{2}{p}\left(\frac{1}{2}-\frac{1}{s}\right).
\end{align}
Then for any $f\in \Hc^\theta$, there exist $\delta>0$ and a unique solution to \eqref{eq:intro NLS} with initial datum $\left.u\right|_{t=0}=f$ satisfying
\[
u \in C([-\delta,\delta], \Hc^\theta) \cap L^p([-\delta,\delta], \Wc^{\gamma, q}).
\]
In particular, if $\|f\|_{\Hc^\theta} \leq K$, then 
\begin{align} \label{est-solu}
\|u(t)\|_{\Hc^\theta} \leq 2CK
\end{align}
for all $|t|\leq \delta\sim K^{-\varrho}$ with $\varrho>0$ and some universal constant $C>0$. In addition, if $u_1(t)$ and $u_2(t)$ are respectively solutions to \eqref{eq:intro NLS} with initial data $\left.u_1\right|_{t=0}=f_1, \left.u_2\right|_{t=0}= f_2 \in \Hc^\theta$ that satisfy $\|f_1\|_{\Hc^\theta}, \|f_2\|_{\Hc^\theta} \leq K$, then 
\begin{align} \label{cont-prop}
\|u_1(t)-u_2(t)\|_{\Hc^\theta} \leq 2C \|f_1-f_2\|_{\Hc^\theta}
\end{align}
for all $|t|\leq \delta\sim K^{-\varrho}$. 
\end{proposition}

	Let us comment briefly on the conditions \eqref{cond-theta} and \eqref{pq-dete-lwp}. The second inequality in \eqref{cond-theta} ensures that the support of the Gibbs measure is contained in $\Hc^\theta$. The first inequality in \eqref{cond-theta} guarantees the existence of a Strichartz-admissible pair $(p,q)$ given in \eqref{pq-dete-lwp}. The first condition in~\eqref{pq-dete-lwp} allows us to use H\"older's inequality in time. The first inequality in the second condition in \eqref{pq-dete-lwp} coupled with the admissibility of $(p,q)$ yields
	\[
	\frac{1}{q}=\frac{1}{2}-\frac{2}{p}<\gamma
	\]
	so that we have the Sobolev embedding $\Wc^{\gamma,q} \subset L^\infty(\R)$. Finally, the second inequality in the second condition in \eqref{pq-dete-lwp} implies $\theta-\gamma>\frac{2}{p}\left(\frac{1}{2}-\frac{1}{s}\right)$ which is needed to use Strichartz estimates with a loss of derivatives 
	\begin{align} \label{str-est-supe}
	\|e^{-\im th}f\|_{L^p([-\delta,\delta], \Wc^{\gamma,q})} \lesssim \|f\|_{\Hc^\theta}.
	\end{align}
	See Appendix~\ref{sec:app2} for more details.

\begin{proof}[Proof of Proposition \ref{prop-dete-lwp}]
	The proof is essentially given in \cite{YajZha-04}. For the reader's convenience, we recall some details. Let $\delta>0$ be a small parameter to be chosen later. We denote
	\[
	X^{\theta,\gamma}_\delta:= C([-\delta,\delta], \Hc^\theta) \cap L^p([-\delta,\delta], \Wc^{\gamma, q})
	\]
	with the norm
	\[
	\|u\|_{X^{\theta,\gamma}_\delta}=\|u\|_{L^\infty([-\delta,\delta],\Hc^\theta)} + \|u\|_{L^p([-\delta,\delta], \Wc^{\gamma,q})}.
	\]
	It suffices to show that the Duhamel functional
	\[
	\Phi_{f}(u(t)) = e^{-\im th} f \mp \im  \int_0^t e^{-\im (t-\tau)h} |u(\tau)|^2 u(\tau) d\tau
	\]
	is a contraction on $(B_{X^{\theta,\gamma}_\delta}(L),d)$, where
	\[
	B_{X^{\theta,\gamma}_\delta}(L):=\left\{ u \in X^{\theta,\gamma}_\delta : \|u\|_{X^{\theta,\gamma}_\delta} \leq L\right\}
	\]
	is the ball in $X^{\theta,\gamma}_\delta$ centered at zero and of radius $L$ and
	\[
	d(u_1,u_2) := \|u_1-u_2\|_{X^{\theta,\gamma}_\delta}.
	\]
	By the unitary of $e^{-\im th}$ on $\Hc^\theta$ and the fractional product rule (see Lemma~\ref{lem-prod-rule}), we have
	\begin{align*}
	\sup_{t\in [-\delta,\delta]}\|\Phi_f(u)(t)\|_{\Hc^\theta} &\leq \|f\|_{\Hc^\theta} + \int_0^\delta \||u(\tau)|^2 u(\tau)\|_{\Hc^\theta} d\tau \\
	&\leq \|f\|_{\Hc^\theta} + C\int_0^\delta \|u(\tau)\|^2_{L^\infty} \|u(\tau)\|_{\Hc^\theta} d\tau \\
	&\leq \|f\|_{\Hc^\theta} + C \|u\|^2_{L^2([-\delta,\delta], L^\infty)} \|u\|_{L^\infty([-\delta,\delta], \Hc^\theta)} \\
	&\leq \|f\|_{\Hc^\theta} + C \delta^{1-\frac{2}{p}} \|u\|^2_{L^p([-\delta,\delta], \Wc^{\gamma,q})} \|u\|_{L^\infty([-\delta,\delta], \Hc^\theta)},
	\end{align*}
	where we have used the H\"older inequality in time and the Sobolev embedding $\Wc^{\gamma,q} \subset L^\infty(\R)$ to get the last inequality. On the other hand, using \eqref{str-est-supe}, we have
	\begin{align*}
	\|\Phi_{f}(u)\|_{L^p([-\delta,\delta], \Wc^{\gamma,q})} &\leq \|e^{-\im th} f\|_{L^p([-\delta,\delta], \Wc^{\gamma,q})} + \Big\|\int_0^t e^{-\im (t-\tau)h} |u(\tau)|^2 u(\tau) d\tau \Big\|_{L^p([-\delta,\delta], \Wc^{\gamma,q})} \\
	&\leq C\|f\|_{\Hc^\theta} + \int_0^\delta \|\mathds{1}_{\{\tau<t\}} e^{-\im (t-\tau)h}|u(\tau)|^2 u(\tau)\|_{L^p([-\delta,\delta], \Wc^{\gamma,q})} d\tau \\
	&\leq C\|f\|_{\Hc^\theta} + \int_0^\delta \|e^{-\im th} e^{i\tau h}|u(\tau)|^2 u(\tau)\|_{L^p([-\delta,\delta], \Wc^{\gamma,q})} d\tau \\
	&\leq C\|f\|_{\Hc^\theta} +  C\int_0^\delta \||u(\tau)|^2 u(\tau)\|_{\Hc^\theta} d\tau \\
	&\leq C\|f\|_{\Hc^\theta} + C \delta^{1-\frac{2}{p}} \|u\|^2_{L^p([-\delta,\delta], \Wc^{\gamma,q})} \|u\|_{L^\infty([-\delta,\delta], \Hc^\theta)}.
	\end{align*}
	In particular, we have
	\[
	\|\Phi_{f}(u)\|_{X^{\theta,\gamma}_\delta} \leq C\|f\|_{\Hc^\theta} + C \delta^{1-\frac{2}{p}} \|u\|^3_{X^{\theta,\gamma}_\delta}.
	\]
	By writing
	\[
	|u_1|^2u_1-|u_2|^2u_2 = (u_1-u_2) (|u_1|^2+|u_2|^2) + (\overline{u}_1-\overline{u}_2) u_1u_2,
	\]
	the same argument gives
	\begin{align*}
	\|\Phi_{f}(u_1) - \Phi_{f}(u_2)\|_{X^{\theta,\gamma}_\delta} &\leq \Big\|\int_0^t e^{-\im (t-\tau)h} (|u_1(\tau)|^2u_1(\tau)-|u_2(\tau)|^2u_2(\tau)) d\tau \Big\|_{X^{\theta,\gamma}_\delta} \\
	&\leq C \delta^{1-\frac{2}{p}} \left(\|u_1\|^2_{X^{\theta,\gamma}_\delta} + \|u_2\|^2_{X^{\theta,\gamma}_\delta} \right) \|u_1-u_2\|_{X^{\theta,\gamma}_\delta}.
	\end{align*}
	Hence there exists $C>0$ such that for each $u_1,u_2\in B_{X^{\theta,\gamma}_\delta}(L)$,
	\begin{align*}
	\|\Phi_{f}(u)\|_{X^{\theta,\gamma}_\delta} &\leq C\|f\|_{\Hc^\theta} + C\delta^{1-\frac{2}{p}} L^3, \\
	d(\Phi_{f}(u_1), \Phi_{f}(u_2)) &\leq C \delta^{1-\frac{2}{p}} L^2 d(u_1,u_2).
	\end{align*}
	Taking $L=2C\|f\|_{\Hc^\theta}$ and choosing $\delta>0$ such that $C \delta^{1-\frac{2}{p}} L^2 \leq \frac{1}{2}$, we see that $\Phi_{f}$ is a contraction mapping on $(B_{X^{\theta,\gamma}_\delta}(L),d)$. This shows the existence of a solution satisfying \eqref{est-solu}. Moreover, if $\|f\|_{\Hc^\theta} \leq K$, then we can take $L=2CK$ and $\delta=\nu K^{-\varrho}$ with $\varrho=\frac{2}{1-\frac{2}{p}}$ and $\nu>0$ sufficiently small so that 
	\[
	C\delta^{1-\frac{2}{p}} L^2 = 4C^3 \nu^{1-\frac{2}{p}} \leq \frac{1}{2}.
	\]
	Thus the solution satisfies $\|u(t)\|_{\Hc^\theta} \leq 2C K$ for all $|t|\leq \delta \sim K^{-\varrho}$.
	
	To see \eqref{cont-prop}, we estimate as before and get
	\begin{align*}
	\|u_1-u_2\|_{X^{\theta,\gamma}_\delta} &\leq C\|f_1-f_2\|_{\Hc^\theta} + C\delta^{1-\frac{2}{p}} \left(\|u_1\|^2_{X^{\theta,\gamma}_\delta} + \|u_2\|^2_{X^{\theta,\gamma}_\delta}\right) \|u_1-u_2\|_{X^{\theta,\gamma}_\delta}.
	\end{align*}
	As $\|f_1\|_{\Hc^\theta}, \|f_2\|_{\Hc^\theta} \leq K$, we have
	\[
	\|u_1\|_{X^{\theta,\gamma}_\delta}, ~ \|u_2\|_{X^{\theta,\gamma}_\delta} \leq 2CK,
	\]
	hence
	\[
	\|u_1-u_2\|_{X^{\theta,\gamma}_\delta} \leq C\|f_1-f_2\|_{\Hc^\theta} + C\delta^{1-\frac{2}{p}} K^2 \|u_1-u_2\|_{X^{\theta,\gamma}_\delta}.
	\]
	We choose $\delta>0$ small enough to have $C\delta^{1-\frac{2}{p}} K^2 \leq \frac{1}{2}$. Then we can absorb the second term in the right hand side in the left hand side to obtain 
	\[
	\|u_1-u_2\|_{X^{\theta,\gamma}_\delta} \leq 2C\|f_1-f_2\|_{\Hc^\theta}
	\]
	which yields \eqref{cont-prop}. The proof is complete.
\end{proof}

\subsection{Measure invariance and global well-posedness}\label{sec:GloSuper}
	
We next show that the flow can be globalized for initial data in a set of full Gibbs measure.
\begin{theorem}[\textbf{Almost sure global well-posedness, $s>2$}] \label{theo:almo-gwp-supe}\mbox{}\\
	Let $s>2$, $V$ satisfy Assumption \ref{assu-V}, and $\theta$ satisfy 
	$$ \frac{1}{2}\left(\frac{1}{2}-\frac{1}{s}\right) < \theta<\frac{1}{2}-\frac{1}{s}.$$
	Then there exists a set $\Sigma \subset \Hc^\theta$ satisfying $\mu(\Sigma)=1$ such that for any $f\in \Sigma$, the corresponding solution to \eqref{eq:intro NLS} with initial datum $\left.u\right|_{t=0}=f$ exists globally in time and satisfies
	\begin{align} \label{grow-norm}
	\|u(t)\|_{\Hc^\theta} \leq C \left(\omega(f)+\log^{\frac{1}{2}}(1+|t|)\right), \quad \forall t\in \R
	\end{align}
	for some constant $\omega(f)>0$ depending on $f$ and some universal constant $C>0$. Moreover, the Gibbs measure $\mu$ is invariant under the flow of~\eqref{eq:intro NLS}.
\end{theorem}

The proof is built on two main ingredients:
\medskip

\noindent $\bullet$ The approximate NLS flow~\eqref{NLS-appro} is globalized with explicit bounds on a set of almost full measure.

\medskip
\noindent $\bullet$ The local solution of the full flow~\eqref{eq:intro NLS} is shown to stay close to the approximate solution. 

\medskip

We start with the former ingredient, where we use crucially the invariance of $\mu_\Lambda$ under the approximate flow. 

\begin{lemma}[\textbf{Uniform estimate for the approximate flow, $s>2$}]\label{lem:globflow1}\mbox{}\\
Let $s, \theta$ be as in the previous statement and take $\theta_1$ satisfying 
\[
\theta<\theta_1<\frac{1}{2}-\frac{1}{s}.
\]
Then for all $\Lambda \geq \lambda_1$ and $T,\vareps>0$. There exist $\Sigma_{\Lambda, T,\vareps} \subset \Hc^{\theta_1}$ and $C>0$ independent of $\Lambda, T,\vareps$ such that: 

\noindent (1) $\mu_\Lambda(\Sigma^c_{\Lambda, T,\vareps}) \leq C\vareps$.

\noindent (2) For $f\in \Sigma_{\Lambda, T,\vareps}$, there exists a unique solution to \eqref{NLS-appro} on $[-T,T]$ satisfying
\begin{align} \label{u-Lambda-T}
\|u_\Lambda (t)\|_{\Hc^{\theta_1}} \leq C \left(\log\frac{T}{\vareps}\right)^{1/2}, \quad \forall |t|\leq T.
\end{align}
\end{lemma}

\begin{proof}
Let $K>0$ and denote
\begin{equation}\label{eq:BK}
B_K:= \{u\in \Hc^{\theta_1} : \|u\|_{\Hc^{\theta_1}} \leq K\}. 
\end{equation}
Thanks to \eqref{eq:est-Q-Lamb}, the deterministic local well-posedness given in Proposition~\ref{prop-dete-lwp} applies mutatis mutandis to the approximate flow and implies that for $f\in B_K$, there exists a unique solution to \eqref{NLS-appro} satisfying 
\begin{align} \label{solu-Lambda}
\|u_\Lambda(t)\|_{\Hc^{\theta_1}} \leq  2 C K, \quad \forall |t|\leq \delta
\end{align}
with 
\begin{align}\label{defi-delta}
\delta =\nu (K+1)^{-\varrho},
\end{align} 
where $\varrho >0$ is as in Proposition~\ref{prop-dete-lwp}, $\nu>0$ small, and $C>0$ is independent of $\Lambda, K$. Here we choose $\delta$ slightly smaller than in the proof of Proposition~\ref{prop-dete-lwp}, which will be convenient later.

Denote $J=\left[\frac{T}{\delta}\right]$ the integer part of $\frac{T}{\delta}$ and set 
	\begin{equation}\label{eq:def Sigma 1}
	\Sigma_{\Lambda,T,K} := \bigcap_{j=-J}^J \Phi_\Lambda\left(-j\delta\right)(B_K),
	\end{equation}
	where $\Phi_\Lambda$ is the solution map for~\eqref{NLS-appro}. Since $\mu_\Lambda$ is invariant under the flow of~\eqref{NLS-appro}, we have
\begin{align*}
	\mu_\Lambda(\Sigma_{\Lambda,T,K}^c) &= \mu_\Lambda \Big( \Big( \bigcap_{j=-J}^J \Phi_\Lambda(-j\delta)(B_K)\Big)^c \Big) \\
	&= \mu_\Lambda \Big(\bigcup_{j=-J}^J (\Phi_\Lambda(-j\delta) (B_K) )^c\Big) \\
	&= \mu_\Lambda \Big( \bigcup_{j=-J}^J \Phi_\Lambda(-j\delta)(B^c_K)\Big) \\
	&\leq \sum_{j=-J}^J \mu_\Lambda(\Phi_\Lambda(-j\delta)(B^c_K)) \\
	&\leq 2 \frac{T}{\delta} \mu_\Lambda (B^c_K).
\end{align*}	
Thanks to \eqref{eq:boun-H-theta}, we have
\begin{align*}
\mu_\Lambda(B^c_K) &=\int \mathds{1}_{B^c_K} d\mu_\Lambda(u) \\
&= \int \mathds{1}_{B^c_K} \frac{1}{Z^{\Lambda}} G_\Lambda(u) d\mu_0(u) \\
&\leq \frac{1}{Z^{\Lambda}} \|G_\Lambda(u)\|_{L^2(d\mu_0)} \left(\mu_0(B^c_K)\right)^{1/2} \\
&\leq C e^{-cK^2},
\end{align*}
where $G_\Lambda(u)$ is as in \eqref{G-Lambda-defo} for the defocusing case and in \eqref{G-Lambda-focu-supe} for the focusing one. Here we have used the fact that $G_\Lambda(u) \in L^2(d\mu_0)$ and $Z^{\Lambda} \geq C>0$ uniformly in $\Lambda$. In particular, we obtain
\[
\mu_\Lambda(\Sigma^c_{\Lambda,T,K}) \leq \frac{2C}{\nu} T (K+1)^\varrho e^{-cK^2} \leq C Te^{-cK^2}
\]
for some constants $C,c>0$. Note that the constants $C,c$ may change from line to line but are independent of $\Lambda, T, K$. By choosing 
\begin{align} \label{choi-K}
K = C \left(\log \frac{T}{\vareps}\right)^{1/2}
\end{align}
for a suitable constant $C>0$, we obtain 
\[
\mu_\Lambda(\Sigma_{\Lambda,T,K}^c) \leq C \vareps.
\]
Setting $\Sigma_{\Lambda,T,\vareps} = \Sigma_{\Lambda,T,K}$ with $K$ as in \eqref{choi-K}, we have the first item. To see the second item, we observe that for $|t|\leq T$, there exist an integer $j$ and $\delta_1 \in[-\delta, \delta]$ such that $t=j\delta + \delta_1$, hence
\[
u_\Lambda(t) = \Phi_\Lambda(\delta_1) \Phi_\Lambda(j \delta) f.
\]
Since $f \in \Phi_{\Lambda}(-j\delta)(B_K)$, we have $\Phi_\Lambda(j\delta) f \in B_K$ which, by \eqref{solu-Lambda}, yields
\begin{align} \label{solu-Lambda-T}
\|u_\Lambda(t)\|_{\Hc^{\theta_1}} \leq  2 C K, \quad \forall |t|\leq T.
\end{align}
Taking into account \eqref{choi-K}, we have the desired estimate.
\end{proof}

We turn to the second main ingredient.

\begin{lemma}[\textbf{Comparison between approximate and exact flows, $s> 2$}]\label{lem:appflow1}\mbox{}\\
Let $\Sigma_{\Lambda, T,\vareps}$ be as in the previous lemma. Then for any $f\in \Sigma_{\Lambda,T,\vareps}$, there exists a unique solution to~\eqref{eq:intro NLS} with initial data $\left.u\right|_{t=0}=f$ satisfying
	\begin{align} \label{est-diff}
	\|u(t)-u_\Lambda(t)\|_{\Hc^\theta} \leq C(T,\vareps)\Lambda^{(\theta-\theta_1)/2}, \quad \forall |t|\leq T
	\end{align}
	for all $\Lambda$ sufficiently large and some constant $C(T,\vareps)>0$ independent of $\Lambda$. In particular, there exist $\Sigma_{T,\vareps} \subset \Hc^\theta$ and $C>0$ independent of $T, \vareps$ such that:
	
	\noindent (1) $\mu(\Sigma^c_{T,\vareps}) \leq C\vareps$.
	
	\noindent (2) For $f\in \Sigma_{T,\vareps}$, there exists a unique solution to \eqref{eq:intro NLS} with initial data $\left.u\right|_{t=0}=f$ on $[-T,T]$ satisfying
	\begin{align}\label{est-u-T}
	\|u(t)\|_{\Hc^\theta} \leq C \left(\log \frac{T}{\vareps}\right)^{1/2}, \quad \forall |t|\leq T.
	\end{align}
\end{lemma}

\begin{proof}
	{\bf Estimating the difference.} Denote 
	$$v_\Lambda:= Q_\Lambda u_\Lambda$$
	with $Q_\Lambda$ as in~\eqref{eq:Q-Lamb}. We first study the difference $u-v_\Lambda$ on 
	\begin{equation}\label{eq:resol space 1}
	X^{\theta,\gamma}([-\delta,\delta]) = L^\infty([-\delta,\delta], \Hc^\theta) \cap L^p([-\delta,\delta], \Wc^{\gamma,q}),
	\end{equation}
	where $\delta$ is as in \eqref{defi-delta}, $(p,q)$ is a Strichartz-admissible pair (see Appendix~\ref{sec:app2}), and $\gamma$ is such that
	\begin{align} \label{pgam}
	p>4, \quad \frac{1}{2}-\frac{2}{p}<\gamma<\theta-\frac{2}{p}\left(\frac{1}{2}-\frac{2}{p}\right).
	\end{align}
	From \eqref{NLS-appro}, we have the following Duhamel formula
	\[
	u(t)-v_\Lambda(t) = e^{-\im th}(f - Q_\Lambda f) \mp \im \int_0^t e^{-\im (t-\tau) h} \left( |u(\tau)|^2 u(\tau) - Q^2_\Lambda (|v_\Lambda(\tau)|^2 v_\Lambda(\tau))\right) d\tau
	\]
	which, by Strichartz estimates (see Appendix~\ref{sec:app2}), yields
	\[
	\|u-v_\Lambda\|_{X^{\theta,\gamma}([-\delta,\delta])} \leq C \|f-Q_\Lambda f\|_{\Hc^\theta} + C\||u|^2u-Q^2_\Lambda(|v_\Lambda|^2 v_\Lambda)\|_{L^1([-\delta,\delta], \Hc^\theta)}.
	\]
	For the linear term, we estimate (recall~\eqref{eq:Q-Lamb})
	\[
	\|f-Q_\Lambda f\|_{\Hc^\theta} \leq C \Lambda^{(\theta-\theta_1)/2} \|f\|_{\Hc^{\theta_1}} \leq C K \Lambda^{(\theta-\theta_1)/2}.
	\]
	For the nonlinear term, we write
	\[
	|u|^2u-Q_\Lambda^2(|v_\Lambda|^2 v_\Lambda) = |u|^2 u - |v_\Lambda|^2 v_\Lambda + (\id-Q_\Lambda^2)(|v_\Lambda|^2 v_\Lambda).
	\]
	Since $f \in \Sigma_{\Lambda, T,\vareps}$ (see \eqref{eq:def Sigma 1}), we have $\|f\|_{\Hc^{\theta_1}}=\|u_\Lambda(0)\|_{\Hc^{\theta_1}}\leq K$ and the local theory for \eqref{NLS-appro} ensures the existence of a unique solution satisfying $\|u_\Lambda\|_{X^{\theta_1,\gamma}([-\delta,\delta])} \leq 2CK$ which, by \eqref{eq:est-Q-Lamb}, yields 
	\[
	\|v_\Lambda\|_{X^{\theta_1,\gamma}([-\delta,\delta])} \leq  2CK,
	\] 
	where $X^{\theta_1,\gamma}([-\delta,\delta])$ is as in~\eqref{eq:resol space 1} with $p,q,\gamma$ as in \eqref{pgam}. Estimating as in the proof of Proposition \ref{prop-dete-lwp}, we have
	\begin{align*}
	\|(\id-Q^2_\Lambda)(|v_\Lambda|^2 v_\Lambda)\|_{L^1([-\delta,\delta], \Hc^{\theta})} &\leq C\Lambda^{(\theta-\theta_1)/2} \||v_\Lambda|^2 v_\Lambda\|_{L^1([-\delta,\delta], \Hc^{\theta_1})} \\
	&\leq C\delta^{1-\frac{2}{p}}\Lambda^{(\theta-\theta_1)/2} \|v_\Lambda\|^3_{X^{\theta_1,\gamma}([-\delta,\delta])} \\
	&\leq C \delta^{1-\frac{2}{p}} K^3 \Lambda^{(\theta-\theta_1)/2}.
	\end{align*}
	Since $\|f\|_{\Hc^\theta} \leq \|f\|_{\Hc^{\theta_1}}\leq K$, the local theory and \eqref{eq:est-Q-Lamb} give
	\[
	\|u\|_{X^{\theta,\gamma}([-\delta,\delta])}, ~ \|v_\Lambda\|_{X^{\theta,\gamma}([-\delta,\delta])} \leq 2CK,
	\]
	which implies
	\begin{align*}
	\||u|^2u&-|v_\Lambda|^2v_\Lambda\|_{L^1([-\delta,\delta], \Hc^\theta)} \\
	&\leq C \delta^{1-\frac{2}{p}} \left( \|u\|^2_{X^{\theta,\gamma}([-\delta,\delta])} + \|v_\Lambda\|^2_{X^{\theta,\gamma}([-\delta,\delta])}\right)\|u-v_\Lambda\|_{X^{\theta,\gamma}([-\delta,\delta])} \\
	&\leq C\delta^{1-\frac{2}{p}} K^2 \|u-v_\Lambda\|_{X^{\theta,\gamma}([-\delta,\delta])}.
	\end{align*}
	By the choice of $\delta$ with some $\nu>0$ small, we deduce
	\[
	\|u-v_\Lambda\|_{X^{\theta,\gamma}([-\delta,\delta])} \leq C K \Lambda^{(\theta-\theta_1)/2} + \frac{1}{2} \|u-v_\Lambda\|_{X^{\theta,\gamma}([-\delta,\delta])}
	\]
	hence
	\[
	\|u-v_\Lambda\|_{X^{\theta,\gamma}([-\delta,\delta])} \leq 2C K \Lambda^{(\theta-\theta_1)/2}.
	\]
	On the other hand, we infer from \eqref{solu-Lambda-T} that
	\[
	\|u_\Lambda(t)-v_\Lambda(t)\|_{\Hc^\theta} \leq C\Lambda^{(\theta-\theta_1)/2} \|u_\Lambda(t)\|_{\Hc^{\theta_1}} \leq C K \Lambda^{(\theta-\theta_1)/2}
	\]
	which gives
	\begin{align}\label{eq:utov}
	\|u(t)-u_\Lambda(t)\|_{\Hc^\theta} &\leq \|u(t)-v_\Lambda(t)\|_{\Hc^\theta} + \|u_\Lambda(t)-v_\Lambda(t)\|_{\Hc^\theta} \nonumber \\
	&\leq 3C K \Lambda^{(\theta-\theta_1)/2}, \quad \forall |t|\leq \delta.
	\end{align}
	We can iterate the above argument $\left[\frac{T}{\delta}\right]$ many times. For instance, at the second iteration, we have
	\[
	\|u-v_\Lambda\|_{X^{\theta,\gamma}([0,2\delta])} \leq C\|u(\delta)-v_\Lambda(\delta)\|_{\Hc^\theta} + \text{ nonlinear term}.
	\]
	Note that 
	\[
	\|u(\delta)-v_\Lambda(\delta)\|_{\Hc^\theta} \leq \|u-v_\Lambda\|_{L^\infty([-\delta,\delta], \Hc^\theta)} \leq \|u-v_\Lambda\|_{X^{\theta,\gamma}([-\delta,\delta])} \leq 2CK \Lambda^{(\theta-\theta_1)/2}.
	\]
	The nonlinear term can be handled as before by noting that $\|u_\Lambda(\delta)\|_{\Hc^{\theta_1}} = \|\Phi_\Lambda(\delta)\|_{\Hc^{\theta_1}} \leq K$ (see \eqref{eq:def Sigma 1}) and
	\[
	\|u(\delta)\|_{\Hc^{\theta}} \leq \|u_\Lambda(\delta)\|_{\Hc^{\theta_1}} + \|u(\delta)-u_\Lambda(\delta)\|_{\Hc^\theta} \leq K + 1
	\]
	provided that $\Lambda$ is taken sufficiently large. The above estimate is the reason why we take $\delta$ as in \eqref{defi-delta}. Thus we get
	\[
	\|u-v_\Lambda\|_{X^{\theta,\gamma}([0,2\delta])} \leq (2C)^2 K \Lambda^{(\theta-\theta_1)/2}.
	\]
	Arguing as in \eqref{eq:utov}, we obtain
	\[
	\|u(t)-u_\Lambda(t)\|_{\Hc^\theta} \leq (3C)^2 K \Lambda^{(\theta-\theta_1)/2}, \quad \forall t\in [0,2\delta].
	\]
	After $\left[\frac{T}{\delta}\right]$ iterations, we can sum over all sub-intervals to get
	\[
	\|u-v_\Lambda\|_{X^{\theta,\gamma}([-T,T])} \leq C e^{c\frac{T}{\delta}} K \Lambda^{(\theta-\theta_1)/2} \leq C e^{cT(K+1)^\varrho} K \Lambda^{(\theta-\theta_1)/2}.
	\]
	By the same reasoning as in \eqref{eq:utov} and invoking \eqref{choi-K}, we prove \eqref{est-diff}.
	
	\noindent {\bf Globalizing the flow.} By \eqref{est-diff}, there exists $\Lambda_1$ sufficiently large such that for $f\in \Sigma_{\Lambda_1, T,\vareps}$, the corresponding solution to \eqref{eq:intro NLS} with initial data $\left.u\right|_{t=0}$ satisfies
	\[
	\|u(t)-u_\Lambda(t)\|_{\Hc^\theta} \ll 1, \quad \forall |t|\leq T, \quad \forall \Lambda \geq \Lambda_1.
	\]
	From this and \eqref{u-Lambda-T}, we have
	\[
	\|u(t)\|_{\Hc^\theta} \leq C \left(\log\frac{T}{\vareps}\right)^{1/2}, \quad \forall |t|\leq T.
	\]
	This proves \eqref{est-u-T} by setting $\Sigma_{T,\vareps} :=\Sigma_{\Lambda_1, T,\vareps}$. It remains to prove the first item. We estimate
	\begin{align*}
	\mu(\Sigma^c_{\Lambda_1, T,\vareps}) = \int \mathds{1}_{\Sigma^c_{\Lambda_1,T,\vareps}} d\mu(u) = \int_{\Sigma^c_{\Lambda_1,T,\vareps}} \frac{1}{Z} G(u) d\mu_0(u) \leq C \int_{\Sigma^c_{\Lambda_1,T,\vareps}}  G(u) d\mu_0(u),
	\end{align*}
	where $G(u)$ is as in \eqref{G-Lambda-defo} for the defocusing nonlinearity and in \eqref{G-Lambda-focu-supe} for the focusing one. Here we have used the fact that $Z\geq C>0$. 
	
	To estimate the last integral in terms of $\mu_{\Lambda_1}$, we consider two cases. For the defocusing nonlinearity, we use the fact that $\|Q_\Lambda u\|^4_{L^4}\leq C_1\|u\|^4_{L^4}$ for some constant $C_1>0$. Without loss of generality, we may assume that $C_1\geq 1$. By H\"older's inequality, we estimate
	\begin{align*}
	\int_{\Sigma^c_{\Lambda_1,T,\vareps}}  e^{-\frac{1}{2}\|u\|^4_{L^4}} d\mu_0(u) &\leq\int_{\Sigma^c_{\Lambda_1,T,\vareps}} e^{-\frac{1}{2C_1}\|Q_{\Lambda_1}u\|^4_{L^4}} d\mu_0(u) \\
	&\leq \Big(\int_{\Sigma^c_{\Lambda_1,T,\vareps}} e^{-\frac{1}{2}\|Q_{\Lambda_1} u\|^4_{L^4}} d\mu_0(u)\Big)^{1/C_1} \Big(\int_{\Sigma^c_{\Lambda_1,T,\vareps}} d\mu_0(u)\Big)^{1/C_1'} \\
	&=\Big(Z^{\Lambda_1}\int_{\Sigma^c_{\Lambda_1,T,\vareps}} d\mu_{\Lambda_1}(u)\Big)^{1/C_1} \Big(\int_{\Sigma^c_{\Lambda_1,T,\vareps}} d\mu_0(u)\Big)^{1/C_1'} \\
	&\leq C\left(\mu_{\Lambda_1}(\Sigma^c_{\Lambda_1,T,\vareps})\right)^{1/C_1} \\
	&< C\vareps^{1/C_1},
	\end{align*}
	where we used that $(C_1,C_1')$ is a H\"older-conjugate pair, $\mu_0$ is a probability measure, and $Z^{\Lambda_1} \leq 1$. Here the last inequality follows from Lemma \ref{lem:globflow1}. In the focusing case, we have
	\begin{align*}
	\int_{\Sigma^c_{\Lambda_1,T,\vareps}}  G(u) d\mu_0(u) &=\int_{\Sigma^c_{\Lambda_1,T,\vareps}}  G(u) \mathds{1}_{\left\{\|u\|^2_{L^2}<m\right\}} d\mu_0(u) \\
	&\leq \|G(u)\|_{L^2(d\mu_0)} \left(\int_{\Sigma^c_{\Lambda_1,T,\vareps}} \mathds{1}_{\left\{\|u\|^2_{L^2}<m\right\}} d\mu_0(u)\right)^{1/2} \\
	&\leq C \left(\int_{\Sigma^c_{\Lambda_1,T,\vareps}} \mathds{1}_{\left\{\|P_{\leq \Lambda_1}u\|^2_{L^2}<m\right\}} d\mu_0(u)\right)^{1/2} \\
	&\leq C \left(\int_{\Sigma^c_{\Lambda_1,T,\vareps}} e^{\frac{1}{2}\|Q_{\Lambda_1} u\|^4_{L^4}}\mathds{1}_{\left\{\|P_{\leq \Lambda_1}u\|^2_{L^2}<m\right\}} d\mu_0(u)\right)^{1/2} \\
	&= C \left(\mu_{\Lambda_1}(\Sigma^c_{\Lambda_1,T,\vareps})\right)^{1/2} \\
	&< C \vareps^{1/2}.
	\end{align*}
	In both cases, we adjust $\vareps$ slightly to get the desired result. The proof is complete.
\end{proof}

Now we can conclude the proof of Theorem \ref{theo:almo-gwp-supe}.

\begin{proof}[Proof of Theorem~\ref{theo:almo-gwp-supe}]
\noindent\textbf{Almost-sure flow, global in time}. Fix $\vareps>0$ and set $T_n = 2^n$, $\vareps_n = \vareps 2^{-n}$. Let $\Sigma_n=\Sigma_{T_n, \vareps_n}$ be as in Lemma \ref{lem:appflow1} and set
\[
\Sigma_\vareps=\bigcap_{n=1}^\infty \Sigma_n.
\]
For $f\in \Sigma_\vareps$, we have $f\in \Sigma_n$ for all $n\geq 1$. By Lemma \ref{lem:appflow1}, the corresponding solution to \eqref{eq:intro NLS} with initial data $\left. u\right|_{t=0}=f$ exists globally in time since it exists on $[2^{-n},2^n]$ for all $n\geq 1$. In addition, we have
\[
\mu(\Sigma^c_\vareps) =\mu\Big(\bigcup_{n=1}^\infty \Sigma^c_n\Big) \leq \sum_{n=1}^\infty \mu(\Sigma_n^c) \leq C \sum_{n=1}^\infty \vareps_n=C\vareps.
\] 
Define now
\begin{align} \label{Sigma}
\Sigma = \bigcup_{\vareps>0} \Sigma_\vareps
\end{align}
so that
\[
\mu(\Sigma^c) = \mu \Big(\bigcap_{\vareps>0} \Sigma^c_\vareps\Big) \leq \inf_{\vareps>0} \mu(\Sigma^c_\vareps)  =0.
\]
Hence we have indeed found the claimed set of full $\mu$ measure on which the flow is globally defined. 

\medskip

\noindent\textbf{Growth in time.} Pick $f\in \Sigma$ and $t\in \R$. It must be that $f\in \Sigma_\vareps$ for some $\vareps >0$ and that 
\[
2^{n-1}\leq 1+|t| \leq 2^n
\]
for some integer $n\geq 1$. In particular, we have $n \leq 1+\log_2(1+|t|) \leq 1+2\log(1+|t|)$. For such $n$, we apply Lemma \ref{lem:appflow1} with $f\in \Sigma_n$ to get
\[
\|u(t)\|_{\Hc^\theta} \leq C \left(\log \frac{T_n}{\vareps_n}\right)^{1/2} = C \left(\log \frac{1}{\vareps} + n \right)^{1/2} \leq C \left(\log^{\frac{1}{2}} \frac{1}{\vareps} + \log^{\frac{1}{2}} (1+|t|)\right)
\]
for some constant $C>0$. Since $\vareps$ depends on $f$, we prove \eqref{grow-norm} with $\omega(f) = \log^{\frac{1}{2}} \frac{1}{\vareps}$. 

\medskip

\noindent\textbf{Measure invariance.} Let $\theta$ satisfy $\frac{1}{2}\left(\frac{1}{2}-\frac{1}{s}\right)<\theta<\frac{1}{2}-\frac{1}{s}$. By the time reversibility of the solution map, it suffices to show
\begin{align} \label{inva-prof}
\mu(A) \leq \mu(\Phi(t)(A))
\end{align}
for all measurable sets $A \subset \Sigma$ and all $t\in \R$. By the inner regularity, there exists a sequence $\{F_n\}_n$ of closed set in $\Hc^\theta$ such that $F_n \subset A$ and $\mu(A) = \lim_{n\rightarrow \infty} \mu(F_n)$.  We claim that it suffices to prove \eqref{inva-prof} for closed sets. Note that since $F_n \subset A$, the uniqueness of solutions implies that $\mu(\Phi(t) (F_n)) \leq \mu(\Phi(t)(A))$. If \eqref{inva-prof} is true for closed sets, then we have
\[
\mu(A) = \lim_{n\rightarrow \infty} \mu(F_n) \leq \limsup_{n\rightarrow \infty} \mu(\Phi(t) (F_n)) \leq \mu(\Phi(t)(A)),
\]
hence \eqref{inva-prof} is true for all measurable sets. Now given a closed set $F \subset \Hc^\theta$. Take $\theta<\theta_1 <\frac{1}{2}-\frac{1}{s}$ and set
\[
U_n := \left\{u \in F : \|u\|_{\Hc^{\theta_1}}\leq n\right\}.
\]
From \eqref{eq:boun-beta-p}, we infer that
\begin{align*}
\mu(F\setminus U_n) &= \frac{1}{Z} \int_{F\setminus U_n} G(u)d\mu_0(u) \\
&\leq C \|G(u)\|_{L^2(d\mu_0)} \left(\mu_0(F\setminus U_n)\right)^{1/2} \\
&\leq C \left(\mu_0(\|u\|_{\Hc^{\theta_1}} >n)\right)^{1/2} \\
&\leq Ce^{-cn^2} \to 0 \text{ as } n \to \infty.
\end{align*}
Thus we have
\[
\mu(F) =\lim_{n\rightarrow \infty} \mu(U_n).
\]
The same argument as above yields that it suffices to prove \eqref{inva-prof} for closed sets of $\Hc^\theta$ which are bounded in $\Hc^{\theta_1}$. 

Let $U$ be such a set and fix $t >0$ (the case $t<0$ is similar). Since $U$ is bounded in $\Hc^{\theta_1}$, the local theory ensures the existence of $K>0$ such that
\[
\{\Phi(\tau)(U) : 0\leq \tau \leq t\} \subset \{u\in \Hc^{\theta_1} : \|u\|_{\Hc^{\theta_1}} \leq K\}=:B_{\theta_1,K}.
\]  
Set 
\[
\delta=\nu K^{-\varrho},
\]
where $\varrho$ is as in Proposition \ref{prop-dete-lwp} and $\nu>0$ is a small constant independent of $K$. It suffices to prove
\begin{align} \label{inva-prof-redu}
\mu(U) \leq \mu(\Phi(t)(U)), \quad \forall t\in [0,\delta].
\end{align}
Indeed, we split $[0,t]$ into intervals of size $\delta$ and apply \eqref{inva-prof-redu} on these intervals. Such an iteration is possible since by the continuity of the solution map, the image of each interval remains closed in $\Hc^\theta$ and included in $B_{\theta_1,K}$.

To prove \eqref{inva-prof-redu}, we take $\vareps>0$ and denote by $B_{\theta, \vareps}$ the open ball in $\Hc^\theta$ centered at the origin and of radius $\vareps$. There exist $0<c \ll 1$ and $\Lambda\geq \lambda_1$ sufficiently large such that
\begin{align*}
\Phi_\Lambda(t) (U+ B_{\theta,c\vareps}) &\subset \Phi_\Lambda(t) (U) + B_{\theta,\vareps/2} \quad \tag{Lipschitz continuity \eqref{cont-prop}} \\
&\subset \Phi(t) (U) + B_{\theta,\vareps}.  \tag{Lemma \ref{lem:appflow1}}
\end{align*} 
Thus we have the chain of inequalities
\begin{align*}
\mu(U) \leq \mu(U+B_{\theta, c\vareps}) &\leq \liminf_{\Lambda\rightarrow\infty} \mu_\Lambda(U+B_{\theta, c\vareps})  \tag{$\mu_\Lambda \rightharpoonup \mu$ weakly}\\
& = \liminf_{\Lambda\rightarrow \infty} \mu_\Lambda(\Phi_\Lambda(t)(U+B_{\theta, c\vareps})) \tag{invariance of $\mu_\Lambda$}\\
&\leq \liminf_{\Lambda\rightarrow \infty} \mu_\Lambda(\Phi(t) (U)+B_{\theta,\vareps})  \\
&\leq \limsup_{\Lambda\rightarrow \infty} \mu_\Lambda(\Phi(t) (U) + \overline{B}_{\theta,\vareps}) \\
&\leq \mu(\Phi(t)(U) + \overline{B}_{\theta,\vareps}). \tag{$\mu_\Lambda \rightharpoonup \mu$ weakly} 
\end{align*}
Letting $\vareps \rightarrow 0$, we obtain $\mu(U) \leq \mu(\Phi(t) (U))$ for all $t\in [0,\delta]$. The proof is complete.
\end{proof}

\begin{remark}[Higher non-linearities] \label{rem-meas-inva-sup}
	The argument presented in this section can be applied to prove the invariance of Gibbs measures associated to \eqref{eq:NLS-general} for any $s>2$ and $2<\kappa<4+s$ as long as the Gibbs measure is well-defined. In fact, the local well-posedness holds in 
	\[
	C([-\delta,\delta], \Hc^\theta) \cap L^p([-\delta,\delta], \Wc^{\gamma,q}),
	\]
	where the exponents are chosen as follows:
	\[
	\frac{1}{2}-\frac{2}{\max\{\kappa-2,4\}}\left(\frac{1}{2}+\frac{1}{s}\right) <\theta<\frac{1}{2}-\frac{1}{s},
	\]
	the pair $(p,q)$ is Strichartz-admissible and we assume
	\[
	p>\max\{\kappa-2,4\}, \quad \frac{1}{2}-\frac{2}{p}<\gamma<\theta-\frac{2}{p}\left(\frac{1}{2}-\frac{1}{s}\right).
	\]
	The condition $\kappa <4+s$ guarantees the existence of such exponents, and then the local well-posedness follows from the following Strichartz estimate
	\[
	\|e^{-\im t h} f\|_{L^p_\delta \Wc^{\gamma,q}} \lesssim \|f\|_{\Hc^\theta}.
	\]
	and nonlinear estimates
	\begin{align*}
	\||u|^{\kappa-2} u\|_{L^1_\delta \Hc^\theta} &\leq \|u\|^{\kappa-2}_{L^{\kappa-2}_\delta L^\infty} \|u\|_{L^\infty_\delta \Hc^\theta} \lesssim \delta^{1-\frac{\kappa-2}{p}}\|u\|^{\kappa-2}_{L^p_\delta \Wc^{\gamma,q}} \|u\|_{L^\infty_\delta \Hc^\theta} \\
	\||u_1|^{\kappa-2} u_1-|u_2|^{\kappa-2}u_2\|_{L^1_\delta \Hc^\theta} &\lesssim \delta^{1-\frac{\kappa-2}{p}} \left(\|u_1\|^{\kappa-2}_{L^p_\delta\Wc^{\gamma,q}} + \|u_2\|^{\kappa-2}_{L^p_\delta\Wc^{\gamma,q}} \right)\|u_1-u_2\|_{L^\infty_\delta \Hc^\theta},
	\end{align*}
	where we have used that $\Wc^{\gamma,q} \subset L^\infty(\R)$ by Sobolev embedding, since $\frac{1}{q}=\frac{1}{2}-\frac{2}{p}<\gamma$. Once the local well-posedness is proved, the invariance of Gibbs measures follows by the same argument as above.\hfill$\diamond$
\end{remark}

\section{Cauchy problem and invariant measure, (sub)-harmonic case}
\label{sec:invari sub}
\setcounter{equation}{0}

We now consider the case $1<s \leq 2$ for which the Gibbs measure is supported on $L^2$-based Sobolev spaces of negative indices. It seems difficult to obtain deterministic local well-posedness for \eqref{eq:intro NLS} in this case. In \cite{BurThoTzv-10}, a deterministic LWP was proved for $s=2$ using some intricate multilinear estimates. The proof of these estimates relies heavily on a detailed understanding of spectral properties of the harmonic oscillator. We do not expect the proof of such multilinear estimates to carry through to the case $1<s<2$. In Section \ref{sec:LWPsub}, we aim to prove an almost sure local well-posedness for \eqref{eq:intro NLS}. Section \ref{sec:GloSub} is devoted to an almost sure global well-posedness and the measure invariance. 

\subsection{Probabilistic local well-posedness}
\label{sec:LWPsub}

Our proof of almost sure local well-posedness relies on the following probabilistic Strichartz estimates. 

\begin{lemma}[\bf Probabilistic Strichartz estimates] \label{lem-est-expo} \mbox{}\\
	Let $1<s\leq 2$, $V$ satisfy Assumption \ref{assu-V}, $0\leq \beta <\frac{1}{2}$, and $p>\max\left\{2,\frac{4}{s(1-2\beta)}\right\}$ be an even integer. Then there exist $C,c>0$ such that
	\begin{align} \label{fern-ineq}
	\int e^{c\|e^{-\im th}f\|^2_{\Wc^{\beta,p}}} d\mu_0(f) \leq C, \quad \forall t\in \R.
	\end{align}
	Furthermore, for $\theta<1/2-1/s$ and all $\lambda>0$,
	\begin{align} \label{prob-str-est-1}
	\mu_0\left(f\in \Hc^\theta : \|e^{-\im th}f\|_{L^\infty(\R,\Wc^{\beta,p})}>\lambda\right) \leq C e^{-c\lambda^2}
	\end{align}
	and for all $T>0$, all $q\geq 1$, and all $\lambda>0$,
	\begin{align} \label{prob-str-est-2}
	\mu_0\left(f\in \Hc^\theta : \|e^{-\im th}f\|_{L^q([-T,T], \Wc^{\beta,p})}>\lambda\right) \leq C e^{-c\frac{\lambda^2}{T^{2/q}}}.
	\end{align}
\end{lemma}
\begin{proof}
	Denote $g_f(t):=e^{-\im th}f$. We write 
	\[
	f=\sum_{j\geq 1} \alpha_j u_j, \quad \alpha_j = \langle f,u_j\rangle,
	\]
	hence
	\[
	g_f(t)= \sum_{j\geq 1} \beta_j(t) u_j, \quad \beta_j(t) = \alpha_j e^{-\im t\lambda_j}.
	\]
	Observe that
	\begin{align*}
	d\mu_0(g_f(t)) &= \prod_{j\geq 1} \frac{\lambda_j}{\pi} e^{-\lambda_j|\beta_j(t)|^2}d\beta_j(t) \\
	&=\prod_{j\geq 1} \frac{\lambda_j}{\pi} e^{-\lambda_j|\alpha_j|^2}d\alpha_j = d\mu_0(f),
	\end{align*}
	where we have used the invariance of the Lebesgue measure under rotations. Thus we get
	\[
	\int e^{c\|e^{-\im th}f\|^2_{W^{\beta,p}}}d\mu_0(f) = \int e^{c\|g_f(t)\|^2_{W^{\beta,p}}}d\mu_0(g_f(t)), \quad \forall t\in \R. 
	\]
	By Lemma \ref{lem-expo-p}, we prove \eqref{fern-ineq}. This immediately yields \eqref{prob-str-est-1}. To see \eqref{prob-str-est-2}, we use H\"older's inequality in time to get
	\[
	\|e^{-\im th}f\|_{L^q([-T,T],\Wc^{\beta,p})} \leq CT^{1/q}\|e^{-\im th}f\|_{L^\infty([-T,T], \Wc^{\beta,p})}.
	\]
	Thus
	\[
	\mu_0\left(f\in \Hc^\theta: \|e^{-\im th}f\|_{L^q([-T,T],\Wc^{\beta,p})} >\lambda\right) \leq \mu_0\left(f\in \Hc^\theta : \|e^{-\im th}f\|_{L^\infty([-T,T],\Wc^{\beta,p})}> C\frac{\lambda}{T^{1/q}}\right)
	\]
	and \eqref{prob-str-est-2} follows from \eqref{prob-str-est-1}.
\end{proof}

We are now able to state the almost sure LWP for \eqref{eq:intro NLS}.

\begin{proposition} [\bf Probabilistic local well-posedness, $1<s\leq 2$] \label{prop-almo-sure-lwp} \mbox{}\\
	Let $1<s\leq 2$, $V$ satisfy Assumption \ref{assu-V}, $0\leq \beta <\frac{s-1}{2s}$, and $\theta <\frac{1}{2}-\frac{1}{s}$. For any $K>0$, there exist $\Sigma(K) \subset \Hc^\theta$ satisfying $\mu_0(\Sigma^c(K)) \leq Ce^{-cK^2}$ for some constants $C,c>0$ and $\delta \sim K^{-4}>0$ such that for all $f \in \Sigma(K)$, there exists a unique solution to \eqref{eq:intro NLS} with initial data $\left.u\right|_{t=0}=f$ satisfying
	\[
	u(t)-e^{-\im th} f \in C([-\delta,\delta],\Hc^\beta) \cap L^8([-\delta,\delta], \Wc^{\beta,4}).
	\]
	In addition, for all $f\in \Sigma(K)$, we have $\|f\|_{\Hc^\theta}\leq K$ and the corresponding solution to \eqref{eq:intro NLS} with initial data $\left.u\right|_{t=0}=f$ satisfies 
	\begin{align} \label{est-solu-sub}
	\|u(t)\|_{\Hc^\theta} \leq 2K, \quad \forall |t|\leq \delta.
	\end{align}
	Furthermore, if $u_1(t)$ and $u_2(t)$ are respectively solutions to \eqref{eq:intro NLS} with initial data $\left.u_1\right|_{t=0}=f_1, \left.u_2\right|_{t=0}=f_2$ with $f_1,f_2 \in \Sigma(K)$, then
	\begin{align} \label{cont-prop-sub}
	\|u_1(t)-u_2(t)\|_{\Hc^\theta} \leq \|f_1-f_2\|_{\Hc^\theta} + \|e^{-\im th}f_1-e^{-\im th}f_2\|_{L^8([-\delta,\delta], \Wc^{\beta,4})}
	\end{align}
	for all $|t|\leq \delta$. 
\end{proposition}

Note that since the pair $(8,4)$ is Strichartz-admissible and $\beta > \theta$,~\eqref{cont-prop-sub} implies an estimate with a loss of derivatives
\begin{align} \label{cont-prop-sub-beta} 
	\|u_1(t)-u_2(t)\|_{\Hc^\theta} \leq C \|f_1-f_2\|_{\Hc^\beta}.
\end{align}

As preparation for the proof of Proposition~\ref{prop-almo-sure-lwp}, denote
	\[
	g_f(t):= e^{-\im th}f, \quad v(t):= u(t)-g_f(t).
	\]
	We seek for a solution $v$ to the equation 
	\[
	\im  \partial_t v - hv = \pm |g_f+v|^2 (g_f+v), \quad \left.v\right|_{t=0}=0.
	\]
	To this end, we define for $0<\delta \leq 1$, 
	\[
	X^{\beta}_\delta:=C([-\delta,\delta], \Hc^\beta) \cap L^8([-\delta,\delta],\Wc^{\beta,4}).
	\]
	Our purpose is to prove that the functional
	\[
	\Phi_f(v(t)):= \mp \im \int_0^t e^{-\im (t-\tau)h} (|g_f+v|^2(g_f+v))(\tau) d\tau
	\]
	is a contraction mapping on the ball
	\[
	B_{X^\beta_\delta}(L):= \left\{ v \in X^\beta_\delta : \|v\|_{X^\beta_\delta} \leq L\right\}
	\]
	equipped with the distance
	\[
	d(v_1,v_2) := \|v_1-v_2\|_{X^\beta_\delta},
	\]
	with $L>0$ to be determined later.
	
	Let us start with the following nonlinear estimates.
	\begin{lemma}[\bf Nonlinear estimates] \label{lem-non-est} \mbox{} \\
		There exists $C>0$ such that for any $0<\delta\leq 1$, any $f\in \Hc^\theta$ satisfying 
		\[
		\|g_f\|_{L^8([-1,1],\Wc^{\beta,4})}\leq K
		\] 
		with some $K>0$, and any $v, v_1,v_2 \in X^\beta_\delta$,
		\begin{align*}
		\|\Phi_f(v)\|_{X^\beta_\delta} &\leq C\delta^{1/2} \left( K^3+\|v\|^3_{X^\beta_\delta}\right), \\
		\|\Phi_f(v_1)-\Phi_f(v_2)\|_{X^\beta_\delta} &\leq C\delta^{1/2} \left(K^2+ \|v_1\|^2_{X^\beta_\delta} +\|v\|^2_{X^\beta_\delta}\right) \|v_1-v_2\|_{X^\beta_\delta}.
		\end{align*}
	\end{lemma}
	\begin{proof}
		By Strichartz estimates (see Appendix~\ref{sec:app2}), we have
		\begin{align*}
		\|\Phi_f(v)\|_{X^\beta_\delta} &\leq C \left\|h^{\beta/2}\left(|g_f+v|^2 (g_f+v) \right)\right\|_{L^{8/7}([-\delta,\delta], L^{4/3})}
		\end{align*}
		By the product rule (see Lemma \ref{lem-prod-rule}), we have, using H\"older in time,
		\begin{align*}
		\|\Phi_f(v)\|_{X^\beta_\delta} &\leq C\|g_f+v\|^2_{L^{24/7}([-\delta,\delta], L^4)} \|h^{\beta/2}(g_f+v)\|_{L^{24/7}([-\delta,\delta],L^4)} \\
		&\leq C\delta^{1/2} \|g_f+v\|^2_{L^8([-\delta,\delta],L^4)} \|h^{\beta/2}(g_f+v)\|_{L^8([-\delta,\delta],L^4)}.
		\end{align*}
		By the norm equivalence \eqref{equi-norm}, we infer that
		\begin{align*}
		\|\Phi_f(v)\|_{X^\beta_\delta} &\leq C\delta^{1/2} \|h^{\beta/2}(g_f+v)\|^3_{L^8([-\delta,\delta],L^4)} \\
		&\leq C \delta^{1/2} \left(\|g_f\|^3_{L^8([-\delta,\delta],\Wc^{\beta,4})} + \|v\|^3_{L^8([-\delta,\delta],\Wc^{\beta,4})}\right) \\
		&\leq C \delta^{1/2} \left(\|g_f\|^3_{L^8([-\delta,\delta],\Wc^{\beta,4})} + \|v\|^3_{X^\beta_\delta}\right).
		\end{align*}
		On the other hand, we have
		\[
		\|\Phi_f(v_1)-\Phi_f(v_2)\|_{X^\beta_\delta} \leq C \left\| h^{\beta/2}\left(|g_f+v_1|^2(g_f+v_1)-|g_f+v_2|^2(g_f+v_2)\right)\right\|_{L^{8/7}([-\delta,\delta],L^{4/3})}.
		\]
		By writing
		\begin{align*}
		|g_f+v_1|^2(g_f+v_1)&-|g_f+v_2|^2(g_f+v_2) \\
		&= |g_f+v_1|^2(v_1-v_2) + |g_f+v_2|^2(v_1-v_2) + (g_f+v_1)(g_f+v_2)(\overline{v}_1-\overline{v}_2),
		\end{align*}
		and estimating as above, we have
		\begin{align*}
		\|\Phi_f(v_1)-\Phi_f(v_2)\|_{X^\beta_\delta} &\leq C \|h^{\beta/2}(v_1-v_2)\|_{L^{24/7}([-\delta,\delta], L^4)} \\
		&\quad \times \left(\|h^{\beta/2}(g_f+v_1)\|^2_{L^{24/7}([-\delta,\delta],L^4)} +\|h^{\beta/2}(g_f+v_2)\|^2_{L^{24/7}([-\delta,\delta],L^4)} \right) \\
		&\leq C \delta^{1/2} \|h^{\beta/2}(v_1-v_2)\|_{L^8([-\delta,\delta],L^4)} \\
		&\quad \times \left( \|h^{\beta/2}(g_f+v_1)\|^2_{L^8([-\delta,\delta],L^4)} +\|h^{\beta/2}(g_f+v_2)\|^2_{L^8([-\delta,\delta],L^4)}  \right) \\
		&\leq C\delta^{1/2} \|v_1-v_2\|_{X^\beta_\delta} \left( \|g_f\|^2_{L^8([-\delta,\delta],\Wc^{\beta,4})} + \|v_1\|^2_{X^\beta_\delta} + \|v_2\|^2_{X^\beta_\delta}\right).
		\end{align*}
		Thus we get the desired estimates.
	\end{proof}
	
	Now we may complete the 
	\begin{proof}[Proof of Proposition~\ref{prop-almo-sure-lwp}]
Denote
	\[
	\Sigma(K):= \left\{f \in \Hc^\theta: \|f\|_{\Hc^\theta} \leq K, \|g_f\|_{L^8([-1,1], \Wc^{\beta, 4})} \leq K\right\},
	\]
	where $g_f(t)=e^{-\im th}f$. By Lemma \ref{lem-H-theta} and Lemma \ref{lem-est-expo}, we have
	\begin{align}
	\mu_0(\Sigma^c(K)) &\leq \mu_0(f \in \Hc^\theta : \|f\|_{\Hc^\theta} >K) + \mu_0(f\in \Hc^\theta : \|g_f\|_{L^8([-1,1],\Wc^{\beta,4})} >K) \nonumber \\
	&\leq Ce^{-cK^2}. \label{mu-0-Sigma-K}
	\end{align}
	By Lemma \ref{lem-non-est}, there exists $C>0$ such that for all $0<\delta\leq 1$, all $f\in \Sigma(K)$, and all $v,v_1,v_2 \in B_{X^\beta_\delta}(L)$,
	\begin{align*}
	\|\Phi_f(v)\|_{X^\beta_\delta} &\leq C \delta^{1/2} (K^3+L^3), \\
	d(\Phi_f(v_1), \Phi_f(v_2)) &\leq C \delta^{1/2} (K^2+L^2) d(v_1,v_2).
	\end{align*}
	Pick $L=K$ and $\delta=\nu K^{-4}$ with $\nu>0$ sufficiently small so that $0<\delta \leq 1$ and 
	\[
	2C\delta^{1/2} K^2 = 2C \sqrt{\nu} \leq \frac{1}{2}.
	\]
	We get
	\begin{align*}
	\|\Phi_f(v)\|_{X^\beta_\delta} &\leq K, \\
	d(\Phi_f(v_1),\Phi_f(v_2)) &\leq \frac{1}{2} d(v_1,v_2),
	\end{align*}
	hence $\Phi_f$ is a contraction $\left(B_{X^\beta_\delta}(K),d\right)$. In particular, for each $f\in \Sigma(K)$, there exists a unique solution to \eqref{eq:intro NLS} with initial datum $\left.u\right|_{t=0}=f$ satisfying
	\[
	u(t)-e^{-\im th}f \in C([-\delta,\delta], \Hc^\beta) \cap L^8([-\delta,\delta], \Wc^{\beta,4}).
	\]
	Moreover, for all $f\in \Sigma(K)$, we have $\|f\|_{\Hc^\theta} \leq K$ and the corresponding solution to \eqref{eq:intro NLS} satisfies
	\begin{align*}
	\|u(t)\|_{\Hc^\theta} &\leq \|e^{-\im th} f\|_{\Hc^\theta} +\|u(t)-e^{-\im th} f\|_{\Hc^\theta} \\
	&\leq \|f\|_{\Hc^\theta}+\|u(t)-e^{-\im th}f\|_{\Hc^\beta} \\
	&\leq 2K, \quad \forall |t| \leq \delta
	\end{align*}
	which is \eqref{est-solu-sub}. 
	
	Let us prove \eqref{cont-prop-sub}. We have
	\begin{align}
	\|u_1(t)-u_2(t)\|_{\Hc^\theta} &\leq \|e^{-\im th}f_1-e^{-\im th}f_2\|_{\Hc^\theta} + \|v_1(t)-v_2(t)\|_{\Hc^\theta} \nonumber\\
	&\leq \|f_1-f_2\|_{\Hc^\theta} +\|v_1(t)-v_2(t)\|_{\Hc^\beta}, \label{cont-prop-sub-prof}
	\end{align}
	where
	\[
	v_1(t):= u_1(t)-e^{-\im th}f_1, \quad v_2(t):= u_2(t)-e^{-\im th}f_2.
	\]
	Using Duhamel's formula, we estimate as above to get
	\begin{align*}
	\|v_1-v_2\|_{X^\beta_\delta} &\leq \||g_{f_1}+v_1|^2 (g_{f_1}+v_1)-|g_{f_2}+v_2|^2(g_{f_2}+v_2)\|_{L^{8/7}([-\delta,\delta], \Wc^{\beta,4/3})} \\
	&\leq C\delta^{1/2} \Big(\|g_{f_1}\|^2_{L^8([-\delta,\delta], \Wc^{\beta,4})}+ \|g_{f_2}\|^2_{L^8([-\delta,\delta], \Wc^{\beta,4})} + \|v_1\|^2_{X^\beta_\delta} + \|v_2\|^2_{X^{\beta}_\delta}\Big) \\
	&\quad \times \Big(\|g_{f_1}-g_{f_2}\|_{L^8([-\delta,\delta],\Wc^{\beta,4})} + \|v_1-v_2\|_{X^\beta_\delta}\Big).
	\end{align*}
	Since $f_1, f_2 \in \Sigma(K)$, we have 
	\[
	\|g_{f_1}\|_{L^8([-1,1],\Wc^{\beta,4})}, ~\|g_{f_2}\|_{L^8([-1,1], \Wc^{\beta,4})},~\|v_1\|_{X^\beta_\delta}, ~ \|v_2\|_{X^\beta_\delta} \leq K,
	\]
	hence
	\begin{align*}
	\|v_1-v_2\|_{X^\beta_\delta} &\leq \||g_{f_1}+v_1|^2 (g_{f_1}+v_1)-|g_{f_2}+v_2|^2(g_{f_2}+v_2)\|_{L^{8/7}([-\delta,\delta], \Wc^{\beta,4/3})} \\
	&\leq C\delta^{1/2} K^2 \Big(\|g_{f_1}-g_{f_2}\|_{L^8([-\delta,\delta],\Wc^{\beta,4})} + \|v_1-v_2\|_{X^\beta_\delta}\Big).
	\end{align*}
	By choosing $\delta>0$ sufficiently small so that $C\delta^{1/2} K^2\leq \frac{1}{2}$, we get
	\[
	\|v_1-v_2\|_{X^\beta_\delta} \leq \|g_{f_1}-g_{f_2}\|_{L^8([-\delta,\delta], \Wc^{\beta,4})}.
	\]
	This together with \eqref{cont-prop-sub-prof} imply \eqref{cont-prop-sub}. The proof is complete.
\end{proof}

\subsection{Measure invariance and global well-posedness}	
\label{sec:GloSub}

\begin{theorem}[\bf Almost sure global well-posedness, $1<s\leq 2$] \label{theo:almo-gwp-sub} \mbox{} \\
	Let $1<s\leq2$, $V$ satisfy Assumption \ref{assu-V}, and $\theta<\frac{1}{2}-\frac{1}{s}$. Assume in addition that $s>\frac{8}{5}$ for the focusing nonlinearity. Then there exists a set $\Sigma \subset \Hc^\theta$ satisfying $\mu(\Sigma)=1$ such that for any $f\in \Sigma$, the corresponding solution to \eqref{eq:intro NLS} with initial data $\left.u\right|_{t=0}=f$ exists globally in time and satisfies
	\[
	\|u(t)\|_{\Hc^\theta} \leq C \left(\omega(f)+\log^{\frac{1}{2}}(1+|t|)\right), \quad \forall t\in \R
	\]
	for some constant $\omega(f)>0$ depending on $f$ and some universal constant $C>0$. Moreover, the Gibbs measure $\mu$ is invariant under the flow of \eqref{eq:intro NLS}.
\end{theorem}

The proof of this theorem follows by the same line of reasoning as in Theorem \ref{theo:almo-gwp-supe}. However, due to the lack of deterministic local well-posedness, we need to proceed differently.

\begin{lemma}[\bf Uniform estimate for the approximate flow, $1<s\leq 2$]  \label{lem:globflow2} \mbox{}\\
	Let $s, \theta$ be as in the previous statement and take $\theta_1$ satisfying
	\[
	\theta<\theta_1<\frac{1}{2}-\frac{1}{s}.
	\] 
	Then for all $\Lambda \geq \lambda_1$ and $T,\vareps>0$. There exist $\Sigma_{\Lambda, T,\vareps} \subset \Hc^{\theta_1}$ and $C>0$ independent of $\Lambda, T,\vareps$ such that:
	
	\noindent (1) $\mu_\Lambda(\Sigma^c_{\Lambda,T,\vareps})\leq C\vareps$.
	
	\noindent (2) For $f\in \Sigma_{\Lambda, T,\vareps}$, there exists a unique solution to \eqref{NLS-appro} on $[-T,T]$ satisfying 
	\begin{align} \label{u-Lambda-T-sub}
	\|u_\Lambda(t)\|_{\Hc^{\theta_1}} \leq C \left(\log\frac{T}{\vareps}\right)^{1/2}, \quad \forall |t|\leq T.
	\end{align}
\end{lemma}

\begin{proof}
	Pick $0\leq \beta <\beta_1<\frac{s-1}{2s}$ so that
	\begin{align} \label{eta}
	\theta_1-\theta = \beta_1-\beta =\eta>0.
	\end{align}
	For $K>0$, we denote
	\[
	\Sigma_{\theta_1, \beta_1} (K):=\left\{ f\in \Hc^{\theta_1} : \|f\|_{\Hc^{\theta_1}} \leq K, \|e^{-\im th} f\|_{L^8([-1,1],\Wc^{\beta_1,4})} \leq K\right\}.
	\]
	Using \eqref{eq:est-Q-Lamb}, the probabilistic local well-posedness given in Proposition \ref{prop-almo-sure-lwp} applies to \eqref{NLS-appro} and we have that for $f\in \Sigma_{\theta_1,\beta_1}(K)$, there exists a unique solution to \eqref{NLS-appro} satisfying
	\begin{align} \label{solu-Lambda-sub}
	\|u_\Lambda(t)\|_{\Hc^{\theta_1}} \leq 2K, \quad \forall |t|\leq \delta
	\end{align}
	with
	\begin{align} \label{defi-delta-sub}
	\delta = \nu (K+1)^{-4},
	\end{align}
	where $\nu>0$ is a small constant independent of $\Lambda, K$. As for the super-harmonic case, here we choose $\delta$ a bit smaller than $\nu K^{-4}$ as in Proposition \ref{prop-almo-sure-lwp} which is useful for the iteration process (see Lemma \ref{lem:appflow2}). 
	
	Denote $J= \left[\frac{T}{\delta}\right]$ and set 
	\begin{align} \label{eq:def Sigma 2} 
	\Sigma_{\Lambda, K, \theta_1, \beta_1} := \bigcap_{j=-J}^J \Phi_\Lambda(-j\delta) (\Sigma_{\theta_1,\beta_1}(K)),
	\end{align}
	where $\Phi_\Lambda$ is the solution map of \eqref{NLS-appro}. Note that the set $\Sigma_{\Lambda, K, \theta_1,\beta_1}$ contains all initial data $f$ such that the corresponding solutions $u_\Lambda(t)$ of \eqref{NLS-appro} satisfy 
	\begin{align} \label{est-u-Lambda-sub}
	\|u_\Lambda(t)\|_{\Hc^{\theta_1}}\leq 2K, \quad \forall |t|\leq T.
	\end{align} 
	In fact, for $|t|\leq T$, there exist an integer $j$ and $\delta_1 \in (-\delta, \delta)$ such that $t=j\delta + \delta_1$. Thus
	\[
	u_\Lambda(t) = \Phi_\Lambda(\delta_1) \Phi_\Lambda(j \delta) f.
	\]
	Since $f \in \Phi_{\Lambda}(-j\delta)(\Sigma_{\theta_1,\beta_1}(K))$, we have $\Phi_\Lambda(j\delta) f \in \Sigma_{\theta_1,\beta_1}(K)$ and the observation follows from \eqref{solu-Lambda-sub}. 
	
	Since $\mu_\Lambda$ is invariant under the flow of \eqref{NLS-appro}, we have
	\begin{align*}
	\mu_\Lambda(\Sigma_{\Lambda, K,\theta_1,\beta_1}^c) &= \mu_\Lambda \Big( \Big( \bigcap_{j=-J}^J \Phi_\Lambda(-j\delta)(\Sigma_{\theta_1,\beta_1}(K))\Big)^c \Big) \\
	&\leq \sum_{j=-J}^J \mu_\Lambda(\Phi_\Lambda(-j\delta)(\Sigma^c_{\theta_1,\beta_1}(K))) \\
	&=\sum_{j=-J}^J \mu_\Lambda(\Sigma^c_{\theta_1, \beta_1}(K)) \\
	&\leq 2\frac{T}{\delta} \mu_\Lambda(\Sigma^c_{\theta_1, \beta_1}(K)).
	\end{align*}
	Thanks to Proposition \ref{prop-almo-sure-lwp} we have
	\begin{align*}
	\mu_\Lambda(\Sigma^c_{\theta_1,\beta_1}(K)) &=\int \mathds{1}_{\Sigma^c_{\theta_1,\beta_1}(K)} d\mu_\Lambda(u) \nonumber\\
	&= \int \mathds{1}_{\Sigma^c_{\theta_1,\beta_1}(K)} \frac{1}{Z^{\Lambda}} G_\Lambda(u) d\mu_0(u) \nonumber\\
	&\leq \frac{1}{Z^{\Lambda}} \|G_\Lambda(u)\|_{L^2(d\mu_0)} \left(\mu_0(\Sigma^c_{\theta_1,\beta_1}(K))\right)^{1/2} \nonumber\\
	&\leq Ce^{-cK^2}, 
	\end{align*}
	where we have used $G_\Lambda(u) \in L^2(d\mu_0)$ and $Z^\Lambda \geq C>0$ uniformly in $\Lambda$ and \eqref{mu-0-Sigma-K}. In particular, we obtain
	\[
	\mu_\Lambda(\Sigma^c_{\Lambda, K,\theta_1,\beta_1}) \leq \frac{2C}{\nu}T(K+1)^4 e^{-cK^2}  \leq CTe^{-cK^2}
	\]
	for some constants $C,c>0$. By choosing $K$ as in \eqref{choi-K} and setting $\Sigma_{\Lambda,T,\vareps} = \Sigma_{\Lambda, K, \theta_1, \beta_1}$, we get the desired estimates.
\end{proof}

\begin{lemma}[\textbf{Comparison between approximate and exact flows, $1<s\leq2$}]\label{lem:appflow2}\mbox{}\\
	Let $\Sigma_{\Lambda, T,\vareps}$ be as in the previous lemma. Then for any $f\in \Sigma_{\Lambda,T,\vareps}$, there exists a unique solution to~\eqref{eq:intro NLS} with initial data $\left.u\right|_{t=0}=f$ satisfying
	\begin{align} \label{est-diff-sub}
	\|u(t)-u_\Lambda(t)\|_{\Hc^\theta} \leq C(T,\vareps)\Lambda^{-\eta/2}, \quad \forall |t|\leq T
	\end{align}
	for all $\Lambda$ sufficiently large and some constant $C(T,\vareps)>0$ independent of $\Lambda$, where $\eta$ is as in \eqref{eta}. In particular, there exist $\Sigma_{T,\vareps} \subset \Hc^\theta$ and $C>0$ independent of $T, \vareps$ such that:
	
	\noindent (1) $\mu(\Sigma^c_{T,\vareps}) \leq C\vareps$.
	
	\noindent (2) For $f\in \Sigma_{T,\vareps}$, there exists a unique solution to \eqref{eq:intro NLS} with initial data $\left.u\right|_{t=0}=f$ on $[-T,T]$ satisfying
	\begin{align}\label{est-u-T-sub}
	\|u(t)\|_{\Hc^\theta} \leq C \left(\log \frac{T}{\vareps}\right)^{1/2}, \quad \forall |t|\leq T.
	\end{align}
\end{lemma}

\begin{proof}
	{\bf Estimating the difference.} We first study the difference between $u$ and $u_\Lambda$ on $I_0=[-\delta,\delta]$, where $\delta$ is as in \eqref{defi-delta-sub}. To do this, we denote $g_f(t)=e^{-\im th}f$ and $v_\Lambda := Q_\Lambda u_\Lambda$. By Duhamel's formula, we write for $t\in I_0$,
	\[
	u(t) = g_f(t) + v(t), \quad v(t):= \mp \im \int_0^t e^{-\im (t-\tau)h} (|g_f+v|^2(g_f+v))(\tau)d\tau
	\]
	and
	\[
	v_\Lambda(t) =Q_\Lambda g_f(t) +w_\Lambda(t), \quad w_\Lambda(t) := \mp \im \int_0^t e^{-\im (t-\tau)h} (Q^2_\Lambda(|Q_\Lambda g_f +w_\Lambda|^2(Q_\Lambda g_f+w_\Lambda)))(\tau)d\tau.
	\]
	Denote
	\[
	X^\beta(I_0):= C(I_0, \Hc^{\beta}) \cap L^8(I_0,\Wc^{\beta,4}).
	\]
	By Strichartz estimates (cf Appendix~\ref{sec:app2}), we have
	\begin{align*}
	\|v-w_\Lambda\|_{X^\beta(I_0)} \lesssim \||g_f+v|^2(g_f+v) - Q^2_\Lambda (|Q_\Lambda g_f+ w_\Lambda|^2 (Q_\Lambda g_f+w_\Lambda))\|_{L^{8/7}(I_0, \Wc^{\beta,4/3})}.
	\end{align*}
	We write
	\begin{align*}
	|g_f+v|^2(g_f+v) &- Q^2_\Lambda(|Q_\Lambda g_f+w_\Lambda|^2(Q_\Lambda g_f+w_\Lambda))\\
	&=|g_f+v|^2(g_f+v)-|Q_\Lambda g_f + w_\Lambda|^2(Q_\Lambda g_f+ w_\Lambda) \\
	&\quad + (\id-Q^2_\Lambda)(|Q_\Lambda g_f + w_\Lambda|^2(Q_\Lambda g_f+ w_\Lambda)). 
	\end{align*}
	Since $f\in \Sigma_{\theta_1, \beta_1}(K)$, by Proposition \ref{prop-almo-sure-lwp} and \eqref{eq:est-Q-Lamb}, we have
	\begin{align} \label{est-v-w-I0}
	\|w_\Lambda\|_{X^{\beta_1}(I_0)} \leq CK.
	\end{align}
	Estimating as in the proof of Lemma \ref{lem-non-est} yields
	\begin{align*}
	\|(\id -Q^2_\Lambda)(|Q_\Lambda g_f+w_\Lambda|^2&(Q_\Lambda g_f+w_\Lambda))\|_{L^{8/7}(I_0, \Wc^{\beta,4/3})} \\
	&\leq C\Lambda^{(\beta-\beta_1)/2} \||Q_\Lambda g_f+w_\Lambda|^2(Q_\Lambda g_f+w_\Lambda)\|_{L^{8/7}(I_0, \Wc^{\beta_1,4/3})}\\
	&\leq C\delta^{1/2} \Lambda^{-\eta/2} \Big(\|Q_\Lambda g_f\|^3_{L^8(I_0,\Wc^{\beta_1,4})} + \|w_\Lambda\|^3_{L^8(I_0,\Wc^{\beta_1,4})} \Big) \\
	&\leq C\delta^{1/2} \Lambda^{-\eta/2} \Big(\|g_f\|^3_{L^8(I_0,\Wc^{\beta_1,4})} + \|w_\Lambda\|^3_{X^{\beta_1}(I_0)} \Big) \\
	&\leq C\delta^{1/2} K^3 \Lambda^{-\eta/2}.
	\end{align*}
	Since $\|f\|_{\Hc^\theta} \leq \|f\|_{\Hc^{\theta_1}} \leq K$ and $\|g_f\|_{L^8([-1,1], \Wc^{\beta,4})} \leq \|g_f\|_{L^8([-1,1],\Wc^{\beta_1,4})} \leq K$ due to $f \in \Sigma_{\theta_1,\beta_1}(K)$, Proposition \ref{prop-almo-sure-lwp} and \eqref{eq:est-Q-Lamb} give
	\[
	\|v\|_{X^\beta(I_0)}, ~ \|w_\Lambda\|_{X^\beta(I_0)} \leq K.
	\]
	By writing
	\begin{align*}
	|g_f+v|^2(g_f+v) &-|Q_\Lambda g_f+w_\Lambda|^2(Q_\Lambda g_f+w_\Lambda) \\
	&= (g_f-Q_\Lambda g_f + v-w_\Lambda) R_2(g_f, Q_\Lambda g_f, v, w_\Lambda), 
	\end{align*}
	where $R_2$ is a homogeneous polynomial of degree $2$ (omitting the complex conjugates for simplicity) and estimating as in Lemma \ref{lem-non-est}, we obtain
	\begin{align*}
	\|Q^2_\Lambda(|g_f+v|^2&(g_f+v) -|Q_\Lambda g_f+w_\Lambda|^2(Q_\Lambda g_f+w_\Lambda))\|_{L^{8/7}(I_0, \Wc^{\beta,4/3})} \\
	&\leq C\delta^{1/2} \left(\|g_f-Q_\Lambda g_f\|_{L^8(I_0, \Wc^{\beta,4})} + \|v-w_\Lambda\|_{L^8(I_0,\Wc^{\beta,4})} \right) \\
	&\quad \times \left(\|g_f\|^2_{L^8(I_0, \Wc^{\beta,4})} + \|Q_\Lambda g_f\|^2_{L^8(I_0, \Wc^{\beta,4})} + \|v\|^2_{L^8(I_0,\Wc^{\beta,4})} +\|w_\Lambda\|^2_{L^8(I_0,\Wc^{\beta,4})}\right) \\
	&\leq C\delta^{1/2} \left(\Lambda^{(\beta-\beta_1)/2} \|g_f\|_{L^8(I_0,\Wc^{\beta_1,4})} + \|v-w_\Lambda\|_{X^\beta(I_0)}\right) \\
	&\quad \times \left(\|g_f\|^2_{L^8(I_0, \Wc^{\beta,4})} +\|v\|^2_{X^{\beta}(I_0)} +\|w_\Lambda\|^2_{X^{\beta}(I_0)}\right) \\
	&\leq C\delta^{1/2} K^2 \left( K \Lambda^{-\eta/2} + \|v-w_\Lambda\|_{X^{\beta}(I_0)}\right).
	\end{align*}
	Thus we obtain
	\[
	\|v-w_\Lambda\|_{X^{\beta}(I_0)} \leq C \delta^{1/2} K^3 \Lambda^{-\eta/2} + C\delta^{1/2} K^2 \|v-w_\Lambda\|_{X^{\beta}(I_0)}.
	\]
	By the choice of $\delta$ (see \eqref{choi-K}), we get
	\[
	\|v-w_\Lambda\|_{X^{\beta}(I_0)} \leq K \Lambda^{-\eta/2}.
	\]
	From this and \eqref{est-u-Lambda-sub}, we deduce
	\begin{align}
	\|u-u_\Lambda\|_{L^\infty(I_0,\Hc^\theta)} &\leq \|u-v_\Lambda\|_{L^\infty(I_0,\Hc^\theta)} + \|v_\Lambda - u_\Lambda\|_{L^\infty(I_0,\Hc^\theta)} \nonumber\\
	&\leq\|g_f-Q_\Lambda g_f\|_{L^\infty(I_0,\Hc^\theta)} + \|v-w_\Lambda\|_{L^\infty(I_0, \Hc^\theta)} + \|Q_\Lambda u_\Lambda - u_\Lambda\|_{L^\infty(I_0,\Hc^\theta)} \nonumber\\
	&\leq \|f-Q_\Lambda f\|_{\Hc^\theta} + \|v-w_\Lambda\|_{L^\infty(I_0,\Hc^\beta)} + \|Q_\Lambda u_\Lambda - u_\Lambda\|_{L^\infty(I_0,\Hc^\theta)} \nonumber\\
	&\leq C \Lambda^{(\theta-\theta_1)/2}\|f\|_{\Hc^{\theta_1}} + \|v-w_\Lambda\|_{X^{\beta}(I_0)} + C\Lambda^{(\theta-\theta_1)/2}\|u_\Lambda\|_{L^\infty(I_0, \Hc^{\theta_1})} \nonumber\\
	&\leq C K \Lambda^{-\eta/2}.\label{est-diff-I0}
	\end{align}
	
	\medskip
	
	 \noindent\textbf{Iterating in time.} We now iterate this argument to other sub-intervals of $[-T,T]$. The next iteration is on $I_1:=[0,2\delta]$. To this end, we claim that, taking $\Lambda$ sufficiently large,
	\begin{align} \label{ite-I1}
	\|u(\delta)\|_{\Hc^{\theta}} \leq K+1, \quad \|e^{-\im th} u(\delta)\|_{L^8([-1,1],\Wc^{\beta,4})} \leq K+1.
	\end{align}
	In fact, the first observation follows directly from \eqref{est-diff-I0} and 
	\[
	\|u_\Lambda(\delta)\|_{\Hc^\theta} = \|\Phi_\Lambda(\delta)f\|_{\Hc^\theta}\leq \|\Phi_\Lambda(\delta)f\|_{\Hc^{\theta_1}}\leq K
	\]
	since $u_\Lambda(\delta) = \Phi_\Lambda(\delta) f \in \Sigma_{\theta_1,\beta_1}(K)$ (see \eqref{eq:def Sigma 2}). We turn to the second inequlity in~\eqref{ite-I1}, using Strichartz estimates and~\eqref{eta}
	\begin{align}\label{eq:ite-I1i}
	\|e^{-\im th} u(\delta)\|_{L^8([-1,1], \Wc^{\beta,4})} & \leq \|e^{-\im th} u_\Lambda(\delta)\|_{L^8([-1,1], \Wc^{\beta,4})} + \|e^{-\im th}(u(\delta)-u_\Lambda(\delta))\|_{L^8([-1,1],\Wc^{\beta,4})}\nonumber \\
	&\leq \|e^{-\im th} u_\Lambda(\delta)\|_{L^8([-1,1], \Wc^{\beta_1,4})} + C\|u(\delta)-u_\Lambda(\delta)\|_{\Hc^{\beta}}\nonumber \\
	&\leq K + C \Lambda^{-\eta/2}\|u-u_\Lambda\|_{L^\infty(I_0,\Hc^{\beta_1})}.
	\end{align}
	Next, using Duhamel's formulas
	\[	
	u(t) = e^{-\im th}f \mp \im \int_0^t e^{-\im (t-\tau)h} (|u|^2u)(\tau)d\tau
	\]
	and
	\[
	u_\Lambda(t) = e^{-\im th}f \mp \im \int_0^t e^{-\im (t-\tau)h}(Q_\Lambda(|Q_\Lambda u_\Lambda|^2 Q_\Lambda u_\Lambda))(\tau)d\tau,
	\]
	we have
	\begin{align}\label{eq:ite-I1ii}
	\|u-u_\Lambda\|_{L^\infty(I_0,\Hc^{\beta_1})} &\leq C\||u|^2 u- Q_\Lambda(|Q_\Lambda u_\Lambda|^2 Q_\Lambda u_\Lambda)\|_{L^{8/7}(I_0, \Wc^{\beta_1,4/3})} \nonumber \\
	&\leq C\delta^{1/2} \left(\|u\|^3_{L^8(I_0,\Wc^{\beta_1,4})} + \|u_\Lambda\|^3_{L^8(I_0,\Wc^{\beta_1,4})}\right) \\
	&\leq C\delta^{1/2} K^3 \nonumber\\
	&\leq K
	\end{align}
	by the choice of $\delta$. Here we have used the fact that
	\begin{align*}
	\|u\|_{L^8(I_0, \Wc^{\beta_1,4})} &\leq \|g_f\|_{L^8(I_0, \Wc^{\beta_1,4})} + \|u-g_f\|_{L^8(I_0,\Wc^{\beta_1,4})} \leq 2K, \\
	\|u_\Lambda\|_{L^8(I_0, \Wc^{\beta_1, 4})} &\leq \|g_f\|_{L^8(I_0, \Wc^{\beta_1,4})} + \|u_\Lambda-g_f\|_{L^8(I_0,\Wc^{\beta_1,4})} \leq 2K,
	\end{align*}
	where the second line follows from the probabilistic local well-posedness. Combining~\eqref{eq:ite-I1i} with~\eqref{eq:ite-I1ii} yields the second inequality in~\eqref{ite-I1}. Note that \eqref{ite-I1} is the reason why we choose $\delta$ as in \eqref{defi-delta-sub}.
	
	Thanks to \eqref{ite-I1} and $u_\Lambda(\delta)=\Phi_\Lambda(\delta) f \in \Sigma_{\theta_1, \beta_1}(K)$, we can repeat the above argument using the following Duhamel formulas for $t\in [0,2\delta]$,
	\begin{align*}
	u(t) &= g_\delta(t) + v(t), \quad v(t) := \mp \im \int_\delta^t e^{-\im (t-\tau) h} (|g_\delta + v|^2(g_\delta +v))(\tau) d\tau \\
	v_\Lambda(t) &= Q_\Lambda g_{\Lambda,\delta}(t) + w_\Lambda(t), \\ 
	w_\Lambda(t):&= \mp \im \int_\delta^t e^{-\im (t-\tau) h} (Q_\Lambda^2(|Q_\Lambda g_{\Lambda,\delta} + w_\Lambda|^2(Q_\Lambda g_{\Lambda,\delta} + w_\Lambda)))(\tau) d\tau
	\end{align*}
	with $g_\delta(t):= e^{-\im (t-\delta)h} u(\delta)$ and $g_{\Lambda,\delta}(t):=e^{-\im (t-\delta)h} u_\Lambda(\delta)$ and get
	\[
	\|v-w_\Lambda\|_{X^\beta(I_1)} \leq K \Lambda^{-\eta/2}.
	\]
	The same reasoning as in \eqref{est-diff-I0} yields
	\[
	\|u-u_\Lambda\|_{L^\infty(I_1,\Hc^{\theta})} \lesssim CK \Lambda^{-\eta/2}.
	\]
	Iterating this bound $\left[\frac{T}{\delta}\right]$ many times and taking into account the choice of $K$ (see \eqref{choi-K}), we prove \eqref{est-diff-sub} .
	
	\noindent {\bf Globalizing the flow.} Finally,~\eqref{est-u-T-sub} follows from~\eqref{est-u-Lambda-sub} and~\eqref{est-diff-sub} exactly as in the proof of Lemma \ref{lem:globflow1}. We omit the details. 
\end{proof}

We are now able to prove Theorem \ref{theo:almo-gwp-sub}.

\begin{proof}[Proof of Theorem \ref{theo:almo-gwp-sub}]
	The almost sure global well-posedness and the growth in time are proved by the same argument as in the proof of Theorem \ref{theo:almo-gwp-supe}. Since the continuity of the solution flow does not depend solely on $\Hc^\theta$ norm of initial data (see \eqref{cont-prop-sub} and \eqref{cont-prop-sub-beta}), the argument given in the super-harmonic case should be modified by choosing a suitable ball in higher Sobolev spaces. We present here another approach using an equivalent characterization of measure invariance (see e.g., \cite[Theorem 6.5]{NOBS}), i.e., for all $F\in L^1(\Hc^\theta, d\mu)$,
	\begin{align} \label{inva-prof-equi}
	\int F(\Phi(t) u) d\mu(u) = \int F(u) d\mu(u), \quad \forall t\in \R.
	\end{align}
	By iteration, it suffices to show \eqref{inva-prof-equi} for $t>0$ small. In addition, by a density argument, the problem is reduced to show \eqref{inva-prof-equi} for $F$ bounded and continuous. 
	
	Now fix $t>0$ small and let $\vareps>0$. We write for $\Lambda \geq \lambda_1$,
	\begin{align*}
	\left|\int F(\Phi(t)u)d\mu(u) - \int F(u) d\mu(u) \right| &\leq \left| \int F(\Phi(t)u) d\mu(u) - \int F(\Phi(t) u) d\mu_\Lambda(u)\right| \\
	&\quad + \left| \int F(\Phi(t)u) d\mu_\Lambda(u) - \int F(\Phi_\Lambda(t) u) d\mu_\Lambda(u)\right| \\
	&\quad + \left| \int F(\Phi_\Lambda(t) u) d\mu_\Lambda(u) - \int F(u) d\mu_\Lambda(u)\right| \\
	&\quad + \left| \int F(u) d\mu_\Lambda(u) - \int F(u) d\mu(u)\right| \\
	&=: (\text{I}) + (\text{II}) + (\text{III}) + (\text{IV}).
	\end{align*} 
	Since $\mu_\Lambda \rightharpoonup \mu$ weakly as $\Lambda \to \infty$, the boundedness of $F$ implies
	\[
	(\text{I}) +(\text{IV}) \to 0 \text{ as } \Lambda \to \infty.
	\]
	Since $\mu_\Lambda$ is invariant under the flow map $\Phi_\Lambda(t)$, an equivalent characterization of invariance as in \eqref{inva-prof-equi} yields $(\text{III})=0$. It remains to estimate $(\text{II})$. Pick $\theta<\theta_1<\frac{1}{2}-\frac{1}{s}$ and $0\leq \beta <\beta_1<\frac{s-1}{2s}$ such that $\theta_1-\theta=\beta_1-\beta=\eta>0$. Denote
	\[
	\Sigma_{\theta_1, \beta_1}(K):= \left\{f \in \Hc^{\theta_1} : \|f\|_{\Hc^{\theta_1}} \leq K,  \|e^{-\im th}f\|_{L^8([-1,1],\Wc^{\beta_1,4})} \leq K\right\}.
	\] 
	We estimate
	\begin{align*}
	(\text{II}) &\leq \left| \int_{\Sigma_{\theta_1, \beta_1}(K)} F(\Phi(t)u) - F(\Phi_\Lambda(t) u) d\mu_\Lambda(u)\right| + \left| \int_{\Sigma^c_{\theta_1, \beta_1}(K)} F(\Phi(t)u) - F(\Phi_\Lambda(t) u) d\mu_\Lambda(u)\right| \\
	&=: (\text{II}_1) + (\text{II}_2).
	\end{align*}
	Since $F$ is bounded, we have 
	\begin{align*}
	(\text{II}_2) \leq 2\|F\|_{L^\infty} \mu_\Lambda(\Sigma^c_{\theta_1,\beta_1}(K)) \leq C\|F\|_{L^\infty} \left(\mu_0(\Sigma^c_{\theta_1,\beta_1}(K))\right)^{1/2} \leq C e^{-cK^2} <\frac{\vareps}{2}
	\end{align*}
	provided that $K$ is taken sufficiently large. For such a $K$, the same argument as in the proof of Lemma \ref{lem:globflow2} yields 
	\[
	\|\Phi(t) u - \Phi_\Lambda(t) u\|_{\Hc^\theta} \leq CK \Lambda^{-\eta/2}
	\]
	for all $u \in \Sigma_{\theta_1, \beta_1}(K)$. Note that $t>0$ is taken small here. Since $F$ is continuous, we have
	\[
	\|F(\Phi(t) u) - F(\Phi_\Lambda(t) u)\|_{\Hc^\theta} <\vareps/2
	\]
	for all $u \in \Sigma_{\theta_1, \beta_1}(K)$ provided that $\Lambda$ is chosen large enough. Since $\mu_\Lambda$ is a probability measure, we have $(\text{II}_1) \leq \vareps/2$, hence 
	\begin{align*}
	(\text{II}) \to 0 \text{ as } \Lambda \to \infty.
	\end{align*}
	Collecting the above estimates, we prove the invariance.
\end{proof}

\begin{proof}[Proof of Theorem \ref{theo-Gibbs-meas}]
	It follows from Theorems \ref{theo:almo-gwp-supe} and \ref{theo:almo-gwp-sub}.
\end{proof}

\begin{remark}[Higher non-linearities] \label{rem-meas-inva-sub}
	The method used in this section can be modified to prove the invariance of Gibbs measures associated to \eqref{eq:NLS-general} for any $1<s\leq 2$ and $\frac{4}{s}<\kappa<6$ provided that the Gibbs measure is well-defined. Indeed, we take $0\leq \beta<\frac{\kappa s -4}{2\kappa s}$ and $\theta <\frac{1}{2}-\frac{1}{s}$. For $K>0$, we define
	\[
	\Sigma(K):= \left\{f \in \Hc^\theta: \|f\|_{\Hc^\theta} \leq K, \quad \|g_f\|_{L^{\frac{4\kappa}{\kappa-2}}([-1,1], \Wc^{\beta,\kappa})} \leq K \right\},
	\] 
	where $g_f(t)=e^{-\im th} f$ and $\left(\frac{4\kappa}{\kappa-2}, \kappa\right)$ is a Strichartz-admissible pair. By the choice of $\beta$, we have (see Lemma \ref{lem-est-expo})
	\[
	\mu_0(\Sigma^c(K)) \leq Ce^{-cK^2}.
	\]
	Thanks to the following nonlinear estimates
	\begin{align*}
		\||g_f+v|^{\kappa-2}(g_f+v)\|_{L^{\frac{4\kappa}{3\kappa+2}}_\delta \Wc^{\beta, \frac{\kappa}{\kappa-1}}} &\leq \|g_f+v\|^{\kappa-2}_{L^{\frac{2\kappa(\kappa-2)}{\kappa+2}}_\delta L^\kappa} \|g_f+v\|_{L^{\frac{4\kappa}{\kappa-2}}_\delta \Wc^{\beta,\kappa}} \\
		&\lesssim \delta^{\frac{6-\kappa}{4}} \|g_f+v\|^{\kappa-2}_{L^{\frac{4\kappa}{\kappa-2}}_\delta L^\kappa} \|g_f+v\|_{L^{\frac{4\kappa}{\kappa-2}}_\delta \Wc^{\beta,\kappa}} \\
		&\lesssim \delta^{\frac{6-\kappa}{4}} \left(\|g_f\|^{\kappa-1}_{L^{\frac{4\kappa}{\kappa-2}}_\delta \Wc^{\beta,\kappa}} + \|v\|^{\kappa-1}_{L^{\frac{4\kappa}{\kappa-2}}_\delta \Wc^{\beta,\kappa}} \right)
	\end{align*}
	and
	\begin{align*}
		&\||g_f+v_1|^{\kappa-2}(g_f+v_1) - |g_f+v_2|^{\kappa-2}(g_f+v_2)\|_{L^{\frac{4\kappa}{3\kappa+2}}_\delta \Wc^{\beta, \frac{\kappa}{\kappa-1}}} \\
		&\lesssim \delta^{\frac{6-\kappa}{4}} \left(\|g_f\|^{\kappa-2}_{L^{\frac{4\kappa}{\kappa-2}}_\delta \Wc^{\beta,\kappa}} + \|v_1\|^{\kappa-2}_{L^{\frac{4\kappa}{\kappa-2}}_\delta \Wc^{\beta,\kappa}} + \|v_1\|^{\kappa-2}_{L^{\frac{4\kappa}{\kappa-2}}_\delta \Wc^{\beta,\kappa}} \right) \|v_1-v_2\|_{L^{\frac{4\kappa}{\kappa-2}}_\delta \Wc^{\beta,\kappa}},
	\end{align*}
	we can repeat the same argument as in the proof of Proposition \ref{prop-almo-sure-lwp} to show the almost sure local well-posedness for \eqref{eq:NLS-general}, where $\left(\frac{4\kappa}{3\kappa+2}, \frac{\kappa}{\kappa-1}\right)$ is the dual pair of $\left(\frac{4\kappa}{\kappa-2}, \kappa\right)$. Note that the condition $\kappa<6$ is needed to close the contraction mapping argument. In particular, for all $f\in \Sigma(K)$, there exist $\delta \sim K^{-\frac{4(\kappa-2)}{6-\kappa}}$ and a unique solution to \eqref{eq:NLS-general} with initial data $\left. u\right|_{t=0}=f$ satisfying
	\[
	u(t)-e^{-\im t h} f \in C([-\delta,\delta], \Hc^\beta) \cap L^{\frac{4\kappa}{\kappa-2}}([-\delta, \delta], \Wc^{\beta,\kappa}).
	\] 
	Once local solutions exist, a straightforward modification of the above argument yields the invariance of Gibbs measures, based on the fact that $\left(\frac{4\kappa}{\kappa-2}, \kappa\right)$ is a Strichartz-admissible pair. The latter constraint is in fact the main restriction that sets the affordable non-linearities. \hfill $\diamond$ 
\end{remark}

\section{Canonical measures}
\label{sec:cano measures}
\setcounter{equation}{0}
We aim at constructing, and proving the invariance of, Gibbs measures conditioned on mass. Towards this purpose, we first define the Gaussian measure with a fixed renormalized mass in Section \ref{sec:Gaus fix renor}. We then define the Gibbs measures conditioned on mass in Section \ref{sec:Gibbs cond mass}. Finally, in Section \ref{sec:invari Gibbs fix mass}, we prove that these measures are invariant under the dynamics of \eqref{eq:intro NLS}. 

\subsection{Gaussian measure with a fixed renormalized mass}
\label{sec:Gaus fix renor}
Our purpose in this section is to give a rigorous definition of Gaussian measure conditioned on mass. There is a difference between the cases $s>2$ and $1<s\leq 2$ in that the latter the mass is infinite almost surely. We however treat the two cases on the same footing by conditioning on the renormalized mass, although this is slightly redundant for $s>2$ (where the mass is finite almost surely). To do this, we shall define the Gaussian measure with a fixed renormalized mass which is formally given by
\begin{align}\label{eq:Gaus-fix-mass}
d\mu_{0}^m(u) ~``="~ \frac{\mathds{1}_{\{\Mcal(u)=m\}}}{\mu_0(\Mcal(u)=m)} d\mu_0(u),
\end{align}
where $m\in \R$, $\Mcal(u)=\|u\|^2_{L^2}-\scal{\|u\|^2_{L^2}}_{\mu_0}$ is the renormalized mass (see Lemma \ref{lem-reno-mass}), and $\mu_0$ is the Gaussian measure. When $s>2$, since the expectation of the mass with respect to the Gaussian measure is finite, 
$$\scal{\|u\|^2_{L^2}}_{\mu_0} = {\Tr}[h^{-1}]<\infty,$$
we obtain a Gaussian measure with a fixed mass $m+{\Tr}[h^{-1}]>0$. 

The formulation \eqref{eq:Gaus-fix-mass} is purely formal because we essentially have that 
$$\mu_0(\Mcal(u)=m)=0.$$ 
Inspired by an idea of Oh and Quastel \cite{OhQua-13}, we will define the above measure as a limit, as $\vareps \rightarrow 0^+$, of the constrained measure
\[
d\mu_{0}^{m,\vareps}(u) = \frac{1}{Z_{0}^{m,\vareps}} \mathds{1}_{\{m-\vareps<\Mcal(u)<m+\vareps\}} d\mu_0(u),
\]
where
\[
Z_{0}^{m,\vareps} = \mu_0 (m-\vareps<\Mcal(u)<m+\vareps).
\]
Note that for each $\vareps>0$, the measure $\mu_0^{m,\vareps}$ is well-defined a priori, however, since the normalized constant $Z_0^{m,\vareps}$ converges to zero as $\vareps \to 0^+$, some care is needed to justify the meaning of this limit. 

The construction of the canonical Gaussian measure consists of two steps:
\medskip

\noindent $\bullet$ We first find a sequence of measures on the finite dimensional spaces 
	$$E_{\leq \Lambda}= \one_{h\leq \Lambda} L^2 (\R)$$
	that will define the cylindrical projections of the target measure.
\medskip

\noindent $\bullet$ We then show that the above sequence of measures is tight on a suitable Hilbert space, so that the target measure can be defined as its limit.

\medskip

This is summarized in the following statement, whose proof occupies this whole subsection.

\begin{proposition}[\bf Gaussian measure with a fixed renormalized mass] \label{prop-gaus-meas-mass} \mbox{} \\
	Let $s>1$, $V$ satisfy Assumption \ref{assu-V}, $\theta <\frac{1}{2}-\frac{1}{s}$, $m\in \R$, and assume $m>-{\Tr}[h^{-1}]$ if $s>2$. 
	
	\medskip
	
	\noindent (1) For any borelian set $A$ of $E_{\leq \Lambda}$, the limit
	\begin{align} \label{limi-Lambda}
		\lim_{\vareps \rightarrow 0^+} \mu_0^{m,\vareps}(A) = \lim_{\vareps \rightarrow 0^+} \frac{\mu_0(A\cap \{m-\vareps<\Mcal(u)<m+\vareps\})}{\mu_0(m-\vareps<\Mcal(u)<m+\vareps)}
	\end{align}
	exists, and thus me may define a probability measure on $E_{\leq \Lambda}$ by setting
	\[
	\mu^{m,\leq \Lambda}_0(A) := \lim_{\vareps \rightarrow 0^+} \mu_0^{m,\vareps} (A).
	\]
	
	\medskip
	
	\noindent (2) There exists a unique measure $\mu_0^m$ supported in $\Hc^\theta \cap \{\Mcal(u)=m\}$ such that for all $\Lambda\geq \lambda_1$, $\mu^{m,\leq \Lambda}_{0}$ is the cylindrical projection of $\mu_0^m$ on $E_{\leq \Lambda}$. Moreover, we have for any measurable set $A \subset \Hc^\theta$
	\[
	\mu_0^m(A) = \lim_{\vareps\rightarrow 0^+} \mu_0^{m,\vareps}(A).
	\]
\end{proposition}

We start by defining the finite dimensional measures as follows.

\begin{lemma}[\bf Cylindrical projections I] \mbox{}\\
	The projection of $\mu_{0}^{m,\vareps}$ on $E_{\leq \Lambda}$ is given by
	\begin{align}
		\left.d\mu_{0}^{m,\vareps}\right|_{E_{\leq \Lambda}} &= \frac{1}{Z_{0}^{m,\vareps}} \mu_0^{>\Lambda}\left( m-\vareps-\vartheta_\Lambda<\Mcal_{>\Lambda}(u)<m+\vareps-\vartheta_\Lambda\right) d\mu_0^{\leq \Lambda}(u)\label{cylin-proj-gaus-meas} \\
		&=: d\mu_{0}^{m,\vareps, \leq \Lambda}(u), \nonumber
	\end{align}
	where 
	$$\vartheta_\Lambda:=\Mcal_{\leq \Lambda}(u)= \sum_{\lambda_j \leq \Lambda} |\alpha_j|^2-\lambda_j^{-1}$$
	and $\Mcal_{>\Lambda}(u)$ is as in \eqref{eq:M-Lamb-hi}. The measure $\mu_{0}^{m,\vareps, \leq \Lambda}$ is the cylindrical projection of $\mu_{0}^{m,\vareps}$ on $E_{\leq \Lambda}$ in the sense that for $\Theta \geq \Lambda$,
	\begin{align} \label{cylin-proj-prop}
		\left.\mu_{0}^{m,\vareps, \leq \Theta} \right|_{E_{\leq \Lambda}} = \mu_{0}^{m,\vareps, \leq \Lambda}.
	\end{align}
\end{lemma}

\begin{proof}
	For any borelian set $A$ of $E_{\leq \Lambda}$, we have
	\begin{align*}
		\mu_0^{m,\vareps}(A) &= \frac{1}{Z_0^{m,\vareps}} \int \mathds{1}_{A \cap \{ m-\vareps <\Mcal(u)<m+\vareps\}} d\mu_0(u) \\
		&= \frac{1}{Z_0^{m,\vareps}} \int_{A} \left(\int \mathds{1}_{\left\{m-\vareps-\vartheta_\Lambda <\Mcal_{>\Lambda}(u) <m+\vareps-\vartheta_\Lambda \right\} } d\mu_{0}^{>\Lambda}(u) \right) d\mu_0^{\leq \Lambda}(u) \\
		&= \frac{1}{Z_0^{m,\vareps}} \int_{A} \mu_{0}^{>\Lambda}\left( m-\vareps-\vartheta_\Lambda <\Mcal_{>\Lambda}(u)<m+\vareps-\vartheta_\Lambda\right) d\mu_0^{\leq \Lambda}(u),
	\end{align*}	
	where $d\mu^{\leq\Lambda}_0(u)$ and $d\mu_0^{>\Lambda}(u)$ are defined as in \eqref{eq:mu-0-Lamb} and \eqref{eq:mu-0-Lamb-hi} respectively. This shows \eqref{cylin-proj-gaus-meas}. To see \eqref{cylin-proj-prop}, we observe that $A \times \C^N$ is a borelian set of $E_{\leq \Theta}$ with $N = \#\{\lambda_j : \Lambda <\lambda_j \leq \Theta\}$, hence
	\begin{align}
		&\left.\mu_{0}^{m,\vareps, \leq \Theta}\right|_{E_{\leq \Lambda}}(A)=\int_{A \times \C^N} d\mu_{0}^{m,\vareps, \leq \Theta}(u)\nonumber\\
		&=\frac{1}{Z_0^{m,\vareps}}\int_{A\times \C^N} \mu_{0}^{>\Theta}(m-\vareps-\vartheta_\Theta<\Mcal_{>\Theta}(u)<m+\vareps-\vartheta_\Theta) d\mu_{0}^{\leq \Theta}(u) \nonumber\\
		&=\frac{1}{Z_0^{m,\vareps}} \int_{A\times \C^N} \left(\int \mathds{1}_{\{m-\vareps-\vartheta_\Theta< \Mcal_{>\Theta}(u)<m+\vareps-\vartheta_\Theta\}} d\mu_{0}^{>\Theta}(u) \right) d\mu_{0}^{\leq \Theta}(u) \nonumber\\
		&=\frac{1}{Z_0^{m,\vareps}}\int_{A} \left(\int_{\C^N} \left(\int \mathds{1}_{\{m-\vareps-\vartheta_\Theta< \Mcal_{>\Theta}(u)<m+\vareps-\vartheta_\Theta\}} d\mu_{0}^{>\Theta}(u) \right) \prod_{\Lambda<\lambda_j\leq \Theta} \frac{\lambda_j}{\pi}e^{-\lambda_j|\alpha_j|^2} d\alpha_j\right) d\mu_{0}^{\leq \Lambda}(u) \nonumber\\
		&=\frac{1}{Z_0^{m,\vareps}} \int_{A} \left(\int \mathds{1}_{\{m-\vareps-\vartheta_\Lambda<\Mcal_{>\Lambda}(u)<m+\vareps-\vartheta_\Lambda\}} d\mu_{0}^{>\Lambda}(u) \right) d\mu_{0}^{\leq \Lambda}(u) \nonumber\\
		&=\frac{1}{Z_0^{m,\vareps}} \int_{A} \mu_{0}^{>\Lambda}(m-\vareps-\vartheta_\Lambda<\Mcal_{>\Lambda}(u)<m+\vareps-\vartheta_\Lambda) d\mu_{0}^{\leq \Lambda}(u) \nonumber\\
		&=\int_{A} d\mu_{0}^{m,\vareps, \leq \Lambda}(u) = \mu_{0}^{m,\vareps, \leq \Lambda}(A) \label{limi-eps-A}
	\end{align}
	which proves \eqref{cylin-proj-prop}.
\end{proof}

To prove that the limit in \eqref{limi-Lambda} exists, we denote $f_\Lambda$ the density function of $\Mcal_{> \Lambda}(u)$ with respect to $\mu_0^{>\Lambda}$. In particular, we have
\[
\mu_{0}^{> \Lambda}\left( m-\vareps-\vartheta_\Lambda <\Mcal_{>\Lambda}(u) <m+\vareps-\vartheta_\Lambda\right) = \int_{m-\vareps-\vartheta_\Lambda}^{m+\vareps -\vartheta_\Lambda} f_{\Lambda}(x) dx
\]
and
\[
Z_0^{m,\vareps}=\mu_0(m-\vareps <\Mcal(u)<m+\vareps) = \int_{m-\vareps}^{m+\vareps} f_0(x) dx,
\]
hence
\begin{align}\label{eq:mu-0-m-A-eps}
\mu_0^{m,\vareps}(A) = \int_{A} \left( \int_{m-\vareps-\vartheta_\Lambda}^{m+\vareps -\vartheta_\Lambda} f_{\Lambda}(x) dx\right)  \left(\int_{m-\vareps}^{m+\vareps} f_0(x) dx \right)^{-1} d\mu_{0}^{\leq \Lambda}(u).
\end{align}

\begin{lemma}[\bf Uniform continuity of the density function] \label{lem-f-Lamb} \mbox{} \\
	For any $\Lambda\geq 0$, $f_\Lambda$ is bounded and uniformly continuous on $(m_0,+\infty)$, where $m_0=-{\Tr}[h^{-1}]$ if $s>2$ and $m_0=-\infty$ if $1<s\leq 2$. In addition, for $\Lambda>0$ sufficiently large, there exists $C>0$ such that $\|f_\Lambda\|_{L^\infty((m_0,+\infty))} \leq C \Lambda$. 
\end{lemma}

\begin{proof}
	Denote $\phi_\Lambda$ the characteristic function of $\Mcal_{>\Lambda}(u)$ with respect to $\mu_0^{>\Lambda}$. We have
	\begin{align*}
		\phi_\Lambda(s) &= \E_{\mu_0^{>\Lambda}}[e^{is \Mcal_{>\Lambda}(u)}] \\
		&= \int e^{is \sum_{\lambda_j>\Lambda} |\alpha_j|^2-\lambda_j^{-1}} d\mu_0^{>\Lambda}(u) \\
		&= \int \prod_{\lambda_j>\Lambda} e^{is (|\alpha_j|^2-\lambda_j^{-1})} \prod_{\lambda_k>\Lambda} \frac{\lambda_k}{\pi} e^{-\lambda_k|\alpha_k|^2} d\alpha_k \\
		&= \prod_{\lambda_j>\Lambda} \left(\int_{\C} e^{-is\lambda_j^{-1}} \frac{\lambda_j}{\pi}e^{-(1-is\lambda_j^{-1})\lambda_j |\alpha_j|^2} d\alpha_j\right) \left(\prod_{\lambda_k\ne \lambda_j\atop \lambda_k> \Lambda} \int_{\C} \frac{\lambda_k}{\pi} e^{-\lambda_k|\alpha_k|^2} d\alpha_k \right) \\
		&=\prod_{\lambda_j>\Lambda} \frac{e^{-is \lambda_j^{-1}}}{1-is \lambda_j^{-1}}. 
	\end{align*}
	Each factor of this product has complex norm smaller than or equal to 1, thus the norm of this product is bounded by the norm of a product of any two terms. In particular, we have
	\[
	|\phi_\Lambda(s)| \leq \left| \frac{e^{-is\lambda_{j_1}^{-1}}}{1-is\lambda_{j_1}^{-1}}\frac{e^{-is\lambda_{j_2}^{-1}}}{1-is\lambda_{j_2}^{-1}}\right| = \frac{1}{\sqrt{1+s^2\lambda_{j_1}^{-2}}} \frac{1}{\sqrt{1+s^2\lambda_{j_2}^{-2}}} \leq \frac{1}{1+s^2\lambda_{j_2}^{-2}},
	\]
	where $j_1$ and $j_2$ are the first two indices such that $\lambda_{j_2}\geq \lambda_{j_1}>\Lambda$. We deduce that $\phi_\Lambda \in L^1(\R)$ with 
	$$\|\phi_\Lambda\|_{L^1(\R)} \leq C\lambda_{j_2}.$$
	We have (see Lemma \ref{lem-numb-eigen}) that the number $N(\Lambda)$ of eigenvalues of $h$ below $\Lambda$ satisfies 
	$$ c \Lambda^{\frac{1}{2}+\frac{1}{s}}\leq N(\Lambda) \leq C \Lambda^{\frac{1}{2}+\frac{1}{s}}$$
	for positive constants $c,C >0$. Hence, we may pick some $k>0$ large enough but independent of $\Lambda$  to ensure that  
	$$ N(k \Lambda) - N(\Lambda) \geq \left( c k ^{1/2+1/s} - C\right)\Lambda^{\frac{1}{2}+\frac{1}{s}} > 1 $$
	for $\Lambda$ sufficiently large. Hence $\lambda_{j_2}$ can be chosen of order  $\Lambda$ in the above argument and we deduce that
	$$\|\phi_\Lambda\|_{L^1(\R)}\leq C\Lambda$$
	for $\Lambda$ sufficiently large. 
	
	Using the following relation between density and characteristic functions
	\[
	f_\Lambda(x) = \frac{1}{2\pi} \int_{\R} e^{-isx} \phi_\Lambda(s) ds,
	\]
	we infer that $f_\Lambda$ is bounded and uniformly continuous on $(m_0,+\infty)$. Note that when $s>2$, we have $\Mcal_{>\Lambda}(u) = \|P_{>\Lambda}u\|^2_{L^2}-\scal{\|P_{>\Lambda}u\|^2_{L^2}}_{\mu_0} >-{\Tr}[h^{-1}]$, hence the density function $f_\Lambda$ is defined on $(-{\Tr}[h^{-1}],+\infty)$. 
\end{proof}

\begin{lemma}[\bf Positivity of the density function] \label{lem-f0-m} \mbox{} \\
	Let $s>1$, $m\in \R$, and assume $m>-{\Tr}[h^{-1}]$ if $s>2$. We have
	\[
	\lim_{\vareps \rightarrow 0^+} \frac{1}{2\vareps} \int_{m-\vareps}^{m+\vareps} f_0(x) dx = f_0(m)>0.
	\]
\end{lemma}

\begin{proof}
	Since $f_0$ is uniformly continuous on $\R$, we have
	\[
	\frac{1}{2\vareps} \int_{m-\vareps}^{m+\vareps} f_0(x)dx \xrightarrow[\vareps \rightarrow 0^+]{} f_0(m).
	\]
	It remains to prove that $f_0(m)>0$. Since the first eigenvalue of $h$ is simple (see e.g., \cite[Theorem 11.8]{LieLos-01}), we have
	\begin{align*}
		&\frac{1}{2\vareps} \int_{m-\vareps}^{m+\vareps} f_0(x) dx \\
		&=\frac{1}{2\vareps} \mu_0(m-\vareps<\Mcal(u)<m+\vareps) \\
		&= \frac{1}{2\vareps}\mu_0(\Mcal_{>\lambda_1}(u) \in (m_0,+\infty), m-\vareps-\Mcal_{>\lambda_1}(u) <|\alpha_1|^2-\lambda_1^{-1}<m+\vareps-\Mcal_{>\lambda_1}(u)) \\
		&= \frac{1}{2\vareps}\mu_0(\Mcal_{>\lambda_1}(u) \in (m_0,+\infty),  \lambda_1^{-1}+m-\vareps-\Mcal_{>\lambda_1}(u) <|\alpha_1|^2<\lambda_1^{-1}+m+\vareps-\Mcal_{>\lambda_1}(u)) \\
		&= \frac{1}{2\vareps} \int_{m_0}^{+\infty} f_{\lambda_1}(x) \left(\int_{\max\{0,\lambda_1^{-1} + m-\vareps-x\}<|\alpha_1|^2<\lambda_1^{-1} + m+\vareps-x} \frac{\lambda_1}{\pi} e^{-\lambda_1|\alpha_1|^2} d\alpha_1 \right)  dx.
	\end{align*}
	We estimate it further as
	\begin{align*}
		\frac{1}{2\vareps} &\int_{m-\vareps}^{m+\vareps} f_0(x) dx\\
		&= \frac{1}{2\vareps}\int_{m_0}^{\lambda_1^{-1}+m-\vareps} f_{\lambda_1}(x) \left(\int_{\lambda_1^{-1} + m-\vareps-x<|\alpha_1|^2<\lambda_1^{-1} + m+\vareps-x} \frac{\lambda_1}{\pi} e^{-\lambda_1|\alpha_1|^2} d\alpha_1 \right)  dx \\
		& \quad +\frac{1}{2\vareps} \int_{\lambda_1^{-1}+m-\vareps}^{\lambda_1^{-1}+m+\vareps} f_{\lambda_1}(x) \left(\int_{0<|\alpha_1|^2<\lambda_1^{-1} + m+\vareps-x} \frac{\lambda_1}{\pi} e^{-\lambda_1|\alpha_1|^2} d\alpha_1 \right) dx \\
		&=\frac{1}{2\vareps} \int_{m_0}^{\lambda_1^{-1}+m-\vareps} f_{\lambda_1}(x) \left(\int_{1+\lambda_1(m-\vareps-x)}^{1+\lambda_1(m+\vareps-x)} e^{-\lambda} d\lambda\right) dx \\
		&\quad +\frac{1}{2\vareps}\int_{\lambda_1^{-1}+m-\vareps}^{\lambda_1^{-1}+m+\vareps} f_{\lambda_1}(x) \left(\int_0^{1+\lambda_1(m+\vareps-x)} e^{-\lambda} d\lambda\right) dx \\
		&\geq \frac{1}{2\vareps} \int_{m_0}^{\lambda_1^{-1}+m-\vareps} f_{\lambda_1}(x) e^{-1-\lambda_1 m +\lambda_1 x} (e^{\lambda_1\vareps} -e^{-\lambda_1 \vareps})dx \\
		&= \frac{1}{2\vareps} \int_{m_0}^{\lambda_1^{-1}+m} f_{\lambda_1}(x) e^{-1-\lambda_1 m +\lambda_1 x} (e^{\lambda_1 \vareps} - e^{-\lambda_1 \vareps})dx \\
		&\quad -\frac{1}{2\vareps}\int_{\lambda_1^{-1}+m-\vareps}^{\lambda_1^{-1} +m} f_{\lambda_1}(x) e^{-1-\lambda_1 m+\lambda_1 x} (e^{\lambda_1 \vareps}-e^{-\lambda_1\vareps})dx.
	\end{align*}
	The last term converges to zero as $\vareps \rightarrow 0^+$ due to the dominated convergence theorem. Thus letting $\vareps \rightarrow 0^+$, we get
	\begin{align*}
		f_0(m) = \lim_{\vareps \to 0^+} \frac{1}{2\vareps}\int_{m-\vareps}^{m+\vareps} f_0(x) dx &\geq \lim_{\vareps \to 0^+} \int_{m_0}^{\lambda_1^{-1}+m} f_{\lambda_1}(x) e^{-1-\lambda_1 m+\lambda_1 x} \frac{e^{\lambda_1 \vareps}-e^{-\lambda_1\vareps}}{2\vareps} dx \\
		&= \int_{m_0}^{\lambda_1^{-1}+m} f_{\lambda_1}(x) e^{-1-\lambda_1 m+\lambda_1 x} \lambda_1 dx.
	\end{align*}
	Assume for contradiction that $f_0(m)=0$. Since $f_{\lambda_1}(x)\geq 0$ for all $x\in (m_0,+\infty)$, we infer that $f_{\lambda_1}(x)=0$ for all $x\in (m_0,\lambda_1^{-1}+m)$. In particular, we have
	\begin{align*}
		&0= \int_{m_0}^{\lambda_1^{-1}+m} f_{\lambda_1}(x)dx \\
		&= \mu_{0}^{>\lambda_1} (m_0<\Mcal_{>\lambda_1}(u)<\lambda_1^{-1}+m) \\
		&= \mu_{0}^{>\lambda_1} \Big(\Mcal_{>\lambda_2}(u)\in (m_0,+\infty), m_0 -\Mcal_{>\lambda_2}(u)< \sum_{\lambda_j=\lambda_2}|\alpha_j|^2-\lambda_j^{-1}<\lambda_1^{-1}+m-\Mcal_{>\lambda_2}(u)\Big).
	\end{align*}
	To proceed further, we denote  
	$$\alpha=(\alpha_j)_{\lambda_j=\lambda_2} \in \C^{N},$$
	with $N$ the multiplicity of $\lambda_2$, hence $d\alpha=\prod_{\lambda_j=\lambda_2} d\alpha_j$ and $|\alpha|^2 = \sum_{\lambda_j=\lambda_2} |\alpha_j|^2$. The above measure becomes
	\begin{align*}
	&\int_{m_0}^{N\lambda_2^{-1}+\lambda_1^{-1}+m} f_{\lambda_2}(x) \Big( \int_{\max\left\{0,m_0+ N\lambda_2^{-1}-x\right\}<|\alpha|^2<N\lambda_2^{-1}+\lambda_1^{-1}+m-x} \left(\frac{\lambda_2}{\pi}\right)^N e^{-\lambda_2|\alpha|^2}d\alpha\Big) dx\\
	&= \int_{m_0}^{N\lambda_2^{-1}+m_0} f_{\lambda_2}(x) \Big(\int_{m_0+N\lambda_2^{-1}-x<|\alpha|^2<N\lambda_2^{-1}+\lambda_1^{-1}+m-x} \left(\frac{\lambda_2}{\pi}\right)^N e^{-\lambda_2|\alpha|^2}d\alpha \Big) dx \\
	&\quad + \int_{N\lambda_2^{-1}+m_0}^{N\lambda_2^{-1}+\lambda_1^{-1}+m} f_{\lambda_2}(x) \Big(\int_{0<|\alpha|^2<N\lambda_2^{-1}+\lambda_1^{-1}+m-x} \left(\frac{\lambda_2}{\pi}\right)^N e^{-\lambda_2|\alpha|^2}d\alpha \Big) dx
	\end{align*}
	Using polar coordinates, this becomes (up to a factor $\sigma(\Sb^{2N-1})$)
	\begin{align*}
		&\int_{m_0}^{N\lambda_2^{-1}+m_0} f_{\lambda_2}(x) \Big(\int_{\sqrt{m_0+N\lambda_2^{-1}-x}}^{\sqrt{N\lambda_2^{-1}+\lambda_1^{-1}+m-x}} \left(\frac{\lambda_2}{\pi}\right)^N e^{-\lambda_2 r^2} r^{2N-1} dr \Big) dx \\
		&+ \int_{N\lambda_2^{-1}+m_0}^{N\lambda_2^{-1}+\lambda_1^{-1}+m-x} f_{\lambda_2}(x) \Big(\int_0^{\sqrt{N\lambda_2^{-1}+\lambda_1^{-1}+m}} \left(\frac{\lambda_2}{\pi}\right)^N e^{-\lambda_2 r^2} r^{2N-1} dr \Big) dx.
	\end{align*}
	By a change of variable $\lambda =\lambda_2 r^2$, the above quantity is (up to a factor $\frac{1}{2\pi^N}$)
	\begin{align*}
		&\int_{m_0}^{N\lambda_2^{-1}+m_0} f_{\lambda_2}(x) \Big(\int_{\lambda_2(m_0+N\lambda_2^{-1}-x)}^{\lambda_2(N\lambda_2^{-1}+\lambda_1^{-1}+m-x)} e^{-\lambda} \lambda^{N-1}\lambda \Big) dx \\
		&+ \int_{N\lambda_2^{-1}+m_0}^{N\lambda_2^{-1}+\lambda_1^{-1}+m} f_{\lambda_2}(x) \Big(\int_0^{\lambda_2(N\lambda_2^{-1}+\lambda_1^{-1}+m-x)} e^{-\lambda} \lambda^{N-1}\lambda \Big) dx.
	\end{align*}
	Since this sum is equal to zero and both terms are non-negative, we infer that $f_{\lambda_2}(x)=0$ for all $x\in (m_0, N\lambda_2^{-1}+\lambda_1^{-1}+m)$. Arguing in a similar manner, we can prove that $f_\Lambda(x)=0$ for all $x\in (m_0,{\Tr}[(P_{\leq \Lambda}h)^{-1}]+m)$ and all $\Lambda\geq 0$. Thus
	\[
	1=\int_{m_0}^{+\infty} f_\Lambda(x) dx =\int_{{\Tr}[(P_{\leq \Lambda} h)^{-1}]+m}^{+\infty} f_\Lambda(x) dx, \quad \forall \Lambda\geq 0.
	\]
	This contradicts the fact that
	\begin{align*}
		\int_{{\Tr}[(P_{\leq \Lambda}h)^{-1}]+m}^{+\infty} f_\Lambda(x) dx &= \mu_0^{>\Lambda}\left(\Mcal_{>\Lambda}(u)>{\Tr}[(P_{\leq \Lambda} h)^{-1}]+m\right) \\
		&\leq \mu_0^{>\Lambda}\left((\Mcal_{>\Lambda}(u))^2 >({\Tr}[(P_{\leq \Lambda} h)^{-1}]+m)^2\right) \\
		&\leq \frac{1}{({\Tr}[(P_{\leq \Lambda} h)^{-1}]+m)^2} \int (\Mcal_{>\Lambda}(u))^2 d\mu_0^{>\Lambda}(u) \\
		&=\frac{1}{({\Tr}[(P_{\leq \Lambda} h)^{-1}]+m)^2}\sum_{\lambda_j>\Lambda} \lambda_j^{-2} \rightarrow 0 \text{ as } \Lambda\rightarrow \infty
	\end{align*}
	due to ${\Tr}[h^{-2}]=\sum \lambda_j^{-2} <\infty$. Note that for $m$ fixed, we have
	\[
	{\Tr}[(P_{\leq \Lambda}h)^{-1}]+m \xrightarrow[\Lambda \rightarrow \infty]{} \left\{
	\begin{array}{cl}
		m+{\Tr}[h^{-1}]>0 &\text{if } s>2, \\
		+\infty &\text{if }1<s\leq 2.
	\end{array}
	\right.
	\]
	The proof is complete.
\end{proof}

\begin{lemma}[\bf Cylindrical projections II] \mbox{} \\
	Let $s>1$, $m\in \R$, and assume $m>-{\Tr[h^{-1}]}$ if $s>2$. Then for any $\Lambda\geq \lambda_1$ and any borelian set $A$ of $E_{\leq \Lambda}$, the limit $\lim_{\vareps \rightarrow 0^+} \mu_{0}^{m,\vareps}(A)$ exists and
	\begin{align} \label{eq:limi-A-eps}
	\lim_{\vareps \rightarrow 0^+} \mu_{0}^{m,\vareps} (A) =\int_{A} \frac{f_\Lambda\left(m-\vartheta_\Lambda\right)}{f_0(m)} d\mu_0^{\leq\Lambda}(u) =:\mu_{0}^{m,\leq \Lambda}(A).
	\end{align}
	In particular,
	\[
	d\mu_{0}^{m,\leq \Lambda}(u)= \frac{f_\Lambda(m-\vartheta_\Lambda)}{f_0(m)} d\mu_0^{\leq \Lambda}(u)
	\]
	satisfies for any $\Theta\geq \Lambda$,
	\begin{align} \label{eq:cyli-proj}
	\left.\mu_{0}^{m,\leq \Theta}\right|_{E_{\leq \Lambda}}=\mu_{0}^{m,\leq \Lambda}.
	\end{align}
\end{lemma}

\begin{proof}
	On the one hand, by the uniform continuity of $f_\Lambda$, we have
	\[
	\left(\frac{1}{2\vareps}\int_{m-\vareps-\vartheta_\Lambda}^{m+\vareps-\vartheta_\Lambda} f_\Lambda(x) dx\right)  \left(\frac{1}{2\vareps}\int_{m-\vareps}^{m+\vareps} f_0(x) dx \right)^{-1} \xrightarrow[\vareps \rightarrow 0^+]{} \frac{f_\Lambda(m-\vartheta_\Lambda)}{f_0(m)}.
	\]
	On the other hand, the boundedness of $f_\Lambda$ gives
	\[
	\frac{1}{2\vareps} \int_{m-\vareps-\vartheta_\Lambda}^{m+\vareps-\vartheta_\Lambda} f_\Lambda(x)dx\leq \|f_\Lambda\|_{L^\infty}.
	\]
	Since $f_0(m)>0$, we have for $\vareps>0$ sufficiently small,
	\[
	\frac{1}{2\vareps} \int_{m-\vareps}^{m+\vareps} f_0(x)dx \geq \frac{f_0(m)}{2}.
	\]
	In particular, we get for $\vareps>0$ sufficiently small,
	\[
	\left(\int_{m-\vareps-\vartheta_\Lambda}^{m+\vareps-\vartheta_\Lambda} f_\Lambda(x) dx\right)  \left(\int_{m-\vareps}^{m+\vareps} f_0(x) dx \right)^{-1} d\mu_0^{\leq \Lambda}(u) \leq \frac{2\|f_\Lambda\|_{L^\infty}}{f_0(m)} d\mu_0^{\leq \Lambda}(u)
	\]
	which is integrable on $A$. From \eqref{eq:mu-0-m-A-eps}, the dominated convergence theorem yields \eqref{eq:limi-A-eps}.
	
	To see \eqref{eq:cyli-proj}, we use \eqref{limi-eps-A} to have for any borelian set $A$ of $E_{\leq \Lambda}$,
	\begin{align*}
		\left.\mu^{m,\leq \Theta}_{0}\right|_{E_{\leq \Lambda}} (A) &= \mu_{0}^{m,\leq \Theta}(A\times \C^{N}) \\
		&= \lim_{\vareps \rightarrow 0^+} \mu_{0}^{m,\vareps,\leq \Theta} (A\times \C^{N}) \\
		&= \lim_{\vareps \rightarrow 0^+} \mu_{0}^{m,\vareps,\leq \Lambda}(A) \\
		&= \mu_{0}^{m,\leq \Lambda}(A),
	\end{align*}
	where $N=\#\{\lambda_j : \Lambda<\lambda_j\leq \Theta\}$.
\end{proof}

We have constructed a sequence of measures $\{\mu_{0}^{m,\leq \Lambda}\}_{\Lambda\geq \lambda_1}$ on the finite dimensional spaces $\{E_{\leq \Lambda}\}_{\Lambda\geq \lambda_1}$ that satisfies the cylindrical property~\eqref{eq:cyli-proj}. Thus we have completed the proof of Item (1) of Proposition~\ref{prop-gaus-meas-mass}.

To prove Item (2), we will show that there exists a unique measure $\mu_0^m$ on an infinite dimensional Hilbert space satisfying
\[
\left.\mu_0^m\right|_{E_{\leq \Lambda}} = \mu_{0}^{m,\leq \Lambda}.
\]
By Skorokhod's criterion (see \cite[Lemma 1]{Skorokhod-74}), it suffices to check that the sequence $\{\mu_{0}^{m,\leq \Lambda}\}_{\Lambda\geq \lambda_1}$ is tight in the sense that
\[
\lim_{R\rightarrow \infty} \sup_{\Lambda\geq \lambda_1} \mu_{0}^{m,\leq \Lambda}(\{ u \in E_{\leq \Lambda} : \|u\|_{\Hc^\theta} \geq R\}) =0
\]
for $\theta <\frac{1}{2}-\frac{1}{s}$. This clearly follows from the following lemma, whose proof will thus complete that of Proposition~\ref{prop-gaus-meas-mass}.

\begin{lemma}[\bf Tightness of the canonical Gaussian measure] \label{lem-Cm} \mbox{} \\
	Let $s>1$, $m\in \R$, and assume $m>-{\Tr}[h^{-1}]$ if $s>2$. Then for any $\theta <\frac{1}{2}-\frac{1}{s}$, there exists $C(m,\theta)>0$ such that
	\begin{align*}
		\int_{E_{\leq \Lambda}} \|u\|^2_{\Hc^\theta} d\mu_0^{m,\leq \Lambda}(u) \leq C(m,\theta), \quad \forall \Lambda \geq \lambda_1.
	\end{align*} 
\end{lemma}

\begin{proof}
	It suffices to prove that 
	\begin{equation}\label{eq:Cm}
		\int \lambda_j |\alpha_j|^2 d\mu_{0}^{m,\leq \Lambda}(u) \leq C(m), \quad \forall \Lambda \geq \lambda_1, \quad \forall \lambda_1\leq \lambda_ j\leq \Lambda
	\end{equation}
	for some constant $C(m)$ depending only on $m$. Indeed, we have
	\begin{align*}
		\int_{E_{\leq \Lambda}} \|u\|^2_{\Hc^\theta} d\mu_{0}^{m,\leq \Lambda}(u) &= \int\sum_{\lambda_j\leq \Lambda} \lambda_j^\theta|\alpha_j|^2 d\mu_{0}^{m,\leq \Lambda}(u) \\
		&=\sum_{\lambda_j\leq \Lambda} \lambda_j^{\theta-1}\int \lambda_j |\alpha_j|^2 d\mu_{0}^{m,\leq \Lambda}(u) \\
		&\leq C(m)\sum_{\lambda_j\leq \Lambda} \lambda_j^{-1+\theta} \\
		&\leq C(m) {\Tr}[h^{-1+\theta}] <\infty,
	\end{align*}
	where we have used the fact that $1-\theta>\frac{1}{2}+\frac{1}{s}$ to get the finiteness of ${\Tr}[h^{-1+\theta}]$ (see Lemma \ref{lem-trac-h-p}).
	
	We are thus reduced to proving~\eqref{eq:Cm}, which we shall deduce from the fact that for all $\Lambda\geq \lambda_1$, all $\lambda_1\leq \lambda_ j\leq \Lambda$, and all $\vareps>0$ sufficiently small,
	\begin{align} \label{eq:Cm-eps}
		\int \lambda_j |\alpha_j|^2 d\mu_{0}^{m,\vareps, \leq \Lambda}(u) \leq C(m).
	\end{align}
	We have
	\begin{align*}
		&\int \lambda_j|\alpha_j|^2  d\mu_{0}^{m,\vareps,\leq \Lambda}(u) \\
		&\quad = \frac{1}{Z_{0}^{m,\vareps}}\int  \lambda_j|\alpha_j|^2 \mu_0^{>\Lambda}(m-\vareps-\vartheta_\Lambda<\Mcal_{>\Lambda}(u)<m+\vareps-\vartheta_\Lambda) d\mu_0^{\leq \Lambda}(u)\\
		&\quad =\frac{1}{Z_{0}^{m,\vareps}} \int \lambda_j |\alpha_j|^2 \left(\int \mathds{1}_{\{m-\vareps-\vartheta_\Lambda<\Mcal_{>\Lambda}(u)<m+\vareps-\vartheta_\Lambda\}} d\mu_0^{>\Lambda}(u)\right) d\mu_0^{\leq \Lambda}(u) \\
		&\quad =\frac{1}{Z_{0}^{m,\vareps}} \int \lambda_j|\alpha_j|^2  \left(\int \mathds{1}_{\{m-\vareps-|\alpha_j|^2 +\lambda_j^{-1}<\Mcal_{\ne j}(u)<m+\vareps-|\alpha_j|^2+\lambda_j^{-1}\}} d\mu_0^{\ne j}(u)\right) \frac{\lambda_j}{\pi}e^{-\lambda_j|\alpha_j|^2} d\alpha_j \\
		&\quad = \frac{1}{Z_{0}^{m,\vareps}} \int \lambda_j|\alpha_j|^2 \mu_0^{\ne j}(m-\vareps-|\alpha_j|^2+\lambda_j^{-1}<\Mcal_{\ne j}(u)<m+\vareps-|\alpha_j|^2+\lambda_j^{-1}) \frac{\lambda_j}{\pi} e^{-\lambda_j|\alpha_j|^2} d\alpha_j,
	\end{align*}
	where 
	\[
	\Mcal_{\ne j} (u) = \sum_{k\ne j} |\alpha_k|^2-\lambda_k^{-1}, \quad d\mu_0^{\ne j}(u) = \prod_{k \ne j} \frac{\lambda_k}{\pi} e^{-\lambda_k|\alpha_k|^2} d\alpha_k.
	\]
	Denote $F_j$ the density function of $\Mcal_{\ne j}(u)$ with respect to $\mu_0^{\ne j}$. We rewrite
	\[
	\int (\lambda_j|\alpha_j|^2)^q  d\mu_{0}^{m,\vareps,\leq \Lambda}(u) = \frac{1}{Z_{0}^{m,\vareps}} \int \lambda_j|\alpha_j|^2  \left(\int_{m-\vareps-|\alpha_j|^2+\lambda_j^{-1}}^{m+\vareps-|\alpha_j|^2+\lambda_j^{-1}} F_j(x) dx\right) \frac{\lambda_j}{\pi} e^{-\lambda_j|\alpha_j|^2} d\alpha_j. 
	\]
	To proceed further, we denote by $\Phi_j$ the characteristic function of $\Mcal_{\ne j}(u)$ with respect to $\mu_0^{\ne j}$. As in the proof of Lemma \ref{lem-f-Lamb}, we have
	\[
	\Phi_j(s) = \E_{\mu_0^{\ne j}}[e^{is \Mcal_{\ne j}(u)}] = \prod_{k\ne j} \frac{e^{-is\lambda_k^{-1}}}{1-is\lambda_k^{-1}}.
	\]
	Since each factor of this product has complex norm smaller than or equal to 1, we bound the norm of this product by the norm of a product of any two terms. Taking the first two terms, we obtain that for all $\lambda_j\geq \lambda_1$,
	\[
	|\Phi_j(s)| \leq \frac{1}{1+s^2 \lambda_3^{-2}}.
	\]
	In particular, $\|\Phi_j\|_{L^1(\R)} \leq C\lambda_3$ for all $\lambda_j\geq \lambda_1$. Thus $F_j$ is bounded (uniformly in $j$) and uniformly continuous for all $\lambda_j\geq \lambda_1$. 
	
	By Lemma \ref{lem-f0-m}, we have for $\vareps>0$ sufficiently small,
	\[
	\frac{1}{2\vareps} Z_{0}^{m,\vareps} \geq \frac{f_0(m)}{2}.
	\]
	We also have
	\[
	\frac{1}{2\vareps} \int_{m-\vareps-|\alpha_j|^2+\lambda_j^{-1}}^{m+\vareps-|\alpha_j|^2+\lambda_j^{-1}} F_j(x) dx \leq  \|F_j\|_{L^\infty} \leq C\|\Phi_j\|_{L^1(\R)} \leq C \lambda_3, \quad \forall \lambda_j\geq \lambda_1.
	\]
	It follows that 
	\begin{align*}
		\int \lambda_j|\alpha_j|^2  d\mu_{0}^{m,\vareps,\leq \Lambda}(u) &\leq \frac{2C\lambda_3}{f_0(m)} \int \lambda_j|\alpha_j|^2 \frac{\lambda_j}{\pi}e^{-\lambda_j|\alpha_j|^2} d\alpha_j \\
		&=\frac{2C\lambda_3}{f_0(m)} \int_0^\infty \lambda e^{-\lambda} d\lambda \\
		&= C(m)
	\end{align*}
	for all $\lambda_j\geq \lambda_1$ and all $\vareps>0$ sufficiently small. This proves \eqref{eq:Cm-eps}.
	
	We now prove \eqref{eq:Cm}. Using the layer cake representation, the problem is reduced to showing that 
	\[
	\int_0^\infty \mu_{0}^{m,\leq \Lambda}( \lambda_j|\alpha_j|^2 >\lambda) d\lambda = \lim_{\vareps \rightarrow 0^+} \int_0^\infty \mu_{0}^{m,\vareps,\leq \Lambda}(\lambda_j|\alpha_j|^2 >\lambda) d\lambda.
	\]
	Thanks to the weak convergence $\mu_{0}^{m,\vareps,\leq \Lambda} \rightarrow \mu_{0}^{m,\leq \Lambda}$, \eqref{eq:Cm} follows from the dominated convergence theorem and the fact that
	\begin{align*}
		&\mu_{0}^{m,\vareps,\leq \Lambda}(\lambda_j|\alpha_j|^2>\lambda) \\
		&\quad = \int \mathds{1}_{\{\lambda_j|\alpha_j|^2>\lambda\}} d\mu_{0}^{m,\vareps,\leq \Lambda}(u) \\
		&\quad = \frac{1}{Z_{0}^{m,\vareps}}\int \mathds{1}_{\{\lambda_j|\alpha_j|^2>\lambda\}} \mu_0^{>\Lambda}(m-\vareps-\vartheta_\Lambda<\Mcal_{>\Lambda}(u)<m+\vareps-\vartheta_\Lambda) d\mu_0^{\leq \Lambda}(u) \\
		&\quad =\frac{1}{Z_{0}^{m,\vareps}}\int \mathds{1}_{\{\lambda_j|\alpha_j|^2>\lambda\}} \left(\int \mathds{1}_{\{m-\vareps-\vartheta_\Lambda<\Mcal_{>\Lambda}(u)<m+\vareps-\vartheta_\Lambda\}} d\mu_0^{>\Lambda}(u)\right) d\mu_0^{\leq \Lambda}(u) \\
		&\quad = \frac{1}{Z_{0}^{m,\vareps}}\int \mathds{1}_{\{\lambda_j|\alpha_j|^2>\lambda\}\cap \{|\alpha_j|^2-\lambda_j^{-1}\in \R\}} \\
		&\quad \quad \times \left( \int \mathds{1}_{\{m-\vareps-|\alpha_j|^2+\lambda_j^{-1}<\Mcal_{\ne j}(u)<m+\vareps-|\alpha_j|^2+\lambda_j^{-1}\}} d\mu_0^{\ne j}(u)\right) \frac{\lambda_j}{\pi} e^{-\lambda_j|\alpha_j|^2} d\alpha_j \\
		&\quad = \frac{1}{Z_{0}^{m,\vareps}} \int \mathds{1}_{\{\lambda_j|\alpha_j|^2>\lambda\}\cap \{|\alpha_j|^2-\lambda_j^{-1}\in \R\}} \left(\int_{m-\vareps-|\alpha_j|^2+\lambda_j^{-1}}^{m+\vareps-|\alpha_j|^2+\lambda_j^{-1}} F_j(x) dx\right) \frac{\lambda_j}{\pi} e^{-\lambda_j|\alpha_j|^2}d\alpha_j \\
		&\quad \leq \frac{2C\lambda_3}{f_0(m)}  \int \mathds{1}_{\{\lambda_j|\alpha_j|^2>\lambda, |\alpha_j|^2-\lambda_j^{-1}\in \R\}} \frac{\lambda_j}{\pi}e^{-\lambda_j|\alpha_j|^2}d\alpha_j \\
		&\quad \leq \frac{2C\lambda_3}{f_0(m)} \int^\infty_{\lambda} e^{-\tau}d\tau \\
		&\quad = \frac{2C\lambda_3}{f_0(m)} e^{-\lambda}
	\end{align*}
	which is integrable on $(0,\infty)$.
\end{proof}

\subsection{Gibbs measures conditioned on the mass}
\label{sec:Gibbs cond mass}
Once the Gaussian measure conditioned on mass is constructed, our next task is to define the Gibbs measure with a fixed renormalized mass as follows:
\[
d\mu^m(u) = \frac{1}{Z^m} e^{\mp\frac{1}{2}\|u\|^4_{L^4}} d\mu_0^m(u),
\]
where
\[
Z^m := \int e^{\mp\frac{1}{2}\|u\|^4_{L^4}} d\mu_0^m(u)
\]
is the normalization constant. Here the minus sign stands for the defocusing nonlinearity, and the plus sign is for the focusing one.

\begin{proposition} [\bf Gibbs measures conditioned on mass] \label{prop-mu-fix-renor} \mbox{}\\
	Let $s>1$, $V$ satisfy Assumption \ref{assu-V}, $m\in \R$, and assume $m>-{\Tr}[h^{-1}]$ if $s>2$. Assume in addition that $s>\frac{8}{5}$ for the focusing nonlinearity. Then $\mu^m$ makes sense as a probability measure.
\end{proposition}

We first need the following observation.

\begin{lemma} [\bf Decay of the renormalized mass] \label{lem-mu-fix-renor} \mbox{} \\
	Let $s>1$, $V$ satisfy Assumption \ref{assu-V}, $0\leq \gamma<\frac{3s-2}{4s}$, $m\in \R$, and assume $m>-{\Tr}[h^{-1}]$ if $s>2$. Then there exist $C,c>0$ such that for $\theta<\frac{1}{2}-\frac{1}{s}$, all $\Lambda \geq \lambda_2$, all $R>0$, and all $\vareps>0$ sufficiently small,
	\begin{align} \label{eq:boun-M-mu-renor-eps}
		\mu_{0}^{m,\vareps}\left(u \in \Hc^\theta : |\Mcal_{>\Lambda}(u)|>R\right) \leq Ce^{-c\Lambda^\gamma R}.
	\end{align}
	In particular, we have
	\begin{align} \label{eq:boun-M-mu-renor}
		\mu_{0}^m \left(u \in \Hc^\theta : |\Mcal_{>\Lambda}(u)|>R\right) \leq Ce^{-c\Lambda^\gamma R}.
	\end{align}
\end{lemma}

\begin{proof}
	Since $\mu_{0}^{m,\vareps}\to \mu_{0}^m$ as $\vareps \to 0^+$ (see Proposition \ref{prop-gaus-meas-mass}), it is enough to prove \eqref{eq:boun-M-mu-renor-eps}. We estimate
	\[
	\mu_{0}^{m,\vareps}\left(|\Mcal_{>\Lambda}(u)|>R\right) \leq \mu_{0}^{m,\vareps}\left(\Mcal_{>\Lambda}(u)>R\right) + \mu_{0}^{m,\vareps}\left(\Mcal_{>\Lambda}(u)<-R\right) = (\text{I}) + (\text{II}).
	\]
	For $(\text{I})$, we have for $0<t<\frac{\Lambda}{2}$ to be chosen shortly,
	\begin{align*}
		\mu_{0}^{m,\vareps}\left(\Mcal_{>\Lambda}(u)>R\right) \leq e^{-tR} \int e^{t\Mcal_{>\Lambda}(u)} d\mu_{0}^{m,\vareps}(u).
	\end{align*}
	By the definition of $\mu_0^{m,\vareps}$,
	\begin{align*}
		\int e^{t\Mcal_{>\Lambda}(u)} d\mu_{0}^{m,\vareps}(u) &= \frac{1}{Z_{0}^{m,\vareps}} \int e^{t \Mcal_{>\Lambda}(u)} \mathds{1}_{\left\{m-\vareps<\Mcal(u)<m+\vareps\right\}} d\mu_0(u) \\
		&=\frac{1}{Z_{0}^{m,\vareps}}\int e^{t\Mcal_{>\Lambda}(u)} \left( \int \mathds{1}_{\left\{m-\vareps-\varrho_\Lambda<\Mcal_{\leq \Lambda}(u)<m+\vareps-\varrho_\Lambda\right\}} d\mu_0^{\leq \Lambda}(u)\right) d\mu_0^{>\Lambda}(u),
	\end{align*}
	where $$\varrho_\Lambda=\Mcal_{>\Lambda}(u).$$ 
	Denote $g_\Lambda$ the density function of $\Mcal_{\leq \Lambda}(u)$ with respect to $\mu_0^{\leq \Lambda}$. We rewrite the above quantity as
	\[
	\frac{1}{Z_{0}^{m,\vareps}} \int e^{t\Mcal_{>\Lambda}(u)} \left(\int_{m-\vareps-\varrho_\Lambda}^{m+\vareps-\varrho_\Lambda} g_\Lambda(x)dx\right)d\mu_0^{>\Lambda}(u).
	\]
	Let $\psi_\Lambda$ be the characteristic function of $\Mcal_{\leq \Lambda}(u)$ with respect to $\mu_0^{\leq \Lambda}$. We compute
	\begin{align*}
		\psi_\Lambda(s) &= \E_{\mu_0^{\leq \Lambda}}[e^{is\Mcal_{\leq \Lambda}(u)}] \\
		&=\int e^{is\Mcal_{\leq \Lambda}(u)} d\mu_0^{\leq \Lambda}(u) \\
		&=\int e^{is \sum_{\lambda_j\leq \Lambda} |\alpha_j|^2-\lambda_j^{-1}} \prod_{\lambda_k \leq \Lambda} \frac{\lambda_k}{\pi}e^{-\lambda_k |\alpha_k|^2} d\alpha_k \\
		&= \prod_{\lambda_j\leq \Lambda} \frac{e^{-is\lambda_j^{-1}}}{1-is\lambda_j^{-1}}.
	\end{align*}
	Since $\Lambda\geq \lambda_2$, we see that
	\[
	|\psi_\Lambda(s)| \leq \frac{1}{1+s^2\lambda_2^{-2}},
	\]
	hence $\|\psi_\Lambda\|_{L^1(\R)} \leq C\lambda_2$. Thus we deduce that $g_\Lambda$ is uniformly bounded (in $\Lambda$) and uniformly continuous. In particular, we have
	\[
	\frac{1}{2\vareps}\int_{m-\vareps-\varrho_\Lambda}^{m+\vareps-\varrho_\Lambda} g_\Lambda(x) dx \leq \|g_\Lambda\|_{L^\infty} \leq C\|\psi_\Lambda\|_{L^1(\R)}\leq C\lambda_2.
	\]
	On the other hand, Lemma \ref{lem-f0-m} implies for $\vareps>0$ sufficiently small,
	\[
	\frac{1}{2\vareps} Z_{0}^{m,\vareps} \geq \frac{f_0(m)}{2}.
	\]
	Thus for $\vareps>0$ small enough, we get
	\[
	\frac{1}{Z_{0}^{m,\vareps}} \int_{m-\vareps-\varrho_\Lambda}^{m+\vareps-\varrho_\Lambda} g_\Lambda(x) dx \leq \frac{2C\lambda_2}{f_0(m)},
	\]
	hence
	\[
	\int e^{t\Mcal_{>\Lambda}(u)} d\mu_{0}^{m,\vareps}(u) \leq \frac{2C\lambda_2}{f_0(m)} \int e^{t\Mcal_{>\Lambda}(u)} d\mu_0^{>\Lambda}(u).
	\]
	By the same argument as in the proof of Lemma \ref{lem-Lambda-sub}, we have
	\begin{align*}
		\int e^{t\Mcal_{>\Lambda}(u)} d\mu_0^{>\Lambda}(u) = \prod_{\lambda_j>\Lambda} e^{-t\lambda_j^{-1}} \frac{1}{1-t\lambda_j^{-1}}\leq C\prod_{\lambda_j>\Lambda} e^{t^2\lambda_j^{-2}} =C e^{t^2\sum_{\lambda_j>\Lambda} \lambda_j^{-2}}.
	\end{align*}
	For $0\leq \gamma<\frac{3s-2}{4s}$, we have
	\[
	\sum_{\lambda_j>\Lambda} \lambda_j^{-2} \leq \Lambda^{-2\gamma} {\Tr}[h^{-2+2\gamma}].
	\]
	In particular, we obtain
	\[
	\mu_0^{m,\vareps}(\Mcal_{>\Lambda}(u)>R) \leq Ce^{-tR+t^2 \Lambda^{-2}{\Tr}[h^{-2+2\gamma}]}.
	\]
	Taking $t=\nu \Lambda^\gamma$ with $\nu>0$ small yields
	\[
	(\text{I}) \leq Ce^{-c\Lambda^\gamma R}
	\]
	for any $0\leq \gamma <\frac{3s-2}{4s}$. The term $(\text{II})$ is treated in a similar manner and we prove \eqref{eq:boun-M-mu-renor-eps}. 
\end{proof}

We also have the following decay estimates.

\begin{lemma}[\bf Decay of the $L^4$-norm ] \label{lem-L4-fix-renor} \mbox{} \\
	Let $s>1$, $V$ satisfy Assumption \ref{assu-V}, $0\leq \gamma<\frac{3s-2}{4s}$, $m\in \R$, and assume $m>-{\Tr}[h^{-1}]$ if $s>2$. Then there exist $C,c>0$ such that for $\theta<\frac{1}{2}-\frac{1}{s}$, all $\Lambda \geq \lambda_2$, all $R>0$ and all $\vareps>0$ sufficiently small,
	\begin{align} \label{eq:boun-L4-mu-renor-eps}
	\mu_0^{m,\vareps} \left(u \in \Hc^\theta : \|P_{>\Lambda} u\|_{L^4}>R\right) \leq C e^{-c\Lambda^\rho R^2}.
	\end{align}
	In particular, we have
	\begin{align} \label{eq:boun-L4-mu-renor}
	\mu_0^m \left(u \in \Hc^\theta : \|P_{>\Lambda} u\|_{L^4}>R\right) \leq C e^{-c\Lambda^\rho R^2}.
	\end{align}
\end{lemma}

\begin{proof}
	Since $\mu_0^{m,\vareps} \to \mu_0^m$ as $\vareps \to 0^+$, it suffices to prove \eqref{eq:boun-L4-mu-renor-eps}. Let $t>0$ be a positive constant to be determined later. We estimate
	\begin{align*}
	\mu_0^{m,\vareps} \left(\|P_{>\Lambda} u\|_{L^4}>R\right) &\leq e^{-tR^2} \int e^{t\|P_{>\Lambda} u\|^2_{L^4}} d\mu_0^{m,\vareps}(u) \\
	&=e^{-tR^2}\sum_{k\geq 0} \frac{t^k}{k!} \int \|P_{>\Lambda}u\|^{2k}_{L^4}d\mu_0^{m,\vareps}(u).
	\end{align*}
	We have
	\begin{align*}
	\int &\|P_{>\Lambda}u\|^{2k}_{L^4}d\mu_0^{m,\vareps}(u) \\
	&=\frac{1}{Z_0^{m,\vareps}} \int \|P_{>\Lambda}u\|^{2k}_{L^4} \mathds{1}_{\left\{m-\vareps<\Mcal(u)<m+\vareps\right\}} d\mu_0(u) \\
	&=\frac{1}{Z_0^{m,\vareps}} \int \|P_{>\Lambda}u\|^{2k}_{L^4} \left( \int \mathds{1}_{\left\{m-\vareps-\varrho_\Lambda<\Mcal_{\leq \Lambda} (u)<m+\vareps-\varrho_\Lambda\right\}} d\mu_0^{\leq \Lambda}(u)\right) d\mu_0^{>\Lambda}(u) \\
	&=\frac{1}{Z_0^{m,\vareps}} \int \|P_{>\Lambda}u\|^{2k}_{L^4} \left(\int_{m-\vareps-\varrho_\Lambda}^{m+\vareps-\varrho_\Lambda} g_\Lambda(x)dx\right) d\mu_0^{>\Lambda}(u),
	\end{align*}
	where $g_{\Lambda}$ is the density function of $\Mcal_{\leq \Lambda} (u)$ with respect to $\mu_0^{\leq \Lambda}$ and $\varrho_\Lambda=\Mcal_{>\Lambda}(u)$. As in the proof of Lemma \ref{lem-mu-fix-renor}, we have for $\vareps>0$ sufficiently small,
	\[
	\frac{1}{Z_0^{m,\vareps}} \int_{m-\vareps-\varrho_\Lambda}^{m+\vareps-\varrho_\Lambda} g_{\Lambda}(x)dx\leq \frac{2C\lambda_2}{f_0(m)}
	\]
	hence
	\[
	\int \|P_{>\Lambda} u\|^{2k}_{L^4} d\mu_0^{m,\vareps}(u) \leq \frac{2C\lambda_2}{f_0(m)} \int \|P_{>\Lambda} u\|^{2k}_{L^4} d\mu_0^{> \Lambda}(u).
	\]
	By the same argument as in the proof of Lemma \ref{lem-expo-p}, we have
	\[
	\int \|P_{>\Lambda} u\|^{2k}_{L^4} d\mu_0^{>\Lambda}(u) \leq 2! k! B_{\Lambda,0,4}^k, \quad \forall k\geq 0,
	\]
	where $B_{\Lambda,0,4}$ is as in \eqref{eq:B-Lamb}. In particular, we have for $\vareps>0$ sufficiently small,
	\[
	\int \|P_{>\Lambda}u\|^{2k}_{L^4}d\mu_0^{m,\vareps}(u)\leq C k! B_{\Lambda,0,4}^k, \quad \forall k\geq 0,
	\]
	hence
	\[
	\mu_0^{m,\vareps}\left(\|P_{>\Lambda}u\|_{L^4}>R\right) \leq Ce^{-tR^2}\sum_{k\geq 0} (tB_{\Lambda,0,4})^k.
	\]
	Arguing exactly as in the proof of Lemma \ref{lem-L4-Lambda}, we obtain \eqref{eq:boun-L4-mu-renor-eps}. 
\end{proof}

\begin{lemma}[\bf Decay of $\mathcal{W}^{\beta,p}$-norm] \label{lem-beta-p-fix-renor} \mbox{}\\
	Let $s>1$, $V$ satisfy Assumption \ref{assu-V}, $0\leq \beta<\frac{1}{2}$, $p>\max\left\{2,\frac{4}{s(1-2\beta)}\right\}$ be an even integer, $m\in \R$, and assume $m>-{\Tr}[h^{-1}]$ if $s>2$. Then there exists $c>0$ such that for $\theta<\frac{1}{2}-\frac{1}{s}$, all $\Lambda\geq \lambda_1$ sufficiently large, all $R>0$, and all $\vareps>0$ sufficiently small,
	\begin{align} \label{eq:boun-beta-p-mu-renor-eps}
		\mu_0^{m,\vareps}\left( u\in \Hc^\theta : \|P_{\leq \Lambda} u\|_{\Wc^{\beta,p}}>R\right) \leq C\Lambda e^{-cR^2}.
	\end{align}
	In particular, we have
	\begin{align} \label{eq:boun-beta-p-mu-renor}
		\mu_0^{m}\left( u\in \Hc^\theta : \|P_{\leq \Lambda} u\|_{\Wc^{\beta,p}}>R\right) \leq C\Lambda e^{-cR^2}.
	\end{align}
\end{lemma}

\begin{proof}
	We only need to prove \eqref{eq:boun-beta-p-mu-renor-eps}. We estimate
	\begin{align*}
		\mu_0^{m,\vareps}(\|P_{\leq \Lambda} u\|_{\Wc^{\beta,p}}>R) &\leq e^{-tR^2} \int e^{t\|P_{\leq \Lambda} u\|^2_{\Wc^{\beta,p}}} d\mu_0^{m,\vareps}(u) \\
		&= e^{-tR^2} \sum_{k\geq 0} \frac{t^k}{k!}\int \|P_{\leq \Lambda}u\|^{2k}_{\Wc^{\beta,p}} d\mu_0^{m,\vareps}(u).
		\end{align*}
		We have
		\begin{align*}
			\int &\|P_{\leq\Lambda}u\|^{2k}_{\Wc^{\beta,p}}d\mu_0^{m,\vareps}(u) \\
			&=\frac{1}{Z_0^{m,\vareps}} \int \|P_{\leq\Lambda}u\|^{2k}_{\Wc^{\beta,p}} \mathds{1}_{\left\{m-\vareps<\Mcal(u)<m+\vareps\right\}} d\mu_0(u) \\
			&=\frac{1}{Z_0^{m,\vareps}} \int \|P_{\leq\Lambda}u\|^{2k}_{\Wc^{\beta,p}} \left( \int \mathds{1}_{\left\{m-\vareps-\vartheta_\Lambda<\Mcal_{> \Lambda} (u)<m+\vareps-\vartheta_\Lambda\right\}} d\mu_0^{> \Lambda}(u)\right) d\mu_0^{\leq \Lambda}(u) \\
			&=\frac{1}{Z_0^{m,\vareps}} \int \|P_{\leq\Lambda}u\|^{2k}_{\Wc^{\beta,p}} \left(\int_{m-\vareps-\vartheta_\Lambda}^{m+\vareps-\vartheta_\Lambda} f_\Lambda(x)dx\right) d\mu_0^{\leq\Lambda}(u),
		\end{align*}
		with $f_{\Lambda}$ the density function of $\Mcal_{> \Lambda} (u)$ with respect to $\mu_0^{> \Lambda}$ and $\vartheta_\Lambda=\Mcal_{\leq\Lambda}(u)$. From Lemmas \ref{lem-f-Lamb} and \ref{lem-f0-m}, we infer that for $\vareps>0$ sufficiently small,
		\[
		\frac{1}{Z_0^{m,\vareps}} \int_{m-\vareps-\vartheta_\Lambda}^{m+\vareps-\vartheta_\Lambda} f_\Lambda(x)dx \leq \frac{2\|f_\Lambda\|_{L^\infty}}{f_0(m)} \leq C\Lambda.
		\]
		We deduce that for $\vareps>0$ small,
		\[
		\int \|P_{\leq\Lambda}u\|^{2k}_{\Wc^{\beta,p}}d\mu_0^{m,\vareps}(u) \leq C\Lambda \int \|P_{\leq \Lambda}u\|^{2k}_{\Wc^{\beta,q}} d\mu_0^{\leq \Lambda}(u).
		\]
		The integral in the right hand side is estimated exactly as in Lemma \ref{lem-expo-p} to get
		\[
		\int \|P_{\leq \Lambda}u\|^{2k}_{\Wc^{\beta,q}} d\mu_0^{\leq \Lambda}(u)\leq \left(\frac{p}{2}\right)! k! B^k_{\beta,p}, \quad \forall k\geq 0.
		\]
		In particular, we obtain
		\[
		\mu_0^{m,\vareps}(\|P_{\leq \Lambda} u\|_{\Wc^{\beta,p}}>R) \leq C\Lambda e^{-tR^2} \sum_{k\geq 0} (t B_{\beta,p})^k \leq C\Lambda e^{-cR^2}
		\]
		provided that $t>0$ is taken small enough.
	\end{proof}

\begin{remark}
	From the argument presented in the proofs of Lemmas \ref{lem-L4-fix-renor} and \ref{lem-beta-p-fix-renor} (see also the proof of Lemma \ref{lem-H-theta}), we can show that: for $s>1$, $\theta<\frac{1}{2}-\frac{1}{s}$, and $m>0$,
	\begin{align} \label{expo-deca-fix-mass}
		\mu_0^m(u \in \Hc^\theta : \|u\|_{\Hc^\theta}>\lambda) \leq Ce^{-c\lambda^2}
	\end{align}
	for some constants $C,c>0$. 
\end{remark}

\begin{proof}[Proof of Proposition \ref{prop-mu-fix-renor}]
	We only consider the harder case of focusing nonlinearity since the defocusing one is easier. Since $\mu_0^m$ is a probability measure, it is clear that $Z^m\geq 1$. It remains to prove 
	\begin{align}\label{Zm-focu}
		Z^m=\int e^{\frac{1}{2}\|u\|^4_{L^4}} d\mu_0^m(u)<\infty.
	\end{align}
	We have
	\[
	Z^m = \int_0^\infty \mu_{0}^m \left(e^{\frac{1}{2}\|u\|^4_{L^4}}>\lambda\right) d\lambda = C(\lambda_0) + \int_{\lambda_0}^\infty \mu_{0}^m \left(e^{\frac{1}{2}\|u\|^4_{L^4}}>\lambda\right) d\lambda
	\]
	for some $\lambda_0>0$ to be fixed later. To show the finiteness of 
	\[
	\int_{\lambda_0}^\infty \mu_{0}^m \left(e^{\frac{1}{2}\|u\|^4_{L^4}}>\lambda\right) d\lambda,
	\]
	we repeat the same line of reasoning as in the proof of Proposition \ref{pro:gc} using Lemmas \ref{lem-mu-fix-renor}, \ref{lem-L4-fix-renor}, and \ref{lem-beta-p-fix-renor}. Note that $\Mcal(u)=m$ on the support of $\mu_0^m$, hence $\|u\|^2_{L^2} = m+{\Tr}[h^{-1}] \in (0,\infty)$ when $s>2$. The only different point is the following estimate (see Lemma \ref{lem-beta-p-fix-renor}):
	\[
	\mu_0^m(\|P_{\leq \Lambda_0} u\|_{\Wc^{\beta,p}}> C(\log \lambda)^\sigma) \leq C\Lambda_0e^{-c(\log \lambda)^{2\sigma}},
	\]
	where the additional term $\Lambda_0$ is just $(\log \lambda)^l$ (see \eqref{eq:Lamb-0}) and can be absorbed by the exponential decay $e^{-c(\log \lambda)^{2\sigma}}$. The constant $\lambda_0$ is chosen so that $\Lambda_0$ is sufficiently large as needed to apply Lemmas \ref{lem-mu-fix-renor}, \ref{lem-L4-fix-renor}, and \ref{lem-beta-p-fix-renor}.
\end{proof}
	
\subsection{Invariance of Gibbs measures conditioned on mass}
\label{sec:invari Gibbs fix mass}
Thanks to the invariance of the standard Gibbs measure $\mu$ (see Theorem \ref{theo-Gibbs-meas}), we can deduce the invariance of the Gibbs measures conditioned on mass.

\begin{proposition}[\bf Invariance of canonical Gibbs measures] \label{prop-inva-fix-renor} \mbox{} \\
	Let $s>1$, $V$ satisfy Assumption \ref{assu-V}, $m\in \R$, and assume $m>-{\Tr}[h^{-1}]$ if $s>2$. Assume in addition that $s>\frac{8}{5}$ for the focusing nonlinearity. Then $\mu^m$ is invariant under the flow of \eqref{eq:intro NLS}.
\end{proposition}

The proof is based on the following lemma.
\begin{lemma}[\bf Approximate canonical measures] \label{lem-weak-conv-mu-a-eps} \mbox{} \\
	Let $s>1$, $V$ satisfy Assumption \ref{assu-V}, $m\in \R$, and assume $m>-{\Tr}[h^{-1}]$ if $s>2$. Assume in addition that $s>\frac{8}{5}$ for the focusing nonlinearity. Denote
	\[
	d\mu^{m,\vareps}(u) = \frac{1}{Z^{m,\vareps}} e^{\mp\frac{1}{2}\|u\|^4_{L^4}} d\mu_0^{m,\vareps}(u), \quad Z^{m,\vareps} = \int e^{\mp\frac{1}{2}\|u\|^4_{L^4}} d\mu_0^{m,\vareps}(u). 
	\]
	Then $\mu^{m,\vareps}$ converges weakly to $\mu^m$ as $\vareps \to 0^+$ in the sense that 
	\[
	\lim_{\vareps \to 0^+} \int F(u) d\mu^{m,\vareps}(u) = \int F(u) d\mu^m(u)
	\]
	for all bounded continuous functions $F: \Hc^\theta \to \R$. 
\end{lemma}

\begin{proof}
	We first show that $Z^{m,\vareps} \to Z^m$ as $\vareps \to 0^+$. To simplify the notation, we denote 
	$$G(u)=e^{\mp\frac{1}{2}\|u\|^4_{L^4}}$$
	and 
	$$w_\Lambda:= P_{\leq \Lambda} u.$$
	For $\Lambda \geq \lambda_1$ sufficiently large and $\vareps>0$, we estimate
	\begin{align*}
	|Z^{m,\vareps}-Z^m| &\leq \Big|\int G(u) - G(w_\Lambda) d\mu_0^{m,\vareps}(u)\Big| + \Big|\int G(w_\Lambda) d\mu_0^{m,\vareps}(u) -\int G(w_\Lambda) d\mu_0^m(u)\Big| \\
	&\quad+ \Big|\int G(u)-G(w_\Lambda) d\mu_0^m(u)\Big|=:(\text{I}) + (\text{II}) + (\text{III}).
	\end{align*}
	Regarding (II), we write
	\[
	\int G(w_\Lambda) d\mu_0^{m,\vareps}(u) = \int_0^\infty \mu_0^{m,\vareps} (G(w_\Lambda)>\lambda) d\lambda.
	\]
	For each $\Lambda \geq \lambda_1$, we have $\mu^{m,\vareps}_0(G(w_\Lambda)>\lambda) \to \mu^m(G(w_\Lambda)>\lambda)$ as $\vareps \to 0^+$. In addition, we have for $\vareps>0$ small,
	\begin{align*}
	\mu_0^{m,\vareps}(G(w_\Lambda)>\lambda) &\leq \frac{1}{\lambda^2}\int G^2(w_\Lambda) d\mu_0^{m,\vareps}(u) \\
	&= \frac{1}{\lambda^2} \int e^{\mp\|P_{\leq \Lambda} u\|^4_{L^4}} d\mu_0^{m,\vareps}(u) \\
	&\leq \frac{C\Lambda}{\lambda^2},
	\end{align*} 
	for $\Lambda$ sufficiently large. The dominated convergence theorem implies that for $\Lambda\geq \lambda_1$ sufficiently large, $(\text{II}) \to 0$ as $\vareps \to 0^+$. 
	
	To estimate (I), we let $\delta>0$. By continuity, there exists $\nu>0$ such that if $\|u-w_\Lambda\|_{L^4} < \nu$, then $|G(u)-G(w_\Lambda)|<\frac{1}{2}\delta$. We write
	\[
	(\text{I}) \leq \Big|\int_{\|u-w_\Lambda\|_{L^4} \geq \nu} G(u)-G(w_\Lambda) d\mu_0^{m,\vareps}(u)\Big| + \Big|\int_{\|u-w_\Lambda\|_{L^4}<\nu} G(u)-G(w_\Lambda) d\mu_0^{m,\vareps}(u)\Big|.
	\]
	Since $G(u), G(w_\Lambda) \in L^2(d\mu_0^{m,\vareps})$ uniformly in $\Lambda$ for $\vareps>0$ sufficiently small, we have
	\[
	(\text{I}) \leq \|G(u)-G(w_\Lambda)\|_{L^2(d\mu_0^{m,\vareps})} \left(\mu_0^{m,\vareps}\left(\|u-w_\Lambda\|_{L^4} \geq \nu\right)\right)^{1/2} + \frac{\delta}{2}.
	\]
	By Lemma \ref{lem-L4-fix-renor}, we also have for $\vareps>0$ small,
	\begin{align*}
	\mu_0^{m,\vareps}\left(\|u-w_\Lambda\|_{L^4}\geq \nu\right) = \mu_0^{m,\vareps} \left(\|P_{>\Lambda} u\|_{L^4}>\nu\right) \leq Ce^{-c\Lambda^\rho \nu^2}
	\end{align*}
	for any $0\leq \rho<\frac{s-1}{2s}$. Fix $0<\rho<\frac{s-1}{2s}$, we deduce, for $\Lambda$ large enough, that 
	\[
	\|G(u)-G(w_\Lambda)\|_{L^2(d\mu_0^{m,\vareps})} \left(\mu_0^{m,\vareps}\left(\|u-w_\Lambda\|_{L^4} \geq \nu\right)\right)^{1/2} <\frac{\delta}{2}
	\] 
	for all $\vareps>0$ sufficiently small. This shows that $(\text{I}) \to 0$ as $\Lambda \to \infty$ provided that $\vareps>0$ is taken sufficiently small. A similar argument goes for $(\text{III})$ and we prove $Z^{m,\vareps} \to Z^m$. 
	
	Now we write, for $F$ a test function,
	\begin{align*}
	\int F(u) d\mu^{m,\vareps}(u) &= \frac{1}{Z^{m,\vareps}} \int F(u) G(u) d\mu_0^{m,\vareps}(u) \\
	&= \left(\frac{1}{Z^{m,\vareps}} - \frac{1}{Z^m}\right) \int F(u) G(u) d\mu_0^{m,\vareps}(u) + \frac{1}{Z^m}\int F(u) G(u) d\mu_0^{m,\vareps}(u).
	\end{align*}
	The first term in the right hand side goes to zero as $\vareps \to 0^+$ because $Z^{m,\vareps} \to Z^m$ as $\vareps \to 0^+$ and
	\[
	\int F(u) G(u) d\mu_0^{m,\vareps}(u) \leq \|F\|_{L^\infty}\|G(u)\|_{L^1(d\mu_0^{m,\vareps})} \leq C
	\]
	for some constant $C>0$ independent of $\vareps>0$ small. 
	
	For the second term, we estimate
	\begin{align*}
	\Big|\int F(u) G(u) d\mu_0^{m,\vareps}(u) &-\int F(u)G(u) d\mu_0^m(u)\Big| \\
	&= \Big|\int F(u) (G(u)-G(w_\Lambda)) + G(w_\Lambda) (F(u)-F(w_\Lambda)) d\mu_0^{m,\vareps}(u)\Big| \\
	&\quad + \Big|\int F(w_\Lambda) G(w_\Lambda) d\mu_0^{m,\vareps}(u) - \int F(w_\Lambda) G(w_\Lambda) d\mu_0^m(u)\Big| \\
	&\quad + \Big|\int F(u) (G(u)-G(w_\Lambda)) + G(w_\Lambda) (F(u)-F(w_\Lambda)) d\mu_0^{m}(u)\Big|.
	\end{align*}
	Arguing as above, we deduce that 
	\[
	\int F(u) G(u) d\mu_0^{m,\vareps}(u) \to \int F(u)G(u) d\mu_0^m(u) \text{ as } \vareps \to 0^+.
	\]
	The proof is complete.
\end{proof}

\begin{proof}[Proof of Proposition \ref{prop-inva-fix-renor}] 
	
	{\bf Invariance of the approximate canonical measure.} We first show that $\mu^{m,\vareps}$ is invariant under the flow of \eqref{eq:intro NLS}. In fact, we can rewrite $\mu^{m,\vareps}$ as
	\[
	d\mu^{m,\vareps}(u)=\frac{1}{Z^{m,\vareps}} \mathds{1}_{\{m-\vareps<\Mcal(u)<m+\vareps\}} d\mu(u).
	\]
	Here $\mu$ is the standard Gibbs measure which is invariant under the flow of \eqref{eq:intro NLS}. Now let $A$ be a measurable set in $\Hc^\theta$ with $\theta <\frac{1}{2}-\frac{1}{s}$ as in Theorem \ref{theo-Gibbs-meas}. We have
	\begin{align*}
		\mu^{m,\vareps}(A) &= \frac{1}{Z^{m,\vareps}} \int_A \mathds{1}_{\{m-\vareps<\Mcal(u)<m+\vareps\}} d\mu(u) \\
		&= \frac{1}{Z^{m,\vareps}} \int \mathds{1}_{\{m-\vareps<\Mcal(u)<m+\vareps\} \cap A} d\mu(u) \\
		&= \frac{1}{Z^{m,\vareps}} \mu\left(\{m-\vareps<\Mcal(u)<m+\vareps\} \cap A \right) \\
		&= \frac{1}{Z^{m,\vareps}} \mu\left( \Phi(t) \left(\{m-\vareps<\Mcal(u)<m+\vareps\} \cap A \right)\right) \tag{\text{invariance of $\mu$}} \\
		&=\frac{1}{Z^{m,\vareps}} \mu \left( \{m-\vareps<\Mcal(u)<m+\vareps\} \cap \Phi(t) (A)\right) \tag{mass conservation} \\
		&=\frac{1}{Z^{m,\vareps}} \int_{\Phi(t)(A)} \mathds{1}_{\{m-\vareps<\Mcal(u)<m+\vareps\}} d\mu(u) \\
		&= \mu^{m,\vareps}(\Phi(t)(A)),
	\end{align*}
	where $\Phi(t)$ is the solution map of \eqref{eq:intro NLS}. This shows that $\mu^{m,\vareps}$ is invariant under the flow of \eqref{eq:intro NLS}.
	
	\medskip
	
	\noindent {\bf Invariance of the canonical Gibbs measure.} We now prove the invariance of $\mu^m$ under the flow of \eqref{eq:intro NLS}. By the same reasoning as in the proof of Theorem \ref{theo:almo-gwp-supe} using \eqref{expo-deca-fix-mass}, it suffices to show 
	\[
	\mu^m(U) \leq \mu^m(\Phi(t)(U)) 
	\]
	for any closed set $U$ of $\Hc^\theta$ which is bounded in $\Hc^{\theta_1}$ with $\theta<\theta_1<\frac{1}{2}-\frac{1}{s}$ and for all $t\in \R$. The problem is further reduced to show
	\begin{align} \label{inva-fix-mass}
	\mu^m(U) \leq \mu^m(\Phi(t)(U)), \quad \forall t\in [0,\delta]
	\end{align}
	with $\delta>0$ sufficiently small. By the local theory, for each $\vareps>0$, there exists $0<c\ll 1$ such that
	\begin{align} \label{inva-prof-3}
	\Phi(t)(U+B_{\gamma,c\vareps}) \subset \Phi(t) (U) + B_{\theta,\vareps},
	\end{align}
	where $B_{\theta, \vareps}$ is the ball in $\Hc^\theta$ centered at zero and of radius $\vareps$. Here $\gamma=\theta$ for $s>2$ (see \eqref{cont-prop}) and $\gamma=\beta$ for $1<s\leq 2$ (see \eqref{cont-prop-sub-beta}). Since $\mu^{m,\vareps} \rightharpoonup \mu^m$ weakly as $\vareps \to 0^+$, we have
	\begin{align*}
	\mu^m(U) &\leq \mu^m(U+B_{\gamma,c\vareps}) \\
	&\leq\liminf_{\vareps \to 0^+} \mu^{m,\vareps}(U+B_{\gamma, c\vareps}) \tag{$\mu^{m,\vareps} \rightharpoonup \mu^m$ weakly}\\
	&= \liminf_{\vareps \to 0^+} \mu^{m,\vareps}\left( \Phi(t) (U+B_{\gamma,c\vareps})\right) \tag{\text{invariance of $\mu^{m,\vareps}$}} \\
	&\leq \liminf_{\vareps \to 0^+} \mu^{m,\vareps} \left( \Phi(t)(U) + B_{\theta, \vareps} \right) \tag{due to \eqref{inva-prof-3}} \\
	&\leq \limsup_{\vareps \to 0^+} \mu^{m,\vareps} \left( \Phi(t)(U) + B_{\theta, \vareps} \right) \\
	&\leq \limsup_{\vareps \to 0^+} \mu^{m,\vareps} \left( \left( \Phi(t)(U) + \overline{B}_{\theta, \vareps} \right)\right) \\
	&\leq \mu^m (\Phi(t)(U) + \overline{B}_{\theta, \vareps}). \tag{$\mu^{m,\vareps} \rightharpoonup \mu^m$ weakly}
	\end{align*}
	Letting $\vareps \to 0$, we get \eqref{inva-fix-mass}. This proves the invariance of $\mu^m$ under the flow of \eqref{eq:intro NLS}.
\end{proof}

\begin{proof}[Proof of Theorem \ref{theo-Gibbs-fix-mass}]
	This follows by combining Propositions \ref{prop-mu-fix-renor} and \ref{prop-inva-fix-renor}.
\end{proof}

\begin{remark}[Higher non-linearities] \label{rem-meas-inva-mass-general}
	Thanks to the argument presented in this section and Remark \ref{rem-Gibbs-general}, one can construct the Gibbs measures with a fixed renormalized mass associated to NLS with more general nonlinearity \eqref{eq:NLS-general}. Moreover, based on the invariance of grand-canonical measures mentioned in Remarks \ref{rem-meas-inva-sup} and \ref{rem-meas-inva-sub}, these fixed renormalized mass Gibbs measures are indeed invariant under the flow of \eqref{eq:NLS-general}. \hfill$\diamond$
\end{remark}

\newpage

\appendix

\section{Bounds on the covariance operator}
\label{sec:app}
\setcounter{equation}{0}

\begin{lemma} [\textbf{Schatten norm of the Green function}]\label{lem-trac-h-p}\mbox{}\\
	Let $s>0$ and $p>\frac{1}{2}+\frac{1}{s}$. Then we have
	\begin{align} \label{trac-h-p-prof}
	{\Tr}[h^{-p}] =\sum_{j\geq 1} \lambda_j^{-p}<\infty.	
	\end{align}
\end{lemma}

\begin{proof} 
We follow an argument of \cite[Example 3.2]{LewNamRou-14d}. Let $\lambda_1>0$ be the first eigenvalue of $h$. We have
\[
h+\lambda_1 \leq 2h,
\]
hence
\[
{\Tr}[h^{-p}] \leq 2^p {\Tr}[(h+\lambda_1)^{-p}].
\]
Thanks to a version of Lieb--Thirring's inequality \cite[Theorem 1]{DolFelLosPat-06}, we have
\[
{\Tr}[(h+\lambda_1)^{-p}] \leq \frac{1}{2\pi} \iint_{\R\times \R} \frac{dx d\xi}{(|\xi|^2+V(x)+\lambda_1)^p}.
\]
Using Assumption \ref{assu-V} and the layer-cake representation, one finds that 
$$ \iint_{\R\times \R} \frac{dx d\xi}{(|\xi|^2+V(x)+\lambda_1)^p} = \int_{\R} \left( V(x) + \lambda_1 \right)^{1/2 - p } dx < +\infty$$
under our stated assumptions. Here is another short proof. We have
\begin{align*}
\iint_{\R\times \R} \frac{dx d\xi}{(|\xi|^2+V(x)+\lambda_1)^p} &= \int_{\R} \int_{|x|\leq 1} \frac{dx d\xi}{(|\xi|^2+V(x)+\lambda_1)^p} + \int_{\R} \int_{|x|\geq 1} \frac{dx d\xi}{(|\xi|^2+V(x)+\lambda_1)^p} \\
&\leq \int_{\R} \int_{|x|\leq 1} \frac{dxd\xi}{(|\xi|^2+\lambda_1)^p} + C\int_{\R} \int_{|x|\geq 1} \frac{dxd\xi}{(|\xi|^2+ |x|^s +1)^p} \\
&=(\text{I}) + (\text{II}),
\end{align*}
where $V(x)\geq 0$ for all $x\in \R$ and $V(x) \geq C|x|^s$ for $|x|\geq 1$. Since $p>\frac{1}{2}+\frac{1}{s}$, it is clear that $(\text{I})<\infty$. To see that $(\text{II})<\infty$, we take 
$$\theta =\frac{4}{2+s}$$ 
so that 
$$\frac{1}{2-\theta}=\frac{2}{\theta s} =\frac{1}{2}+\frac{1}{s}.$$ 
By Young's inequality, we have
\begin{align*}
|\xi|^2+|x|^s +1 &\gtrsim (1+|\xi|^2) + (1+|x|^s) \\
&\gtrsim (1+|\xi|)^2 + (1+|x|)^s \\
&\gtrsim (1+|\xi|)^{2-\theta} (1+|x|)^{\frac{\theta s}{2}}.
\end{align*}
It follows that
\begin{align*}
\int_{\R}\int_{|x|\geq 1} \frac{dxd\xi}{(|\xi|^2+|x|^s+1)^p} \lesssim \Big( \int_{\R} \frac{d\xi}{(1+|\xi|)^{(2-\theta)p}}\Big) \Big(\int_{\R} \frac{dx}{(1+|x|)^{\frac{\theta sp}{2}}}\Big) <\infty.
\end{align*}
Note that $(2-\theta)p=\frac{\theta sp}{2}=\frac{p}{\frac{1}{2}+\frac{1}{s}}>1$.
\end{proof}

\section{Basic functional-analytic estimates}
\label{sec:app2}
\setcounter{equation}{0}

We collect here known functional inequalities used repeatedly in the paper. We start with the following norm equivalence due to \cite{YajZha-01}.

\begin{lemma}[\textbf{Norm equivalence}] \mbox{} \\
	Let $s>1$, $V$ satisfy Assumption \ref{assu-V}, $\beta>0$, and $1<p<\infty$. Then we have
	\begin{align} \label{equi-norm}
	\|h^{\beta/2} u\|_{L^p} \sim \|\scal{D}^\beta u\|_{L^p} + \|\scal{x}^{\beta s/2} u\|_{L^p},
	\end{align}
	where $\scal{D}:= \sqrt{1-\partial^2_x}$.	
\end{lemma}

\begin{proof}
	In \cite[Lemma 2.4]{YajZha-01}, the above norm equivalence was proved for $s>2$ using pseudo-differential calculus. However, the same proof applies to $1<s\leq 2$ as well.
\end{proof}

\begin{lemma}[\textbf{Sobolev embedding}]\mbox{}\\ 
	Let $s>1$ and $V$ satisfy Assumption \ref{assu-V}. 
	
	\medskip
	
	\noindent (i) If $1<r<\overline{r}<\infty$ and $\beta>0$ satisfy $\frac{1}{\overline{r}} \geq \frac{1}{r}-\beta$, then $\Wc^{\beta,r}\subset L^{\overline{r}}(\R)$. 
	
	\medskip
	
	\noindent (ii) If $\beta>\frac{1}{r}$, then $\Wc^{\beta,r} \subset L^\infty(\R)$.
\end{lemma}

\begin{proof}
	Direct consequence of embeddings for standard Sobolev spaces on $\R$ and the norm equivalence~\eqref{equi-norm}.
\end{proof}

\begin{lemma}[\textbf{Fractional product rule}] \label{lem-prod-rule}\mbox{}\\
	Let $s>1$, $V$ satisfy Assumption \ref{assu-V}, $\beta>0$, and $1<p<\infty$. Then
	\begin{align} \label{prod-rule}
	\|h^{\beta/2}(fg)\|_{L^2} \leq C\|f\|_{L^{q_1}} \|h^{\beta/2} g\|_{L^{q_2}} + C\|g\|_{L^{r_1}} \|h^{\beta/2} f\|_{L^{r_2}}
	\end{align}
	provided that
	\[
	\frac{1}{p} = \frac{1}{q_1}+\frac{1}{q_2} = \frac{1}{r_1}+\frac{1}{r_2}, \quad q_2,r_2 \in (1,\infty), \quad q_1,r_1 \in (1,\infty].
	\]
\end{lemma}

\begin{proof}
	Direct consequence of~\eqref{equi-norm} and the following product rule (see e.g.,~\cite[Proposition 1.1 of Chapter 2]{Taylor}):
	\[
	\|\scal{D}^\beta(fg)\|_{L^2} \leq C\|f\|_{L^{q_1}} \|\scal{D}^\beta g\|_{L^{q_2}} + C\|g\|_{L^{r_1}} \|\scal{D}^\beta f\|_{L^{r_2}}.
	\]
\end{proof}

We next recall some Strichartz estimates whose proofs can be found in~\cite{Fujiwara,Carles} for the case $s\leq 2$ and in~\cite{YajZha-04} for $s>2$.

\begin{definition}[\textbf{Strichartz-admissible pairs}]\mbox{}\\
	A pair $(p,q)$ is called Strichartz-admissible if
	\[
	\frac{2}{p}+\frac{1}{q} = \frac{1}{2}, \quad q\in [2,\infty].
	\]
\end{definition}

\begin{proposition}[\textbf{Strichartz estimates}]\mbox{}\\ 
	\noindent (i) \cite{Fujiwara, Carles} Let $1<s\leq 2$ and $(p,q)$ be a Strichartz-admissible pair. Then there exists $C>0$ such that
	\begin{align} \label{homo-str-est-sub}
	\|e^{-ith} f\|_{L^p((-1,1), L^q(\R))} \leq C \|f\|_{L^2(\R)}.
	\end{align}
	Moreover, for any Strichartz-admissible pairs $(p,q)$ and $(a,b)$, there exists $C>0$ such that
	\begin{align}\label{inho-str-est-sub}
	\Big\|\int_0^t e^{-i(t-\tau)h} F(\tau) d\tau\Big\|_{L^p((-1,1), L^q(\R))} \leq C\|F\|_{L^{a'}((-1,1), L^{b'}(\R))}.
	\end{align}
	
	\medskip
	
	\noindent(ii) Let $s>2$, $(p,q)$ be a Strichartz-admissible pair, and $\sigma>\frac{2}{p}\left(\frac{1}{2}-\frac{1}{s}\right)$. Then there exists $C>0$ such that
		\begin{align} \label{homo-str-est-supe}
		\|e^{-ith} f\|_{L^p((-1,1), L^q(\R))} \leq C \|f\|_{\Hc^\sigma}.
		\end{align}
		In particular, we have
		\begin{align}\label{inho-str-est-supe}
		\Big\|\int_0^t e^{-i(t-\tau)h} F(\tau) d\tau\Big\|_{L^p((-1,1), L^q(\R))} \leq C\|F\|_{L^1((-1,1), \Hc^{\sigma})}.
		\end{align}
\end{proposition}

For potentials that grow at most quadratically at infinity, i.e., $|\partial^\alpha V(x)| \leq C_\alpha$ for $|\alpha|\geq 2$, it is known (see \cite{Fujiwara}) that the following dispersive estimate holds
\begin{align} \label{dis-est}
\|e^{-ith} f\|_{L^\infty} \leq C|t|^{-1/2} \|f\|_{L^{1}}, \quad \forall |t|\leq \delta
\end{align}
with some small constant $\delta>0$. Using this and the unitary property, we obtain Strichartz estimates (see e.g., \cite{KeeTao-98}) on the time interval $[-\delta,\delta]$. After dividing $(-1,1)$ into intervals of size $2\delta$ and applying Strichartz estimates for these intervals, we can sum over all sub-intervals to get \eqref{homo-str-est-sub} and \eqref{inho-str-est-sub}. 

For super-quadratic potentials, there is a loss of derivatives in \eqref{homo-str-est-supe}. Moreover, dispersive estimates as in~\eqref{dis-est} are no longer available due to the unboundedness and lack of smoothness of the kernel of $e^{-ith}$ (see \cite{Yajima-96}). Thus the above inhomogeneous Strichartz estimates~\eqref{inho-str-est-supe} may not hold true for $(a,b) \ne (\infty,2)$.


\section{$L^p$-boundedness of the smoothened spectral projectors}
\label{sec:app3}
\setcounter{equation}{0}
We here discuss the boundedness of the smoothened spectral projector $Q_\Lambda$ (defined in \eqref{eq:Q-Lamb}) as an operator on $L^p (\R)$. This comes from known facts that we collect in the following lemma.

\begin{lemma}[\textbf{$L^p$-boundedness of smooth spectral projections}]\mbox{}\\
    Let $1<p<\infty$. Then there exists $C>0$ such that for all $\Lambda \geq \lambda_1$,
    \[
    \|Q_\Lambda\|_{L^p \to L^p} \leq C.
    \]
\end{lemma}

\begin{proof}
    The proof follows from the $L^p$-boundedness of pseudo-differential operators with symbol in $S(1,g)$ (see e.g., \cite[Theorem 22.3]{Wong-14}). Here 
    $$g=\frac{dx^2}{\scal{x}^2} + \frac{d\xi^2}{\scal{\xi}^2}$$ 
    is a metric and $S(1,g)$ is the H\"ormander symbol class consisting of functions $a \in C^\infty(\mathbb{R}^2)$ such that
    \[
    |\partial^j_x \partial^k_\xi a(x,\xi)|\leq C_{jk} \scal{x}^{-j} \scal{\xi}^{-k}, \quad \forall (x,\xi) \in \mathbb{R}^2.
    \]
    We write the symbol of $Q_\Lambda$ as 
    $$p_\Lambda(x,\xi) = \chi(q_\Lambda(x,\xi))$$
    with 
    $$q_\Lambda(x,\xi) =\Lambda^{-1}(\xi^2 +V(x)),$$
    and we will show that $p_\Lambda \in S(1,g)$ uniformly in $\Lambda \geq \lambda_1$. To see this, we observe that for $j,k\geq 0$, there exist $C_j,C_k>0$ independent of $\Lambda$ such that
    \begin{align} \label{est-q-Lambda}
        |\partial^j_xq_\Lambda(x,\xi)|\leq C_j\scal{x}^{-j}, \quad |\partial^k_\xi q_\Lambda(x,\xi)|\leq C_k\scal{\xi}^{-k}
    \end{align}
    for all $(x,\xi)$ in the support of $p_\Lambda$. Note that the mixed derivatives 
    $$\partial^j_x\partial^k_\xi q_\Lambda(x,\xi)=0$$
    for $j, k \ne 0$. In fact, for the $x$-derivative, it is straightforward when $j=0$. For $j\geq 1$, we have $\partial^j_x q_\Lambda(x,\xi) = \Lambda^{-1} \partial^j_x V(x)$. By Assumption~\ref{assu-V}, we have
    $$
    |\partial^j_xq_\Lambda(x,\xi)| \leq C_j \Lambda^{-1} \scal{x}^{s-j} \leq C_j \scal{x}^{-j}.
    $$
    Actually, for $|x|\leq 1$, we have $\scal{x}^s \leq 2^{s/2}$ and $\Lambda^{-1} \leq \lambda_1^{-1}$, hence $\Lambda^{-1} \scal{x}^s \leq C$. On the other hand, for $|x|\geq 1$, since $V(x)\leq \Lambda$ on the support of $p_\Lambda$, the assumption on $V$ gives $\frac{1}{C} \scal{x}^s \leq V(x) \leq \Lambda$, hence $\Lambda^{-1}\scal{x}^s \leq C$. For the $\xi$-derivative, it is obvious when $k=0$ and $k\geq 3$ since $\partial^k_\xi q_\Lambda(x,\xi)=0$ for $k\geq 3$. For $k=1,2$, we have $\partial^k_\xi q_\Lambda(x,\xi) = 2\Lambda^{-1} \xi^{2-k}$. On the support of $p_\Lambda$, we have $|\xi|\leq \Lambda^{\frac{1}{2}}$, hence
    $$
    |\partial^k_\xi q_\Lambda(x,\xi)|\leq C\Lambda^{-\frac{k}{2}} \leq C\scal{\xi}^{-k}. 
    $$
    Here, to obtain the second estimate, we have used the fact that
    $$
    \frac{\lambda_1}{1+\lambda_1} \Lambda^{-1} \leq (1+\Lambda)^{-1} \leq \scal{\xi}^{-2}
    $$
    as $\Lambda \geq \lambda_1$ and $|\xi|^2\leq \Lambda$.
    
    The result now follows from \eqref{est-q-Lambda} and the Fa\`a di Bruno formula saying that $\partial^j_x g(f(x))$ is a linear combination of 
    $$
    g^{(n)}(f(x)) \left(f'(x)\right)^{n_1} \left(f''(x)\right)^{n_2} \cdots \left(f^{(n)}(x)\right)^{n_j},
    $$
    where $n=n_1 +\cdots + n_j$ and the sum is over all partitions of $j$, i.e., all $j$-tuples $(n_1, \cdots, n_j)$ satisfying $n_1 + 2n_2 +\cdots +j n_j =j$. 
\end{proof}

\section{An estimate on the number of eigenvalues}
\label{sec:app4}
\setcounter{equation}{0}

We are interested in the number of eigenvalues of the Schr\"odinger operator $h=-\partial^2_x+V(x)$ below a certain energy threshold which is needed in the construction of Gaussian measure conditioned on mass (see Lemma \ref{lem-f-Lamb}). 

\begin{lemma}[\textbf{Cwikel-Lieb-Rozenbljum law}] \label{lem-numb-eigen}\mbox{}\\
	Let $s>1$ and $V$ satisfy Assumption \ref{assu-V}. Then, for $\Lambda>0$ sufficiently large, we have that
	\[
	c \Lambda^{\frac{1}{2}+\frac{1}{s}} \leq \#\{\lambda_j : \lambda_j\leq \Lambda\} \leq C \Lambda^{\frac{1}{2}+\frac{1}{s}} 
	\]
	for fixed constants $c,C >0.$
\end{lemma} 

\begin{proof}
	To see this, we recall the following result due to Rozenbljum \cite{Rozenbljum-74}. Let $V(x)\geq 1$ and $V(x) \to +\infty$ as $|x|\to \infty$. Assume that the following conditions hold:
	
	\medskip
	
	\noindent (V1) $\sigma(2\Lambda) \leq C \sigma(\Lambda)$, where $\sigma(\Lambda):=|\{x\in \R : V(x) <\Lambda\}|$.
	
	\medskip
	
	\noindent (V2)  $V(x) \leq CV(y)$ almost everywhere when $|x-y|<1$.
	
	\medskip
	
	\noindent (V3) There exist a continuous function $\eta(t)\geq 0$, $0\leq t<1$, $\eta(0)=0$, and an index $\beta \in [0,1/2)$ such that
	\[
	\int_{|x-y|\leq 1 \atop |x+z-y|\leq 1} |V(x+z) -V(x)|dx <\eta(|z|) |z|^{2\beta} (V(y))^{1+\beta}
	\]
	for any $y,z\in \mathbb{R}$ and $|z|<1$. Then, for $\Lambda$ sufficiently large, we have that
	\begin{align} \label{N-Lamb}
		\#\{\lambda_j : \lambda_j \leq \Lambda\} \propto \int_{\mathbb{R}} (\Lambda- V(x))_+^{1/2} dx,
	\end{align}
	where $(f(x))_+:= \max\{f(x),0\}$. 
	
	We will check that the above conditions are fulfilled for $V$ as in Assumption \ref{assu-V}. First, we observe that by using the change of variable $u(t,x)\mapsto e^{-iat} u(t,x)$ for \eqref{eq:intro NLS} with some constant $a>0$, we can assume (in addition to Assumption \ref{assu-V}) that $V(x)\geq 1$ for all $x\in \mathbb{R}$. From this, we infer that there exists $C\geq 1$ such that
	\begin{align} \label{prop-V}
		\frac{1}{C}\langle x \rangle^s \leq V(x)\leq C\langle x\rangle^s, \quad \forall x\in \mathbb{R}.
	\end{align}
	
	Let us now check the conditions (V1)-(V3). 
	
	$\bullet$ Checking (V1):  From \eqref{prop-V}, $V(x)< \Lambda$ implies $|x|<\langle x \rangle <(C\Lambda)^{1/s}$, hence
	\begin{align}\label{sigma-Lamb-1}
	\sigma(\Lambda)\leq |\{x\in \mathbb{R} : |x|< (C\Lambda)^{1/s}\}|.
	\end{align}
	On the other hand, for $|x|<\left(\frac{1}{C2^{s/2}} \Lambda\right)^{1/s}$, we have for $\Lambda$ large,
	$$
	\langle x \rangle =\sqrt{1+|x|^2} <\sqrt{1+\left(\frac{1}{C2^{s/2}}\Lambda\right)^{2/s}}\leq \left(\frac{1}{C2^{s/2}}\Lambda\right)^{1/s} 2^{1/2}
	$$
	hence, by \eqref{prop-V},
	$$
	V(x)\leq C\langle x \rangle^s <C\left(\frac{1}{C}\Lambda\right)=\Lambda.
	$$
	This shows that 
	\begin{align} \label{sigma-Lamb-2}
	\left|\left\{x\in \mathbb{R} : |x|<\left(\frac{1}{C2^{s/2}}\Lambda\right)^{1/s} \right\}\right| \leq \sigma(\Lambda).
	\end{align}
	From \eqref{sigma-Lamb-1} and \eqref{sigma-Lamb-2}, we get (V1).
	
	$\bullet$ Checking (V2): We have
	$$
	V(x) \sim \langle x \rangle^s \leq \langle y \rangle^s \langle x-y \rangle^s \leq C\langle y\rangle^s \sim CV(y)
	$$
	for all $|x-y|\leq 1$.
	
	$\bullet$ Checking (V3): By Taylor's formula, we have
	$$
	|V(x+z)-V(x)| \leq |z|\int_0^1|V'(x+tz)|dt,
	$$
	which, by Assumption \ref{assu-V}, yields
	$$
	|V(x+z)-V(x)|\leq C|z| \int_0^1 \langle x+tz\rangle^{s-1}dt.
	$$
	As $s>1$, we infer that
	$$
	|V(x+z)-V(x)|\leq C |z|\langle x \rangle^{s-1} \langle z \rangle^{s-1}.
	$$
	Now denote 
	$$
	\Omega := \{x \in \mathbb{R} : |x-y|\leq 1, |x+z-y|\leq 1 \}.
	$$
	For $x\in \Omega$, we have $|x|\leq 1+|y|< 2\langle y\rangle$.	In particular,
	$$
	|V(x+z)-V(x)|\leq C|z| \langle y\rangle^{s-1}\langle z \rangle^{s-1}, \quad \forall x\in \Omega
	$$
	and 
	$$
	|\Omega|\leq |\{x\in \mathbb{R} : |x|< 2\langle y \rangle\}| \sim \langle y\rangle.
	$$
	Therefore, we obtain
	$$
	\int_{|x-y|\leq 1 \atop |x+z-y|\leq 1} |V(x+z) -V(x)|dx < C |z| \langle y \rangle^{s} \langle z \rangle^{s-1} \leq C |z| V(y)
	$$
	for all $y, z \in \mathbb{R}^d$ with $|z|<1$. This is (V3) with $\eta(t) =t$ and $\beta=0$.
	
	We have so far proved that \eqref{N-Lamb} holds for $V$ as in Assumption \ref{assu-V}. The claim of the lemma follows, as we now explain. For an upper bound, we use \eqref{sigma-Lamb-1} to get
	$$
	\int_{\mathbb{R}} (\Lambda- V(x))_+^{1/2} dx = \int_{V(x)\leq \Lambda} (\Lambda-V(x))^{1/2} dx \leq (2\Lambda)^{1/2} \sigma(\Lambda) \leq C\Lambda^{1/2+1/s}.
	$$
	For a lower bound, we estimate for a constant $\nu>0$ small to be chosen later,
	$$
	\int_{\mathbb{R}}(\Lambda-V(x))_+^{1/2}dx \geq \int_{\nu \Lambda<V(x)\leq \Lambda/2} (\Lambda-V(x))^{1/2}dx \geq \left(\frac{\Lambda}{2}\right)^{1/2} (\sigma(\Lambda/2)-\sigma(\nu \Lambda)).
	$$
	Thanks to \eqref{sigma-Lamb-1} and \eqref{sigma-Lamb-2}, we deduce
	\begin{align*}
		\int_{\mathbb{R}}(\Lambda-V(x))_+^{1/2}dx &\geq\left(\frac{\Lambda}{2}\right)^{1/2} (\sigma(\Lambda/2)-\sigma(\nu \Lambda))\\
		&\geq \left(\frac{\Lambda}{2}\right)^{1/2}\left(\left|B\left(0,\left(\frac{1}{C2^{s/2}\frac{\Lambda}{2}}\right)^{1/s}\right)\right|-|B(0,C\nu \Lambda)^{1/s}|\right) \\
		&\geq C\Lambda^{1/2+1/s}
	\end{align*}
	provided that $\nu>0$ is taken sufficiently small, where $B(0,R)$ is the segment $[-R,R]$. 
\end{proof}
%
%
%

\newpage


\end{document}